\documentclass[a4paper,reqno]{amsart}
\usepackage{longtable} 
\usepackage{stmaryrd,amssymb,mathtools,booktabs,upref}
\usepackage[inline,shortlabels]{enumitem}
\usepackage[hidelinks]{hyperref}
\usepackage{color}
\usepackage{float}
\usepackage{todonotes}
\allowdisplaybreaks

\newtheorem{theorem}{Theorem}[section]
\newtheorem{proposition}[theorem]{Proposition}
\newtheorem{lemma}[theorem]{Lemma}
\newtheorem{corollary}[theorem]{Corollary}

\newtheorem*{notation}{Notation}

\theoremstyle{definition}
\newtheorem{definition}[theorem]{Definition}
\newtheorem{example}[theorem]{Example}
\theoremstyle{remark}
\newtheorem{remark}[theorem]{Remark}

\numberwithin{equation}{section}

\newcommand{\hook}{\mathbin{\lrcorner}}

\newcommand{\cL}{\mathcal{L}}

\newcommand{\cG}{\mathcal{G}}
\newcommand{\cR}{\mathcal{R}}
\newcommand{\cRc}{\mathcal{RC}}

\newcommand{\bC}{\mathbb{C}}
\newcommand{\bR}{\mathbb{R}}

\newcommand{\lie}[1]{\mathfrak{#1}}
\newcommand{\mfa}{\lie{a}}

\newcommand{\mfg}{\lie{g}}
\newcommand{\g}{\lie{g}}
\newcommand{\mfh}{\lie{h}}
\newcommand{\h}{\lie{h}}
\newcommand{\mfn}{\lie{n}}

\newcommand{\mfr}{\lie{r}}
\newcommand{\aff}{\lie{aff}}

\newcommand{\yes}{\checkmark}

\newcommand{\spa}[1]{\mathrm{span}(#1)}

\DeclareMathOperator{\im}{Im}

\DeclareMathOperator{\End}{End}

\DeclareMathOperator{\ad}{ad}
\DeclareMathOperator{\diag}{diag}
\DeclareMathOperator{\id}{id}

\DeclareMathOperator{\tr}{tr}

\DeclareMathOperator{\op}{\oplus}

\DeclareMathOperator{\dd}{d}

\setlist{nosep}

\begin{document}

\title{Generalised Einstein metrics on Lie groups}

\author{Vicente Cort\'{e}s}

\author{Marco Freibert}

\author{Mateo Galdeano}
\thanks{}

\begin{abstract}
We continue the systematic study of left-invariant generalised Einstein metrics on Lie groups initiated in \cite{CK}. 
Our approach is based on a new reformulation of the corresponding algebraic system. For a fixed 
Lie algebra $\mathfrak g$, the unknowns of the system consist of a scalar product $g$ and a $3$-form $H$ on $\mathfrak g$ as well as  
a linear form $\delta$  on  
$\mathfrak g\oplus \mathfrak g^*$.  As in \cite{CK}, the Lie bracket of $\mathfrak g$ is considered part of the unknowns.
In the Riemannian case, we show that the generalised Einstein condition always reduces to the commutator ideal and we provide a full classification of solvable generalised Einstein Lie groups. In the Lorentzian case, under the additional assumption $\delta =0$, we classify---up to one case---all almost Abelian generalised Einstein Lie groups. We then particularize to four dimensions and provide a full classification of generalised Einstein Riemannian Lie groups as well as generalised Einstein Lorentzian Lie groups with $\delta =0$ and non-degenerate commutator ideal.
\end{abstract}
\maketitle

\section{Introduction}

Generalised geometry was introduced by Hitchin \cite{Hi} and further developed by Gualtieri \cite{Gu1,Gu2} as a tool to unify complex and symplectic geometry. Over time, generalised geometry has developed into a full branch of geometry and its range of applications has grown far beyond its original purpose. Particularly interesting is its relation to physics, as generalised geometry provides a way to geometrize certain supergravity theories \cite{CSW,GSh}.

Perhaps the most fundamental version of generalised geometry is the one  formulated on a (say exact) Courant algebroid over a manifold $M$. There is a convenient way to describe such algebroids: every exact Courant algebroid over $M$ is isomorphic to the bundle $TM\oplus T^*M$ equipped with the natural pairing, the natural projection to $TM$, and an $H$-twisted Dorfman bracket. This bracket extends the usual Lie bracket and crucially depends on a closed 3-form $H$.

Many familiar concepts from classical geometry have an analogue in generalised geometry, a main example being the generalised metric $\mathcal{G}$. An exact Courant algebroid over $M$ equipped with a generalised metric $\mathcal{G}$ is (isometrically) isomorphic to a bundle as described above together with a metric $g$ on $M$ which in turn determines the generalised metric $\mathcal{G}_g$ on the bundle $TM\oplus T^*M$.

The main object of interest for us will be the generalised Ricci tensor. A key observation is that, unlike in classical geometry, this tensor depends not only on the generalised metric $\mathcal{G}_g$ but also on the three-form $H$ and a divergence operator $\delta$. Therefore, we say that $(M,H,\mathcal{G}_g,\delta)$ is \emph{generalised Einstein} if the associated generalised Ricci tensor vanishes.

The motivation to study generalised Einstein manifolds comes both from geometry and physics. First of all, these manifolds can be understood as the generalised analogue of Ricci flat manifolds, which constitute one of the most important classes of manifolds in classical geometry. A particular subclass is that of Bismut Ricci-flat manifolds, compare Remark \ref{Bismut:rem}.

On the other hand, the generalised Einstein equations naturally emerge as some of the equations of motion of certain supergravity theories \cite{CSW,GSh}, in the same spirit as the classical Einstein equations appear as the equations of motion of the Einstein--Hilbert action. This further motivates the analysis of not just Riemannian but also Lorentzian metrics. In this context, the four-dimensional case is clearly of particular interest.

Moreover, it has been shown \cite{GLP} that, for the NS-NS sector of ten-dimensional supergravity, the most general supersymmetric backgrounds are given (locally) by a principal bundle with fibre a Lorentzian Lie group and base a suitable Riemannian manifold. These backgrounds satisfy the corresponding supergravity equations of motion, which in some particular cases reduce to the generalised Einstein condition on the fibre. Hence, the generalised Einstein equations on Lie groups are especially interesting since they are related to compactifications relevant for string theory.

In this work we thus focus on left-invariant Riemannian and Lorentzian generalised Einstein structures on exact Courant algebroids over a Lie group $G$. In particular, we will concentrate on classifying these Lie groups in four dimensions as well as obtaining some classification results in arbitrary dimensions.

This topic was first studied on \cite{CK, K} (see also \cite{ADG} for an alternative approach using the language of quadratic Lie algebras), where the general theory of left-invariant generalised Einstein structures on Lie groups was developed and a full classification for the three-dimensional case was obtained. The classification was extended to the simplest class of heterotic Courant algebroids in 
\cite{CD}. Examples and partial classifications of Bismut Ricci-flat homogeneous spaces have been obtained in \cite{PR1,PR2}. The key aspect is that in this setting the non-linear partial differential equations corresponding to the generalised Einstein condition can be recast as algebraic equations on the underlying Lie algebra $\mfg$, making the problem amenable to Lie theory, algebra and even classifications.

Following this philosophy, we present in Proposition \ref{pro:gEexplicitly} a new rewriting of the generalised Einstein equations of \cite{CK} explicitly in terms of tensors encoding the geometric data and the Lie bracket on the underlying Lie algebra 
$\mfg$.  Hence, we call solutions to these equations \emph{generalised Einstein Lie algebras}. Our new formulas prove to be extremely useful for classifying possible solutions to the generalised Einstein equations and they considerably reduce some of the required computations.

In particular, we are able to show in Theorem \ref{th:structRiemGE} that in the Riemannian setting the generalised Einstein condition automatically reduces to the generalised Einstein equations on the commutator ideal supplemented by an extra equation on the restriction of the divergence operator. We can then use this result to provide a full description of solvable generalised Einstein Riemannian Lie algebras, which reduce to the classical flat Abelian Lie algebras classified by Milnor \cite{Mi}, see Corollary~\ref{co:Riemsolv}. We also present in Theorem \ref{th:4dRiem} a full classification in four dimensions of all generalised Einstein Riemannian Lie algebras.

We then switch our attention to the Lorentzian setting, which turns out to be much less rigid than its Riemannian counterpart. In order to achieve some partial classification, we will focus on the case where the divergence operator vanishes, $\delta=0$. We are then able to obtain in Theorem \ref{th:almostAbelian} a classification---up to one particular subcase\footnote{The subcase corresponds to the situation where the codimension one Abelian ideal $\mathfrak n$ is non-degenerate and the symmetric part of the operator $f$ of $\mathfrak n$ defining the almost Abelian Lie algebra $\mathfrak g$ has a non-real eigenvalue.}, which is excluded in four dimensions---of all generalised Einstein Lorentzian almost Abelian Lie algebras with $\delta=0$ in arbitrary dimensions. 

Furthermore, we show in Theorem \ref{th:4dLorentzian} a full classification of generalised Einstein Lorentzian Lie algebras in four dimensions with vanishing divergence operator, under the additional assumption that the commutator ideal is non-degenerate. All our results in the four dimensional case are summarised in Table \ref{table:4d}.

It is worth pointing out that our classification results can also be used to obtain generalised Einstein Lorentzian Lie algebras with non-zero divergence operator following Corollary \ref{co:gEdeltanonzero}. We present in Corollary \ref{co:almostAbeliandeltaneq0} the results for Lorentzian almost Abelian Lie algebras, and some additional results for the four-dimensional case in Corollary \ref{co:reductivecasesdeltaneq0}.

Similarly, we obtain some partial results for metrics of signature $(n-2,2)$, where $n$ is the dimension of the Lie group. We discuss the almost Abelian case in Corollary \ref{co:signaturen-22}. The case of four dimensions, which corresponds to split signature, is covered in Corollary \ref{co:signature22}. This signature is even richer than the Lorentzian one, as illustrated by Example \ref{ex:abeliansignature32} where we construct an Abelian generalised Einstein Lie algebra with  $n=5$ for which $H\neq0$.

Our work leads to several natural follow-up questions. Since the main focus of this paper is to obtain general classification results rather than constructing particular explicit solutions, certain cases remain largely unexplored. It would be interesting to further study them, and in fact we expect many other generalised Einstein Lie algebras to exist. For example, the subcase that is not covered by our Lorentzian almost Abelian classification provides new solutions already in dimension five, cf.\ Example \ref{ex:5dalmostAbelian}, although a general classification result seems to be out of reach at the moment. In the same vein, four-dimensional Lorentzian solutions with non-zero divergence are expected to be abundant---this was indeed the case in three dimensions \cite{CK}---and many more examples should exist beyond those of Corollary \ref{co:reductivecasesdeltaneq0}. 

In the Riemannian case, we found that the generalised Einstein condition automatically reduces to the commutator ideal. Similarly, the non-degeneracy of the commutator ideal was one of the key assumptions in our classification of four-dimensional Lorentzian Lie algebras. One can thus wonder what the role of the commutator ideal is in more generality, and whether some variant of the Riemannian reduction result could also hold in other situations.

Finally, supergravity theories have additional (bosonic) equations of motion beyond the generalised Einstein equations. In the language of generalised geometry, these can be expressed as the vanishing of the \emph{generalised scalar curvature}. It would be interesting to study which of our solutions satisfy this property and are honest supergravity backgrounds as a result. We are planning to study this in the future. 

The paper is organised as follows: in Section \ref{sec:generalisedEinstein} we set notation and present our convenient rewriting of the generalised Einstein condition on Lie algebras. Section \ref{sec:Riemannian} is devoted to the Riemannian case and includes our results both in four and in higher dimensions. We then move on to the Lorentzian setting with zero divergence operator: in Section \ref{sec:almostAbelian} we explore the almost Abelian case in arbitrary dimensions, whereas in Section \ref{sec:4dLorentzian} we restrict ourselves to the case of four dimensions and non-degenerate commutator ideal.

\subsection*{Acknowledgements}

We would like to thank David Krusche for his contributions during the early stages of this project.

Research of MG and VC is funded by the Deutsche Forschungsgemeinschaft (DFG, German Research Foundation) under Germany's Excellence Strategy, EXC 2121 ``Quantum Universe,''  390833306. 
Research of VC is also funded by the DFG under -- SFB-Gesch\"aftszeichen 1624 -- Projektnummer 506632645.

\section{The generalised Einstein condition on Lie groups and Lie algebras}
\label{sec:generalisedEinstein}
\subsection{Exact Courant algebroids and generalised pseudo-Riemannian metrics}
In this subsection, we briefly recall some basic notions from generalised pseudo-Riemannian geometry. For a more detailed introduction to this subject as well as for proofs of the mentioned facts, we refer to \cite{GSt}, \cite{K} and \cite{GSh}, where the latter two treat also the genuine \emph{pseudo}-Riemannian case.

We begin with the basic notion of an \emph{(exact)} Courant algebroid:
\begin{definition}
Let $M$ be a manifold:
\begin{itemize}[wide]
\item
 A \emph{Courant algebroid (on $M$)} is a quadruple $(E,\langle \cdot,\cdot\rangle, [\cdot,\cdot],\pi)$ consisting of a vector bundle $E$ over $M$, a non-degenerate symmetric bilinear form $\langle \cdot,\cdot \rangle$ on $E$, a bilinear map $[\cdot,\cdot]:E\times E\rightarrow E$ on $E$ (called the \emph{bracket}) and a bundle map $\pi:E\rightarrow TM$ such that for all sections $a,b,c$ of $E$ and all $f\in C^{\infty}(M)$ the following conditions hold:
\begin{itemize}
\item[(i)]
$[a,[b,c]]=[[a,b],c]+[b,[a,c]]\,$.
\item[(ii)]
$\pi([a,b])=[\pi(a),\pi(b)]\,$.
\item[(iii)]
$[a,fb]=f[a,b]+\pi(a)(f)\, b\,$.
\item[(iv)]
$\pi(a)\langle b,c\rangle=\langle [a,b],c\rangle+\langle b,[a,c]\rangle\,$.
\item[(v)]
$[a,b]+[b,a]=\pi^*(\dd\langle a,b\rangle)\,$.
\end{itemize}
Here, in condition (v), we use $\langle \cdot,\cdot\rangle$ to identify $E^*$ with $E$, and so the dual map $\pi^*$ of $\pi:E\rightarrow TM$ is a bundle map $\pi^*:T^*M\rightarrow E$ from $T^*M$ to $E$. The conditions (ii) and (iii) easily follow from the others and are only 
included for convenience.

In the following, we often denote the Courant algebroid $(E,\langle \cdot,\cdot\rangle, [\cdot,\cdot],\pi)$ simply by $E$.
\item
Let $(E,\langle \cdot,\cdot\rangle, [\cdot,\cdot],\pi)$ be a Courant algebroid. Then we always have the following short sequence of vector bundles:
\begin{equation*}
0\rightarrow T^*M \stackrel{\pi^*}{\rightarrow} E\stackrel{\pi}{\rightarrow} TM\rightarrow 0 \, .
\end{equation*}
We call $E$ \emph{exact} if this short sequence is exact.
\item
A Courant algebroid $(E_1,\langle\cdot,\cdot\rangle_{E_1},[\cdot,\cdot]_{E_1},\pi_{E_1})$ over $M$ and a Courant algebroid $(E_2,\langle\cdot,\cdot\rangle_{E_2},[\cdot,\cdot]_{E_2},\pi_{E_2})$ over $N$ are called \emph{isomorphic} if there is a pair $(f,F)$ of maps, where $f:M\rightarrow N$ is a diffeomorphism and $F:E_1\rightarrow E_2$ is a vector bundle isomorphism covering $f$ such that the pullback via $F$ of the Courant algebroid data $(\langle\cdot,\cdot\rangle_{E_2},[\cdot,\cdot]_{E_2},\pi_{E_2})$ on $E_2$ equals the Courant algebroid data $(\langle\cdot,\cdot\rangle_{E_1},[\cdot,\cdot]_{E_1},\pi_{E_1})$ on $E_1$
\end{itemize}
\end{definition}
\begin{remark}
\begin{itemize}[wide]
	\item 
We note that condition (i) shows that the bracket $[\cdot,\cdot]$ satisfies the Jacobi identity. However, it is (if $\dim(M)>0$) not a Lie bracket as it fails to be anti-symmetric, cf. condition (v).
\item The subbundle $\pi^*(T^*M)$ is an isotropic subbundle of $E$.
\end{itemize}	
\end{remark}
We emphasize that we are only interested in (left-invariant) exact Courant algebroids (on Lie groups) in this article and so will concentrate only on this class of Courant algebroids in the following. In fact, all exact Courant algebroids are isomorphic to the following examples of exact Courant algebroids on $M$:
\begin{definition}
Let $M$ be a manifold and $\mathbb{T}M:=TM\oplus T^*M$ be the \emph{generalised tangent bundle} over $M$. The generalised tangent bundle comes equipped with a natural non-degenerate symmetric bilinear form $\langle \cdot,\cdot\rangle$ given by
\begin{equation*}
\langle X+\xi,Y+\eta\rangle =\frac{1}{2} \left(\eta(X)+\xi(Y)\right)\, ,
\end{equation*}
for $X,Y\in \Gamma(TM)$, $\xi,\eta\in \Gamma(T^*M)$.

Moreover, given any closed three-form $H\in \Omega^3 M$, we have a natural bilinear map $[\cdot,\cdot]_H:\mathbb{T}M\times \mathbb{T}M\rightarrow \mathbb{T}M$ given by
\begin{equation*}
[X+\xi,Y+\eta]_H:=[X,Y]+\cL_X \eta-\iota_Y \dd\xi+H(X,Y,\cdot)\,.
\end{equation*}
$[\cdot,\cdot]_H$ is called the \emph{$H$-twisted Dorfman bracket}.
\end{definition}
The following proposition is the content of \cite[Proposition 2.10, Proposition 2.17]{GSt}:
\begin{proposition}
Let $M$ be a manifold.
\begin{enumerate}
	\item[(a)]
	Let $H\in \Omega^3 M$ be closed. Then the generalised tangent bundle $\mathbb{T}M$ on $M$ endowed with the natural symmetric bilinear form $\langle \cdot,\cdot \rangle$, the $H$-twisted Dorfman bracket $[\cdot,\cdot]_H$ and the natural bundle map $\pi:\mathbb{TM}=TM\oplus T^*M\rightarrow TM$ given by the projection to $TM$ along $T^*M$ is an exact Courant algebroid.
	 \item[(b)] Any exact Courant algebroid $E$ on $M$ is isomorphic to $(\mathbb{T}M,\langle\cdot,\cdot \rangle,[\cdot,\cdot]_H,\pi)$ for some closed $H\in \Omega^3 M$.
\end{enumerate}
\end{proposition}
Next, we add more structure to a given exact Courant algebroid $E$:
\begin{definition}
Let $E$ be an exact Courant algebroid over an $n$-dimensional manifold $M$.
\begin{itemize}[wide]
	\item
A \emph{generalised pseudo-Riemannian metric} is an endomorphism $\cG$ of $E$ which is an involution, i.e. $\cG^2=\id_E\,$, and for which
\begin{equation*}
	\tilde{\cG}:=\langle \cG\cdot,\cdot\rangle
\end{equation*}
is a symmetric bilinear form on $E^*$ such that $\tilde{\cG}|_{S^2 \pi^*(T^*M)}$ is non-degenerate.

We say that $\cG$ has \emph{signature} $(p,n-p)$ if $\tilde{\cG}|_{S^2 \pi^*(T^*M)}$ has signature $(p,n-p)$ and we call $\cG$ \emph{Riemannian} if $\tilde{\cG}|_{S^2 \pi^*(T^*M)}$ is positive definite (i.e. has signature $(n,0)$) and \emph{Lorentzian} if $\tilde{\cG}|_{S^2 \pi^*(T^*M)}$ is \emph{Lorentzian}, i.e. has signature $(n-1,1)$.
\item
If $\cG$ is a generalised pseudo-Riemannian metric on $M$, then we denote by $E_{\pm}$ the $\pm 1$-eigenbundles and observe that these have dimension $n$ as well. Moreover, we denote by
\begin{equation*}
\pi_{\pm}:E\rightarrow E_{\pm}
\end{equation*}
the projections onto $E_{\pm}$ along $E_{\mp}\,$. We observe that these are explictly given by
\begin{equation*}
\pi_{\pm} (e)=\frac{1}{2}(e\pm \cG e)\,.
\end{equation*}
\item 
Two exact Courant algebroids $E_1$ and $E_2$ endowed with generalised pseudo-Riemannian metrics $\cG_{E_1}$ and $\cG_{E_2}$ are said to be \emph{isometrically isomorphic} if there is an isomorphism $(f,F)$ of Courant algebroids from $E_1$ to $E_2$ such that $\cG_{E_2}\circ F=F\circ \cG_{E_1}$.
\end{itemize}
\end{definition}
\begin{remark}
\begin{itemize}[wide]
	\item
	Let $\cG$ be a generalised pseudo-Riemannian metric of signature $(p,n-p)$ on an exact Courant algebroid $E$. Then $\pi_{\pm}|_{\pi^*(T^*M)}:\pi^*(T^*M)\rightarrow E_{\pm}$ is a vector bundle isomorphism satisfying
	\begin{equation*}
	\tilde{\cG}(\xi,\eta)=\pm 2\,\langle \pi_{\pm}(\xi),\pi_{\pm}(\eta)\rangle \, .
	\end{equation*}
In particular, the restriction of $\langle \cdot,\cdot \rangle$ to $E_+$ is non-degenerate and also has signature $(p,n-p)$. It is then fairly easy to see that the definition of an \emph{admissible} generalised pseudo-Riemannian metric of signature $(p,n-p)$ given in \cite{GSh} coincides with our definition of a generalised pseudo-Riemannian metric of signature $(p,n-p)$.
\item 
As $\langle E_+,E_-\rangle=0$, the last item shows that $\tilde{\cG}$ is also non-degenerate on the entire bundle $E$ and has signature $(2p,2(n-p))$.
\end{itemize}
\end{remark}
There is a special class of generalised pseudo-Riemannian metrics on the exact Courant algebroid $\mathbb{T}M$ to which any exact Courant algebroid over $M$ endowed with a generalised pseudo-Riemannian metric is isometrically isomorphic to:
\begin{proposition}\label{pro:eCAisometricisom}
Let $M$ be an $n$-dimensional manifold.
\begin{enumerate}[(a)]
	\item Let $g$ be a pseudo-Riemannian metric of signature $(p,n-p)$ on $M$ and set
	\begin{equation*}
	\cG_g:=\begin{pmatrix} 0 & g^{-1} \\ g & 0\end{pmatrix}
	\end{equation*}
on $\mathbb{T}M$ with respect to the splitting $\mathbb{T} M=TM\oplus T^*M$, where $g$ is considered as a bundle isomorphism from $TM$ to $T^*M$. Moreover, let $H\in \Omega^3 M$ be closed. Then
$\cG_g$ is a generalised pseudo-Riemannian metric of signature $(p,n-p)$ on $(\mathbb{T}M,\langle \cdot,\cdot\rangle,[\cdot,\cdot]_H,\pi)$.
\item Let $E$ be an exact Courant algebroid over $M$ endowed with a generalised pseudo-Riemannian metric $\cG$ of signature $(p,n-p)$. There exists a pseudo-Riemannian metric $g$ of signature $(p,n-p)$ and a closed three-form $H$ on $M$ such that $(E,\cG)$ is isometrically isomorphic to $(\mathbb{T}M,\langle \cdot,\cdot\rangle,[\cdot,\cdot]_H,\pi,\cG_g)$.
\end{enumerate}
\end{proposition}
\subsection{Generalised Ricci curvatures and the generalised Einstein condition}
We start by giving the definition of a generalised connection, which will be essential to define the generalised Ricci curvatures below and then finally the generalised Einstein condition:
\begin{definition}
Let $E$ be an exact Courant algebroid over $M$.
\begin{itemize}[wide]
\item A \emph{generalised} connection $D$ on $E$ is an $\bR$-linear map $D:\Gamma(E)\rightarrow \Gamma(\End(E))$, $e\mapsto De=(\tilde{e}\mapsto D_{\tilde{e}} e)$ such that
\begin{itemize}
	\item[(1)] $D_{\tilde{e}}(fe)=\pi(\tilde{e})(f)\, e+ f\, D_{\tilde{e}} e$ for all $e,\tilde{e}\in \Gamma(E)$, $f\in C^{\infty}(M)$
	\item[(2)] and $D_{\tilde{e}} \langle e_1,e_2\rangle=\langle D_{\tilde{e}} e_1,e_2\rangle+\langle e_1,D_{\tilde{e}} e_2\rangle$ for all $\tilde{e},e_1,e_2\in \Gamma(E)$.
\end{itemize}
\item The \emph{torsion} of $D$ is the $(1,2)$-tensor field $T_D$ defined by
\begin{equation*}
T_D(e_1,e_2):=D_{e_1} e_2-D_{e_2} e_1-[e_1,e_2]+(D e_1)^* e_2 \, ,
\end{equation*}
for $e_1,e_2\in \Gamma(E)$, where $(D e_1)^*\in \End(E)$ is the adjoint of $De_1\in \End(E)$ with respect to $\langle\cdot,\cdot\rangle$.
\item Let $\cG$ be a generalised pseudo-Riemanian metric on $E$. Then a generalised connection $D$ is called a \emph{generalised Levi-Civita connection for $\cG$} if $D$ is \emph{torsion-free}, i.e. $T_D=0$, and \emph{metric}, i.e. $D\cG=0$.
\item 
For a generalised connection $D$ on $E$ compatible with a generalised pseudo-Riemannian metric $\cG$ on $E$, we may define two \emph{generalised curvature} endomorphisms
\begin{equation*}
\cR^+_D\in \Gamma(E_+^*\otimes E_-^*\otimes \End(E_+)) \, ,\quad
\cR^+_D(a,b) c=D_a D_b c-D_b D_a c-D_{[a,b]} c \, ,
\end{equation*}
for $a,c\in \Gamma(E_+)$, $b\in \Gamma (E_-)$ and
\begin{equation*}
	\cR^-_D\in \Gamma(E_-^*\otimes E_+^*\otimes \End(E_-)) \, ,\quad
	\cR^-_D(a,b) c=D_a D_b c-D_b D_a c-D_{[a,b]} c \, ,
\end{equation*}
for $a,c\in \Gamma(E_-)$, $b\in \Gamma (E_+)$. Consequently, we define two \emph{generalised Ricci curvatures} by
\begin{equation*}
	\cRc^{+}_D\in \Gamma(E_-\otimes E_+) \, ,\quad \cRc^+_D(a,c)=\tr(E_+\ni a\mapsto \cR^D(a,b)c\in E_+) \, ,
\end{equation*}
for $b\in E_-\,$, $c\in E_+$ and
\begin{equation*}
	\cRc^{-}_D\in \Gamma(E_+\otimes E_-) \, ,\quad \cRc^-_D(a,c)=\tr(E_-\ni a\mapsto \cR^D(a,b)c\in E_-) \, ,
\end{equation*}
for $b\in E_+\,$, $c\in E_-\,$. Finally, we set
\begin{equation*}
	\cRc_D:=\cRc^+_D-\cRc^-_D\in  \Gamma(E_-\otimes E_+ \oplus  E_+\otimes E_-) \, ,
\end{equation*}
and call $\cRc_D$ also the \emph{generalised Ricci curvature of $D$}.
\end{itemize}
\end{definition}
\begin{remark}
We note that trying to define analogous generalised curvature operators for other combinations of sections of $E_+$ and $E_-$ would not lead to a $C^{\infty}(M)$-trilinear map.
\end{remark}
Now, by \cite[Proposition 3.15]{GSt}, a generalised Levi-Civita connection exists for any generalised pseudo-Riemannian metric but is never unique. So, given a generalised pseudo-Riemannian metric $\cG$, there is also no natural generalised Ricci curvature associated to $\cG$. However, the generalised Ricci curvature $\cRc_D$ of a generalised Levi-Civita connection $D$ only depends on the \emph{divergence operator} of $D$:
\begin{definition}
Let $E$ be an exact Courant algebroid over a manifold $M$. Then a \emph{divergence operator} is a first-order differential operator $\delta:\Gamma(E)\rightarrow C^{\infty}(M)$ satisfying $\delta(f e)=\pi(e)(f)+f\,\delta(e)$ for all $e\in \Gamma(E)$, $f\in C^ {\infty}(M)$.
\end{definition}
\begin{lemma}\label{le:divandRic}
Let $E$ be an exact Courant algebroid over a manifold $M$ endowed with a pseudo-Riemannian metric $\cG$. Moreover, let $D$ be a generalised Levi-Civita connection of $\cG$. Then
\begin{equation*}
\delta_D:\Gamma(E)\rightarrow C^{\infty}(M) \, ,\quad \delta_D(e):=\tr(De) \, ,
\end{equation*}
is a divergence operator on $E$ and the map
\begin{equation*}
D\mapsto \delta_D
\end{equation*}
from all generalised Levi-Civita connections of $\cG$ to the space of all divergence operators on $E$ is surjective. Furthermore, if $\delta_{D_1}=\delta_{D_2}$ for two generalised Levi-Civita connections then $\cRc_{D_1}=\cRc_{D_2}$ for the associated generalised Ricci curvatures.
\end{lemma}
Lemma \ref{le:divandRic} allows us now to define:
\begin{definition}
Let $E$ be an exact Courant algebroid over a manifold $M$ endowed with a generalised pseudo-Riemannian metric $\cG$ and let $\delta$ be a divergence operator on $E$.
\begin{itemize}
	\item 
	 The \emph{generalised Ricci curvature $\cRc_{(\cG,\delta)}$ of $(\cG,\delta)$} is the generalised Ricci curvature of any generalised Levi-Civita connection $D$ of $\cG$ with $\delta_D=\delta$.
	 \item The pair $(\cG,\delta)$ is called \emph{generalised Einstein} if $\cRc_{(\cG,\delta)}=0$.
\end{itemize}
\end{definition}
\subsection{Generalised pseudo-Riemannian geometry on Lie algebras}
From now on, using Proposition \ref{pro:eCAisometricisom}, we assume that our exact Courant algebroid $E$ endowed with a generalised pseudo-Riemannian metric $\cG$ is given by $(\mathbb{T}M,\langle \cdot,\cdot\rangle,[\cdot,\cdot]_H,\pi,\cG_g)$ for a closed three-form $H$ on $M$ and a pseudo-Riemannian metric $g$ on $M$.

Moreover, we now look at the special case where $M$ is a Lie group and assume that all involved data is left-invariant:
\begin{definition}
Let $G$ be a Lie group with associated Lie algebra $\mfg$.
\begin{itemize}[wide]
	\item 
	 An exact Courant algebroid on $G$ of the form  $(\mathbb{T}G,\langle \cdot,\cdot\rangle,[\cdot,\cdot]_H,\pi)$ with $H\in\Omega^3 G$ endowed with a generalised pseudo-Riemannian metric of the form $\cG_g$ for $g$ a pseudo-Riemannian metric on $G$ is called \emph{left-invariant} if both $H$ and $g$ are left-invariant. We usually denote the pair $\Bigl((\mathbb{T}G,\langle \cdot,\cdot\rangle,[\cdot,\cdot]_H,\pi),\cG_g\Bigr)$ of the exact Courant algebroid and the generalised pseudo-Riemannian metric as above simply by $(G,H,\cG_g)$ and say, by a slight abuse of notation, that $(H,\cG_g)$ is a \emph{left-invariant generalised pseudo-Riemannian metric} on $G$.
     \item We also say that a pair $(H,\cG_g)\in \Lambda^3 \mfg^*\times \End(E)$ with $H$ closed and $g$ a pseudo-Riemannian metric on $\mfg$ is a \emph{generalised pseudo-Riemannian metric on (the Lie algebra) $\mfg$}, where we set
     \begin{equation*}
     	E:=E(\mfg):=\mfg\oplus \mfg^* \, .
     \end{equation*}
 We then also call $(\mfg,H,\cG_g)$ a \emph{generalised pseudo-Riemannian Lie algebra}.
 
      Note that generalised pseudo-Riemannian metrics on $\mfg$ are $1:1$ to left-invariant generalised pseudo-Riemannian metrics on $G$.
     \item A divergence operator $\delta$ on $\mathbb{T}G$ is called \emph{left-invariant} if $\delta(e)$ is left-invariant for all left-invariant sections $e$ of $\mathbb{T}G$. Such a divergence operator is equivalent to an element in $E^*=(\mfg\oplus \mfg^*)^*$ and so any element $\delta\in E^*$ is called a \emph{divergence operator on $\mfg$}.
     \item It thus makes sense to talk about the \emph{generalised Ricci curvature} of a pair of a generalised pseudo-Riemannian metric $(H,\cG_g)$ on $\mfg$ and a divergence operator $\delta$ on $\mfg$ and also to say that $(H,\cG_g,\delta)$ is \emph{generalised Einstein}. We then also call $(\mfg,H,\cG_g,\delta)$ a \emph{generalised Einstein pseudo-Riemannian Lie algebra}.
      \end{itemize}
\end{definition}
\begin{remark}
	Note that on a generalised pseudo-Riemannian Lie algebra $(\mfg,H,\cG_g)$, the $H$-twisted Dorfman bracket $[\cdot,\cdot]_H$ on $E=\mfg\oplus \mfg^*$ is given by
	\begin{equation}\label{eq:DfBr}
		\begin{split}
			[X+\xi,Y+\eta]_H&=[X,Y]_{\mfg}+\cL_X \eta-Y\hook \dd\xi+ H(X,Y,\cdot)\\
			&=[X,Y]+X\hook d\eta-Y\hook \dd\xi+H(X,Y,\cdot) \, ,
		\end{split}
	\end{equation}
	where $[\cdot,\cdot]_{\mfg}$ is the Lie bracket of $\mfg$ and we usually omit the index $\mfg$ in the following. So here $[\cdot,\cdot]_H$ is actually anti-symmetric and, thus, defines a Lie algebra structure on $E$. Note that condition (iv) of a Courant algebroid shows that $(E, {[} \cdot , \cdot {]}_H, \langle \cdot , \cdot \rangle)$ is actually a quadratic Lie algebra.
	
	Moreover, we observe that $\mfg^*$ is an Abelian ideal of 
	$E$ such that $E/\mfg^*\cong \mfg$
	as Lie algebras. So $E$ is an Abelian extension of $\mfg$ by the Abelian Lie algebra $\mfg^*$, i.e. we have the following short exact sequence of Lie algebras:
	\begin{equation*}
		0\rightarrow \mfg^*\rightarrow E\rightarrow \mfg\rightarrow 0 \, .
	\end{equation*}
	In the case $H=0$, this sequence splits and $E$ is a semidirect product $E=\mfg\ltimes \mfg^*$ with $\mfg$ acting on the Abelian subalgebra $\mfg^*$ by the coadjoint action. Consequently, for $H=0$, $E$ is the cotangent Lie algebra of $\mfg$.
\end{remark}
\subsection{New useful formulas for the generalised Einstein condition}
Next, we recall some explicit formulas from \cite{CK} for the generalised Ricci curvatures---and so for the generalised Einstein condition---for a generalised pseudo-Riemannian Lie algebra $(\mfg,H,\cG_g)$ endowed with a divergence operator $\delta$. We will then rewrite these formulas so that the generalised Einstein condition is explicitly expressed in terms of the adjoint operators of the underlying Lie algebra (and associated endomorphisms) as well as of the three-form $H$ (and associated two-forms and endomorphisms), cf. Proposition \ref{pro:gEexplicitly} below. This proposition will allow us to derive in an easy way a reduction result in the Riemannian case, cf. Theorem \ref{th:structRiemGE}, and will also be very helpful in the Lorentzian case which will be considered in Sections \ref{sec:almostAbelian} and \ref{sec:4dLorentzian}.

We begin presenting the already known formulas from \cite{CK}. Let $u\in E_+$ and $v\in E_-$ be given and define two linear maps by
\begin{equation*}
\begin{split}
\Gamma_{u}:& E_-\rightarrow E_+ \, ,\qquad \Gamma_u:=\pi_+\circ [u,\cdot]_H|_{E_-} \, ,\\
\Gamma_{v}:& E_+\rightarrow E_- \, ,\qquad \Gamma_v:=\pi_-\circ [v,\cdot]_H|_{E_+} \, .
\end{split}
\end{equation*}
Using these linear maps, one has:
\begin{proposition}[\cite{CK}]
Let $(\mfg,H,\cG_g)$ be a generalised pseudo-Riemannian Lie algebra and let $\delta\in E^*$. Then the associated generalised Ricci curvatures $\cRc^+=\cRc^+_{(\cG_g,\delta)}$ and  $\cRc^-=\cRc^-_{(\cG_g,\delta)}$ are given by 
\begin{equation}\label{eq:gRcs}
\begin{split}
\cRc^+(v,u)&=-\tr(\Gamma_v\circ \Gamma_u)+\delta(\pi_+([v,u]_H)) \, ,\\
\cRc^-(u,v)&=-\tr(\Gamma_v\circ \Gamma_u)+\delta(\pi_-([u,v]_H))=-\tr(\Gamma_v\circ \Gamma_u)-\delta(\pi_-([v,u]_H)) \, ,\\
\end{split}
\end{equation}
for $u\in E_+$ and $v\in E_-\,$. Moreover, $(\mfg,H,\cG_g,\delta)$ is generalised Einstein if and only if 
\begin{equation}\label{eq:gE}
	2 \tr(\Gamma_v\circ \Gamma_u)=\delta(\cG_g\, [v,u]_H) \, ,\qquad \delta([u,v]_H)=0 \, ,
\end{equation}
for all $u\in E_+$ and $v\in E_-\,$.
\end{proposition}
\begin{proof}
The first part of the proposition, i.e. the formulas for $\cRc^+$ and $\cRc^-$ are given in Theorem 1.1. in \cite{CK}. For the second part, we simply add and substract the two equations in \eqref{eq:gRcs} and set, using that $\pi_+-\pi_-=\cG_g$ and $\pi_++\pi_-=\id_E\,$, the two resulting terms equal to zero, 
\end{proof}
Let us now compute $\tr(\Gamma_v\circ \Gamma_u)$ in more detail. Now since for $u\in E_+$ and $v\in E_-\,$, the map $\Gamma_u$ is from $E_-$ to $E_+$ and $\Gamma_v$ is from $E_+\to E_-\,$, we first conjugate both homomorphisms to endomorphisms of $\mfg$ noting that $\pi_{\pm}|_{\mfg}:\mfg\rightarrow E_{\pm}$ is an isomorphism. Observing that the inverse map of $\pi_{\pm}|_{\mfg}$ is given by $2\,\pi|_{E_{\pm}}:E_{\pm}\rightarrow \mfg$, we thus define
\begin{equation*}
\begin{split}
\tilde{\Gamma}_u\in \End(\mfg) \, ,\qquad \tilde{\Gamma}_u:=2\,\pi\circ \Gamma_u\circ \pi_-=f_+\circ [u,\cdot]_H\circ \pi_-|_{\mfg} \, ,\\
\tilde{\Gamma}_v\in \End(\mfg) \, ,\qquad \tilde{\Gamma}_v:=2\,\pi\circ \Gamma_v\circ \pi_+=f_-\circ [v,\cdot]_H\circ \pi_+|_{\mfg} \, ,
\end{split}
\end{equation*}
for $u\in E_+$ and $v\in E_-\,$, where we have set $f_{\pm}:=2\,\pi\circ \pi_{\pm}:E\rightarrow \mfg$. Explicitly, the maps $f_{\pm}$ are given by
\begin{equation*}
f_{\pm}(X+\xi)=2\,\pi\left(\tfrac{1}{2}\left(X\pm \xi^{\sharp}+\xi\pm X^\flat\right)\right)=X\pm \xi^{\sharp} \, ,
\end{equation*}
for $X\in \mfg$, $\xi\in \mfg^*$. Finally, we set
\begin{equation*}
	\begin{split}
		\Gamma_X^+ &:\mfg\rightarrow \mfg \, ,\qquad \Gamma_X^+:=\tilde{\Gamma}_{\pi_+(X)} \, ,\\
		\Gamma_X^- &:\mfg\rightarrow \mfg \, ,\qquad \Gamma_X^-:=\tilde{\Gamma}_{\pi_-(X)} \, ,
	\end{split}
\end{equation*}
for $X\in \mfg$ and define the following bilinear form on $\mfg$:
\begin{equation}\label{eq:beta}
	\beta:\mfg\times \mfg\rightarrow \bR \, ,\qquad \beta(X,Y):=2\tr(\Gamma_X^-\circ \Gamma_{Y}^+) \, .
\end{equation}
Using this notation, we may now re-express the generalised Einstein condition as follows:
\begin{lemma}\label{le:gE}
	Let $(\mfg,H,\cG_g)$ be a generalised pseudo-Riemannian Lie algebra. Then $(\mfg,H,\cG_g,\delta)$ is generalised Einstein if and only if
	\begin{equation}\label{eq:GEbeta}
		\beta(X,Y)=-\delta(\cG_g\, [\pi_+(Y),\pi_-(X)]_H) \, ,\qquad \delta([\pi_+(Y),\pi_-(X)]_H)=0 \, ,
	\end{equation}
	for all $X,Y\in \mfg$.
\end{lemma}
\begin{proof}
Let $u\in E_+$ and $v\in E_-\,$. Since $\pi_{\pm}|_{\mfg}:\mfg\rightarrow E_{\pm}$ are isomorphisms, $u=\pi_+(Y)$ and $v=\pi_-(X)$ for uniquely defined $X,Y\in \mfg$. This shows that the second equation in \eqref{eq:gE} is equivalent to the second equation in \eqref{eq:GEbeta}. Moreover,
\begin{equation*}
\begin{split}
\beta(X,Y)&=2\tr(\Gamma_X^-\circ \Gamma_{Y}^+)=2\tr(\tilde{\Gamma}_{\pi_-(X)}\circ \tilde{\Gamma}_{\pi_+(Y)})\\
&=2\tr((\pi_-|_{\mfg})^{-1}\circ\Gamma_{\pi_-(X)}\circ \Gamma_{\pi_+(Y)}\circ \pi_-|_{\mfg})=2\tr(\Gamma_{\pi_-(X)}\circ \Gamma_{\pi_+(Y)}) \, ,
\end{split}
\end{equation*}
and so also the first equation in \eqref{eq:gE} is equivalent to the first equation in \eqref{eq:GEbeta}.
\end{proof}

Our next goal is to derive an explicit formula for the bilinear form $\beta$. For this, we first compute $\Gamma_X^+,\Gamma_X^-\in \End(\mfg)$ more explicitly. Let us start with $\Gamma_X^+$:
\begin{equation*}
	\begin{split}
		4\, \Gamma^+_X(Y)&=4\, \tilde{\Gamma}_{\frac{1}{2}(X+X^\flat)}(Y)=f_+([X+X^\flat,Y-Y^\flat]_H)\\
		&=f_+([X,Y]-dY^\flat(X)-dX^\flat(Y)+H(X,Y,\cdot))\\
		&=f_+([X,Y]+g(Y,\ad_X)+g(X,\ad_Y)+H(X,Y,\cdot))\\
		&=f_+(\ad_X(Y)+g(\ad_X^*(Y),\cdot)+g(\ad_Y^*(X),\cdot)+H(X,Y,\cdot))\\
		&=\ad_X(Y)+\ad_X^*(Y)+\ad_Y^*(X)+H(X,Y,\cdot)^{\sharp}\\
		&=2\ad_X^S(Y)+\ad_Y^*(X)+H(X,Y,\cdot)^{\sharp} \, ,
\end{split}
\end{equation*}
where $f^S$ denotes the symmetric part of an endomorphism $f$ with respect to $g$. Using that $[X-X^\flat,Y+Y^\flat]_H=-[Y+Y^\flat,X-X^\flat]_H\,$, we compute, similarly, that
\begin{equation*}
	\begin{split}
		4\, \Gamma^-_X(Y)&=f_-([X-X^\flat,Y+Y^\flat]_H)=-f_-([Y+Y^\flat,X-X^\flat]_H)\\
		&=-f_-(\ad_Y(X)+g(\ad_Y^*(X),\cdot)+g(\ad_X^*(Y),\cdot)+H(Y,X,\cdot))\\
		&=\ad_X(Y)+\ad_Y^*(X)+\ad_X^*(Y)-H(X,Y,\cdot)^{\sharp}\\
		&=2\, \ad_X^S(Y)+\ad_Y^*(X)-H(X,Y,\cdot)^{\sharp} \, .
	\end{split}
\end{equation*}
\begin{remark}\label{Bismut:rem}
	Comparing with \cite[Proposition 3.14]{GSt}, the above formulas show that $\nabla^{\pm}_XY=-2\, \Gamma^{\pm}_YX$ with $\nabla^\pm =\nabla^g\pm \frac12 H^\sharp$ being the Bismut connections for $(g,H)$, i.e.\ the unique metric connection with totally skew-symmetric torsion $H$ or $-H$, respectively. Note also that comparing the formulas for the various Ricci tensors in \cite{GSt} one can easily verify that the vanishing of the generalised Ricci curvature $\cRc_{(\cG,\delta)}^+$ in the case of constant dilaton, i.e.\ when 
	$\delta = \delta^g$, is equivalent to the vanishing of the Ricci curvature of $\nabla^+$ (and similarly for $\cRc_{(\cG,\delta)}^-$ and $\nabla^-$). Here $\delta^g\in \mathfrak g^*\subset E^*$ stands for the (pseudo-)Riemannian divergence $\delta^g(X) = \tr \nabla^gX$. 
	However in this paper we do not restrict to the case of constant dilaton. In \cite{CK} it is shown that $\delta^g=-\tau$, where $X \mapsto \tau (X) 
	= \tr (\mathrm{ad}_X)$ is the trace-form.

\end{remark}
Now we note that the endomorphism $\ad^*(X)\in \End(\mfg)$ of $\mfg$ given by
\begin{equation}\label{eq:defofadstarX}
\mfg \ni Y\mapsto \ad^*(X)(Y):=\ad_Y^*(X)\in \mfg
\end{equation}
is skew-symmetric as
\begin{equation*}
g(\ad_Y^*(X),Z)=g(X,\ad_Y(Z))=-g(\ad_Z(Y),X)=-g(Y,\ad^*_Z(X)) \, .
\end{equation*}
Moreover, also the endomorphism $H_X\in \End(\mfg)$ of $\mfg$ defined by
\begin{equation}\label{eq:defofHX}
\mfg\ni Y\mapsto H_X(Y):= H(X,Y,\cdot)^{\sharp}\in \mfg
\end{equation}
is skew-symmetric since
\begin{equation*}
	g(H(X,Y,\cdot)^{\sharp},Z)=H(X,Y,Z)=-H(X,Z,Y)=-g(Y,H(X,Z,\cdot)^{\sharp}) \, .
\end{equation*}
Thus, since $\ad_X^S$ is by definition symmetric, we have
\begin{equation*}
4 (\Gamma^+_X)^*=2\ad_X^S-\ad^*(X)-H_X \, ,\qquad 4 (\Gamma^-_X)^*=2\ad_X^S-\ad^*(X)+H_X \, .
\end{equation*}
Next, choose an orthonormal basis $(e_1,\ldots,e_n)$ of $(\mfg,g)$ with $g(e_i,e_i)=\epsilon_i$ for $i=1,\ldots,n$. Recalling that for any endomorphism $f\in \End(\mfg)$ of $\mfg$ we have
\begin{equation*}
	g(f,f)=\sum_{i=1}^n \epsilon_i\, g(f(e_i),f(e_i)) \, ,
\end{equation*}
and using that $g$-symmetric endomorphisms are orthogonal to $g$-antisymmetric endomorphisms, we compute now
\begin{equation*}
	\begin{split}
8\, \beta(X,Y)&=16\, \tr(\Gamma^-_X\circ \Gamma_Y^+)=16\,\sum_{i=1}^n \epsilon_i\, g(\Gamma^-_X(\Gamma^+_Y(e_i)),e_i)
\\
&=16\,\sum_{i=1}^n \epsilon_i\, g(\Gamma^+_Y(e_i),(\Gamma^-_X)^*(e_i))=16\,g(\Gamma^+_Y,(\Gamma^-_X)^*)\\
&=g(2\ad_Y^S+\ad^*(Y)+H_Y,2\ad_X^S-\ad^*(X)+H_X)\\
&=4\, g(\ad_X^S,\ad_Y^S)+g(-\ad^*(X)+H_X,\ad^*(Y)+H_Y)\\
&=4\, g(\ad_X^S,\ad_Y^S)-g(\ad^*(X),\ad^*(Y))+2\,g(H(X,\cdot,\cdot),H(Y,\cdot,\cdot))\\
&+g(H_X,\ad^*(Y))-g(H_Y,\ad^*(X)) \, ,
	\end{split}
\end{equation*}
for all $X,Y\in \mfg$, where we used that $g(H_X,H_Y)=2\, g(H(X,\cdot,\cdot),H(Y,\cdot,\cdot))$ in the usual convention for the scalar product on two-forms that $g(e^i\wedge e^j,e^i\wedge e^j)=\epsilon_i\,\epsilon_j\,$. Thus,
the symmetric part $\beta^S$ of $\beta$ is given by
\begin{equation}\label{eq:betaS}
\beta^S(X,Y)=\frac{1}{2}\, g(\ad_X^S,\ad_Y^S)-\frac{1}{8}\, g(\ad^*(X),\ad^*(Y))+\frac{1}{4}\,g(H(X,\cdot,\cdot),H(Y,\cdot,\cdot)) \, ,
\end{equation}
whereas the anti-symmetric part $\beta^A$ of $\beta$ is given by
\begin{equation}\label{eq:betaA}
	\beta^A(X,Y)=\frac{1}{8}\, \left(g(H_X,\ad^*(Y))-g(H_Y,\ad^*(X))\right) \, .
\end{equation}
Computing that
\begin{equation*}
\begin{split}
\cG_g\, [\pi_+(Y),\pi_-(X)]_H&=\frac{1}{4} \cG_g\, [Y+Y^\flat,X-X^\flat]_H\\
&=\frac{1}{4} \cG_g\left([Y,X]+g(\ad_Y^*(X)+\ad_X^*(Y),\cdot)+H(Y,X,\cdot)\right)\\
&=-\frac{1}{4} \cG_g\left([X,Y]-g(\ad_X^*(Y)+\ad_Y^*(X),\cdot)+H(X,Y,\cdot)\right)\\
&=-\frac{1}{4} \left(-\ad_X^*(Y)-\ad_Y^*(X)+H(X,Y,\cdot)^{\sharp}+ [X,Y]^\flat\right) \, ,
\end{split}
\end{equation*}
we see that the symmetric part of $-\delta(\cG_g\, [\pi_+(Y),\pi_-(X)]_H)$ is given by
\begin{equation*}
-\frac{1}{4} \delta\left(\ad_X^*(Y)+\ad_Y^*(X)\right) \, ,
\end{equation*}
whereas the anti-symmetric part of $-\delta(\cG_g\, [\pi_+(X),\pi_-(Y)]_H)$ is given by
\begin{equation*}
\frac{1}{4} \delta\left(H(X,Y,\cdot)^{\sharp}+ [X,Y]^\flat\right) \, .
\end{equation*}
The above calculation also yields that the symmetric part of $\delta([\pi_+(Y),\pi_-(X)]_H)$ is given by
\begin{equation*}
\frac{1}{4}\delta((\ad_X^*(Y))^\flat+(\ad_Y^*(X))^\flat) \, ,
\end{equation*}
whereas the anti-symmetric part of $\delta([\pi_+(Y),\pi_-(X)]_H)$ equals
\begin{equation*}
	-\frac{1}{4}\delta([X,Y]+H(X,Y,\cdot)) \, .
\end{equation*}
Recalling that a symmetric bilinear form is zero when the associated quadratic form is zero, our computations imply the following more explicit reformulation of Lemma \ref{le:gE}:
\begin{proposition}\label{pro:gEexplicitly}
Let $(\mfg,H,\cG_g)$ be a generalised pseudo-Riemannian Lie algebra and $\delta\in E^*$. Then $(\mfg,H,\cG_g,\delta)$ is generalised Einstein if and only if
\begin{equation}\label{eq:GEbetaexplicitly}
\begin{split}
-2\,\delta(\ad_X^*(X))=4\, g(\ad_X^S,\ad_X^S)&-g(\ad^*(X),\ad^*(X))+2\, g(H(X,\cdot,\cdot),H(X,\cdot,\cdot)) \, ,\\
2\, \delta\left(H(X,Y,\cdot)^{\sharp}+ [X,Y]^\flat\right)&=g(H_X,\ad^*(Y))-g(H_Y,\ad^*(X)) \, ,\\
\delta((\ad_X^*(X))^\flat)&=0 \, ,\\
\delta([X,Y]+H(X,Y,\cdot))&=0 \, ,
\end{split}
\end{equation}
for all $X,Y\in \mfg$.	
\end{proposition}
From Section \ref{sec:almostAbelian} on, we will concentrate only on generalised pseudo-Riemannian Lie algebras $(\mfg,H,\cG_g)$ which are generalised Einstein for zero divergence $\delta=0$. This is surely a restriction, however, we observe from Proposition \ref{pro:gEexplicitly} that then, under suitable assumptions, the generalised Einstein condition also holds for some non-zero divergences. On the other hand, the condition to be generalised Einstein for some $\delta\in E^*$ puts restrictions on the possible divergences $\delta$. We summarise our observations in the following corollary, where we have used that for a quadratic Lie algebra $(\mfg,g)$, all endomorphisms $\ad_X$ are skew-symmetric and so $\ad_X^S=0$ and $\ad^*(X)=\ad_X$ for all $X\in \mfg$:
\begin{corollary}\label{co:gEdeltanonzero}
Let $(\mfg,H,\cG_g)$ be a generalised pseudo-Riemannian Lie algebra.
\begin{enumerate}[(a)]
\item
Let $(\mfg,H,\cG_g,0)$ be generalised Einstein for zero divergence. We then have that $(\mfg,H,\cG_g,\delta)$ is generalised Einstein for some $\delta\in E^*$ if and only if $\delta$ satisfies
\begin{equation*}
\begin{split}
\delta(\ad_X^*(X))=\delta((\ad_X^*(X))^\flat)&=0 \, ,\\
\delta\left( [X,Y]+H(X,Y,\cdot)\right)&=0 \, ,\\
\delta\left(H(X,Y,\cdot)^{\sharp}+ [X,Y]^\flat\right)&=0 \, ,
\end{split}
\end{equation*}
for all $X,Y\in \mfg$.
\item
Let $H=0$ and let $(\mfg,H=0,\cG_g,\delta)$ be generalised Einstein for some $\delta\in E^*$. Then $\delta(\mfg'\oplus (\mfg')^\flat)=0$. In particular, if $\mfg$ is perfect, e.g. when $\mfg$ is semisimple,  we must have $\delta=0$.
\item 
Let $H=0$ and $\ad_X^*(X)=0$ for all $X\in \mfg$. Then $(\mfg,H=0,\cG_g,\delta)$ is generalised Einstein for some $\delta\in E^*$ with $\delta(\mfg'\oplus (\mfg')^\flat)=0$ (e.g. $\delta=0$) if and only if it is generalised Einstein for all $\delta\in E^*$ with $\delta(\mfg'\oplus (\mfg')^\flat)=0$.
\item
Let $(\mfg,g)$ be a quadratic Lie algebra. Then $(\mfg, H=0,\cG_g,\delta)$ is generalised Einstein if and only if $\ad_X$ is null for all $X\in \mfg$ and $\delta(\mfg'\oplus (\mfg')^\flat)=0$.
\end{enumerate}

\end{corollary}
\section{Generalised Einstein Riemannian Lie algebras}
\label{sec:Riemannian}
In this section we show that, in the Riemannian case, the generalised Einstein condition always reduces to the generalised Einstein condition on the commutator ideal. This reduction allows us to give a full classification of solvable generalised Einstein Riemannian Lie algebras. Using results from \cite{CK}, we are able to obtain also a full classification of all four-dimensional generalised Einstein Riemannian Lie algebras.

We start with the mentioned reduction to the commutator ideal:
\begin{theorem}\label{th:structRiemGE}
Let $(\mfg,H,\mathcal{G}_g)$ be a generalised Riemannian Lie algebra and $\delta\in E^*$ be a divergence operator. Then $(\mfg,H,\mathcal{G}_g,\delta)$ is generalised Einstein if and only if $\mfh:=(\mfg')^{\perp}$ is an Abelian Lie subalgebra of $\mfg$ which acts skew-symmetrically on $\mfg'$, $\delta([\mfh,\mfg']\oplus[\mfh,\mfg']^\flat)=0$, $X\hook H=0$ for all $X\in \mfh$ and $(\mfg',H|_{\mfg'},\mathcal{G}_{g|_{\mfg'}},\delta|_{\mfg'\oplus (\mfg')^*})$ is generalised Einstein.

In this case, $\mfg\cong \mfg'\rtimes \bR^m$ ($m=\dim \mathfrak h$) as Lie algebras and $H$ may be considered as a three-form on $\mfg'$.
\end{theorem}
\begin{proof}
Denote by $H'$ the restriction of $H$ to $\mfg'$ and observe first that $\ad^*(X)=0$ for any $X\in \mfh$ since for any $Y\in \mfg$ we have 
\begin{equation*}
X\in \mfh=(\mfg')^{\perp}\subseteq \mathrm{im}(\ad_Y)^{\perp}=\ker(\ad_Y^*) \, .
\end{equation*}
Hence, the first equation in \ref{eq:GEbetaexplicitly} reduces for $X\in \mfh$ to
\begin{equation*}
0=4\, g(\ad_X^S,\ad_X^S)+2\, g(H(X,\cdot,\cdot),H(X,\cdot,\cdot)) \, .
\end{equation*}
As $g$ is Riemannian, this implies $\ad_X^S=0$ and $X\hook H=0$, i.e. $H=H'$. Now if $\tilde{X}\in \mfh$, then
\begin{equation*}
0=2\ad_X^S(\tilde{X})=\ad_X(\tilde{X})+\ad^*_X(\tilde{X})=[X,\tilde{X}] \, ,
\end{equation*}
i.e. $\mfh$ is Abelian. Hence, $\ad_X^S=0$ for all $X\in \mfh$ is equivalent to $\mfh$ being Abelian and acting skew-symmetrically on $\mfg'$.

We observe that under these conditions, the validity of the second equation in \eqref{eq:GEbetaexplicitly} for pairs $(X,Y)\in \mfh\times \mfg'$ is equivalent to $\delta([X,Y]^\flat)=0$ for all such pairs, i.e. to $\delta|_{[\mfh,\mfg']^\flat}=0$. Now if $X,Y$ are both in $\mfg'$, the second equation in \eqref{eq:GEbetaexplicitly} reduces to the corresponding generalised Einstein equation for $(\mfg',H',\cG_{g|_{\mfg'}},\delta|_{\mfg'\oplus (\mfg')^*})$ due to $H=H'$ and so also
\begin{equation*}
g(H_X,\ad^*(Y))-g(H_Y,\ad^*(X))=g(H'_X,(\ad^{\mfg'})^*(Y))-g(H'_Y,(\ad^{\mfg'})^*(X)) \, .
\end{equation*}
Let us look at the other equations. We observe that the validity of the fourth equation in \eqref{eq:GEbetaexplicitly} for pairs $(X,Y)\in \mfh\times \mfg'$ is equivalent to $\delta|_{[\mfh,\mfg']}=0$ and that if $(X,Y)$ are both in $\mfh$, the fourth equation reduces to the corresponding generalised Einstein equation for $(\mfg',H',\cG_{g|_{\mfg'}},\delta|_{\mfg'\oplus (\mfg')^*})$. Moreover, the third equation in \eqref{eq:GEbetaexplicitly} is automatically fulfilled for $X\in \mfh$ and reduces to the corresponding generalised Einstein equation for $(\mfg',H',\cG_{g|_{\mfg'}},\delta|_{\mfg'\oplus (\mfg')^*})$ if $X\in \mfg^*$. However, additionally, we have to polarise the third equation in \eqref{eq:GEbetaexplicitly} and then insert $X\in \mfh$ and $Y\in \mfg'$ into that equation, which yields that we must have
\begin{equation*}
0=\delta((\ad_X^*(Y)+\ad_Y^*(X))^\flat)=\delta((\ad_X(Y))^\flat)=\delta([X,Y]^\flat) \, ,
\end{equation*}
which, however, holds due to $\delta|_{[\mfh,\mfg']^\flat}=0$. 

We still need to investigate the first equation in \eqref{eq:GEbetaexplicitly} for $X$ not being in $\mfh$. Note that the polarisation of this equation for $X\in \mfh$ and $Y\in \mfg'$ yields zero on the right hand side due to $\ad_X^S=0$, $\ad^*(X)=0$ and $X\hook H=0$. The left hand side is given by $-2\delta(\ad_X^*(Y)+\ad_Y^*(X))=2\delta([X,Y])=0$ due to $\delta|_{[\mfh,\mfg']}=0$. 

So finally, we need to check that the first equation reduces for elements in $\mfg'$ to the corresponding generalised Einstein equation for $(\mfg',H',\cG_{g|_{\mfg'}},\delta|_{\mfg'\oplus (\mfg')^*})$. We denote the element that we use by $Y\in \mfg'$ and so, due to $H=H'$, need to show that
\begin{equation*}
\begin{split}
4\,g((\ad^{\mfg}_Y)^S,(\ad^{\mfg}_Y)^S) -g((\ad^{\mfg})^*(Y),(\ad^{\mfg})^*(Y))\\
=4\,g((\ad^{\mfg'}_Y)^S,(\ad^{\mfg'}_Y)^S) -g((\ad^{\mfg'})^*(Y),(\ad^{\mfg'})^*(Y)) \, ,
\end{split}
\end{equation*}
for all $Y\in \mfg'$, where the different upper indices denote if we consider the corresponding endomorphisms on $\mfg$ or $\mfg'$. To prove this assertion, we note first that
\begin{equation*}
\ad^{\mfg}_Y=\begin{pmatrix} \ad^{\mfg'}_Y & f_Y \\ 0 & 0 \end{pmatrix}
\end{equation*}
with respect to the splitting $\mfg=\mfg'\oplus \mfh$ for the linear map $f_Y:\mfh\rightarrow \mfg'$, $f_Y(X)=\ad^{\mfg}_Y(X)$, and so
\begin{equation*}
2(\ad^{\mfg}_Y)^S=\begin{pmatrix} 2(\ad^{\mfg'}_Y)^S & f_Y\\ f_Y^* & 0 \end{pmatrix} \, .
\end{equation*}
Moreover, we have
\begin{equation*}
(\ad^{\mfg})^*(Y)=\begin{pmatrix} (\ad^{\mfg'})^*(Y) & h_Y \\ -h_Y^* & 0 \end{pmatrix} \, ,
\end{equation*}
for a linear map $h_Y:\mfh\rightarrow \mfg'$ defined by $h_Y(X)=(\ad^{\mfg}_X)^*(Y)$ for $X\in \mfh$, using that $(\ad^{\mfg})^*(Y)$ is skew-symmetric and noting that the lower right corner of the matrix is zero since $\mfh$ is Abelian. Now since $\ad^{\mfg}_X$ is skew-symmetric, we get
\begin{equation*}
h_Y(X)=(\ad^{\mfg}_X)^*(Y)=-\ad^{\mfg}_X(Y)=\ad^{\mfg}_Y(X)=f_Y(X) \, ,
\end{equation*}
for all $X\in \mfh$, i.e. $f_Y=h_Y\,$. Therefore,
\begin{equation*}
\begin{split}
	\left\|2(\ad^{\mfg}_Y)^S\right\|^2-\left\|(\ad^{\mfg})^*(Y)\right\|^2&=	\left\|2(\ad^{\mfg'}_Y)^S\right\|^2+\left\|f_Y\right\|^2+\left\|f_Y^*\right\|^2\\
	&-\left\|(\ad^{\mfg'})^*(Y)\right\|^2-\left\|h_Y\right\|^2-\left\|-h_Y^*\right\|^2 \\
 &=\left\|2(\ad^{\mfg'}_Y)^S\right\|^2-\left\|(\ad^{\mfg'})^*(Y)\right\|^2 \, ,
\end{split}
\end{equation*}
which finishes the proof.
\end{proof}
As an immediate corollary, we obtain:
\begin{corollary}\label{co:Riemnil}
Let $(\mfg,H,\mathcal{G}_g)$ be a \emph{nilpotent} generalised Riemannian Lie algebra and $\delta\in E^*$. Then $(\mfg,H,\mathcal{G}_g,\delta)$ is generalised Einstein if and only if $\mfg$ is Abelian and $H=0$.
\end{corollary}
\begin{proof}
By Theorem \ref{th:structRiemGE}, the vector space $\mfh:=(\mfg')^{\perp}$ is an Abelian Lie algebra which acts skew-symmetrically on $\mfg'$ and $X\hook H=0$ for all $X\in \mfh$. By Engel's theorem, the endomorphism $\ad_X$ is also nilpotent for any $X\in \mfh$, so we must have $\ad_X=0$ for all $X\in \mfh$. Consequently, $\mfg\cong \mfg'\oplus \bR^m$ as Lie algebras. Therefore, $\mfg'=[\mfg,\mfg]=[\mfg',\mfg']$, i.e. $\mfg'$ is a perfect nilpotent Lie algebra, which implies $\mfg'=\{0\}$ and $\mfg=\mfh$. Thus, $\mfg$ is Abelian and $H=0$. Conversely, it is clear from \eqref{eq:GEbetaexplicitly} that $(\bR^n,0,\mathcal{G}_g,\delta)$ is generalised Einstein for all $\delta\in E^*$.
\end{proof}
Corollary \ref{co:Riemnil} implies the following structural result on \emph{solvable} generalised Einstein Riemannian Lie algebras:
\begin{corollary}\label{co:Riemsolv}
Let $(\mfg,H,\mathcal{G}_g)$ be a solvable generalised Riemannian Lie algebra and $\delta\in E^*$. Then $(\mfg,H,\mathcal{G}_g,\delta)$ is generalised Einstein if and only if $H=0$, $\mfg'$ is an Abelian ideal, $\mfh:=(\mfg')^{\perp}$ is an Abelian subalgebra, $\mfh$ acts skew-symmetrically on $\mfg'$ and $\delta(\mfg'\oplus (\mfg')^\flat)=0$.

 Equivalently, $(\mfg,H,\mathcal{G}_g,\delta )$ is generalised Einstein if and only if $H=0$, $g$ is flat and $\delta(\mfg'\oplus (\mfg')^\flat)=0$.
\end{corollary}
\begin{proof}
As $\mfg$ is solvable, the commutator ideal $\mfg'$ is nilpotent. Moreover, by Theorem \ref{th:structRiemGE}, $(\mfg,H,\mathcal{G}_g,\delta)$ is generalised Einstein if and only if $\mfh$ is an Abelian subalgebra of $\mfg$ acting skew-symmetrically on $\mfh$, $\delta|_{[\mfh,\mfg']\oplus [\mfh,\mfg']^\flat}=0$, $X\hook H=0$ for all $X\in \mfh$ and $(\mfg',H|_{\mfg'},\mathcal{G}_{g|_{\mfg'}},\delta|_{\mfg'})$ is generalised Einstein. The latter condition is by Corollary \ref{co:Riemnil} equivalent to $\mfg'$ being Abelian and $H|_{\mfg'}=0$, i.e. to $H=0$. Moreover, as $\mfg'$ and $\mfh=(\mfg')^\perp$ are both Abelian, we have $\mfg'=[\mfh,\mfg']$. This shows the first claimed equivalence.

The second equivalence follows from the classical result of Milnor on the structure of flat Riemannian Lie algebras \cite{Mi}.
\end{proof}
In the four-dimensional case, we obtain the following explicit classification of all generalised Einstein Riemannian Lie algebras:
\begin{theorem}\label{th:4dRiem}
	Any four-dimensional generalised Einstein Riemannian Lie algebra $(\mfg,H,\cG_g,\delta)$ is isomorphic to one of the following four-dimensional generalised Einstein Riemannian Lie algebras:
	\begin{itemize}
		\item[(i)]
		$(\bR^4,H=0,\cG_g,\delta)$ for an arbitrary Riemannian metric $g$ and an arbitary divergence operator $\delta\in E^*$,
		\item[(ii)]
		or $(\mathfrak{e}(2)\oplus \bR, H=0,\cG_g,\delta)$ such that $(\mathfrak{e}(2)\oplus \bR)$ admits an orthonormal basis $(e_1,\ldots,e_4)$ with the only non-zero Lie brackets (up to anti-symmetry) being
		\begin{equation*}
			[e_3,e_1]=a\, e_2 \, ,\quad [e_3,e_2]=-a\, e_1 \, ,
		\end{equation*}
		for some $a>0$ and $\delta\in \spa{e_3,e_4,e^3,e^4}$,
		\item[(iii)]
		or $(\mathfrak{so}(3)\oplus \bR, H,\cG_g,\delta)$ such that there exist $a\in \bR^*$, $b\in \bR$ and $\epsilon\in \{-1,1\}$ and an orthonormal basis $(e_1,\ldots,e_4)$ with $H=a\, e^{123}$,
	       the only non-zero Lie brackets (up to anti-symmetry) being
		\begin{equation*}
			[e_1,e_2]=\epsilon a\, e_3 \, ,\;\;\; [e_2,e_3]=\epsilon a\, e_1 \, ,\;\;\; [e_3,e_1]=\epsilon a\, e_2 \, ,\;\;\; [e_3,e_4]=b\, e_2 \, ,\;\;\; [e_4,e_2]=b\, e_3 \, .
		\end{equation*}
		 In addition, if $b\neq 0$ we have $\delta(e_1)=-\epsilon\,\delta(e^1)$, $\delta(\spa{e_2,e_3,e^2,e^3})=0$ and if $b=0$
		we have instead $\delta(e_i)=-\epsilon\,\delta(e^i)$ for all $i=1,2,3$.
	\end{itemize}
\end{theorem}
\begin{proof}
If $\mfg$ is Abelian, Corollary \ref{co:Riemnil} shows that $(\mfg,H,\cG_g,\delta)$ has to be as in (i). 

If $\mfg$ is solvable but not Abelian, then Corollary \ref{co:Riemsolv} yields $H=0$, $\mfg=\mfg'\rtimes \mfh$ with $\mfg'$ being Abelian and $\mfh=(\mfg')^{\perp}$ acting skew-symmetrically on $\mfg'$. As $\mfg$ is not Abelian, this action cannot be trivial and so $\dim(\mfg')=\dim(\mfh)=2$. We may then choose an orthonormal basis $(e_1,e_2)$ of $\mfg'$ and an orthonormal basis $(e_3,e_4)$ of $\mfh$ such that $\ad_{e_4}=0$, i.e. $e_4$ is central, and 	$[e_3,e_1]=a\, e_2\,$, $[e_3,e_2]=-a\, e_1$ for some $a>0$. Then $\mfg\cong \mathfrak{e}(2)\oplus \bR$ and $\delta([\mfh,\mfg']\oplus [\mfh,\mfg']^\flat)=\delta(\spa{e_1,e_2,e^1,e^2})=0$, i.e. $\delta\in \spa{e_3,e_4,e^3,e^4}$

Finally, let us assume that $\mfg$ is non-solvable. By the classification of four-dimensional Lie algebras, $\mfg\cong \mathfrak{so}(3)\oplus \bR$ or $\mfg\cong \mathfrak{so}(2,1)\oplus \bR$ and so the commutator ideal $\mfg'$ is either isomorphic to $\mathfrak{so}(3)$ or to $\mathfrak{so}(2,1)$. Theorem \ref{th:structRiemGE} imposes that the restrictions $(H',\cG_{g'},\delta')$ of $(H,\cG_g,\delta)$ to $\mfg'$ must give a generalised Riemannian Einstein metric on $\mfg'$, and by \cite{CK} only $\mathfrak{so}(3)$ admits such a metric. Therefore, we must have $\mfg'\cong \mathfrak{so}(3)$, i.e. $\mfg\cong  \mathfrak{so}(3)\oplus \bR$, and by \cite[Theorem 3.12]{CK}, there exists an orthonormal basis $(e_1,e_2,e_3)$ of $\mfg'$ and $a\in \bR^*$, $\epsilon\in \{-1,1\}$ such that
\begin{equation*}
[e_1,e_2]=\epsilon a\, e_3 \, ,\quad [e_2,e_3]=\epsilon a\, e_1 \, ,\quad [e_3,e_1]=\epsilon a\, e_2 \, ,\quad H'=a\, e^{123} \, ,
\end{equation*}
and $\delta|_{E_{\epsilon}}=0$, which is equivalent to $\delta(e_i)=-\epsilon\delta(e^i)$ for $i=1,2,3$. Now note first that by Theorem \ref{th:structRiemGE}, $H=H'$, and so $H=a\, e^{123}$. Moreover, choose $e_4$ such that $(e_1,e_2,e_3,e_4)$ is an orthonormal basis of $\mfg$. Then
\begin{equation*}
e_4=f+\sum_{i=1}^3 b_i\, e_i \, ,
\end{equation*}
for some $f\in \bR\subseteq \mathfrak{so}(3)\oplus \bR$, $f\neq 0$ and certain $b_1,b_2,b_3\in \bR$. By applying a rotation to the orthonormal basis $(e_1,e_2,e_3)$, we may assume that $b_2=b_3=0$. Setting $b:=b_1 a \epsilon$, we thus have
\begin{equation*}
[e_4,e_2]=b\, e_3 \, ,\quad [e_4,e_3]=-b\, e_2 \, .	
\end{equation*}
Finally, if $b\neq 0$ then $[\mfh,\mfg']=\spa{e_1,e_2}$, and so $\delta(\spa{e_2,e_3,e^2,e^3})=0$ by Theorem \ref{th:structRiemGE}.
\end{proof}
\section{Almost Abelian generalised Einstein Lie algebras with zero divergence}\label{sec:almostAbelian}
In this section, we will give, up to one case, a full classifications of so-called \emph{almost Abelian} generalised Riemannian or Lorentzian Lie algebras $(\mfg,H,\cG_g)$ with zero divergence.

For this, we first recall the definition of \emph{almost Abelian}, and, more generally, of \emph{almost nilpotent} Lie algebras:
\begin{definition}
	A Lie algebra $\mfg$ is called \emph{almost nilpotent} if it admits a codimension one nilpotent ideal $\mfn$. If $\mfn$ is actually Abelian, we call $\mfg$ \emph{almost Abelian}.
\end{definition}
\begin{remark}
Note that in low dimensions, the vast majority of Lie algebras are almost nilpotent or almost Abelian:
	
	In dimensions $1$ and $2$, all Lie algebras are almost Abelian. In dimension $3$, all but the simple Lie algebras $\mathfrak{so}(3)$ and $\mathfrak{so}(2,1)$ are almost Abelian. In dimension $4$---which is the most relevant case for us---all but the reductive Lie algebras $\mathfrak{so}(3)\oplus \bR$ and $\mathfrak{so}(2,1)\oplus \bR$ as well as two other Lie algebras, namely $\mathfrak{aff}_\bC$ and $\mathfrak{aff}_\bR\oplus \mathfrak{aff}_\bR$, are almost nilpotent.
\end{remark}
Note that if $\mfg$ is almost Abelian with codimension one Abelian ideal $\mfn$ and if $X\in \mfg\setminus \mfn$, then the endomorphism $f:=\ad_X|_{\mfn}\in \End(\mfn)$ of $\mfn$ completely determines the Lie bracket structure of $\mfg$. Moreover, different choices of $X\in \mfg\setminus \mfn$ may only result in a non-zero scaling of $f$ and so $f$ is, essentially independent of this choice.

Referring to the results that we obtain later in this section, our classification reads as follows:
\begin{theorem}\label{th:almostAbelian}
Let $(\mfg,H,\cG_g)$ be an almost Abelian generalised Lorentzian Lie algebra of arbitrary dimension with codimension one Abelian ideal $\mfn$. Moreover, let $f\in \End(\mfn)$ be as above. In the case that $\mfn$ is non-degenerate and $H\neq 0$, we assume additionally that the symmetric part $f^S$ of $f$ has non non-real eigenvalue. Then $(\mfg,H,\cG_g)$ is generalised Einstein for zero divergence operator if and only if one of the following conditions hold:
\begin{itemize}
\item[(i)]
$\mfn$ is non-degenerate, $H=0$ and $f$ is as in Theorem \ref{th:almostAbeliannondegH=0},
\item[(ii)]
or $\mfn$ is non-degenerate, $H\neq 0$ and $f$ is as in Theorem \ref{th:almostAbeliannondegHneq0},
\item[(iii)]
or $\mfn$ is degenerate, $H=0$ and $f$ is as in Theorem \ref{th:almostAbeliandeg}.
\end{itemize}
\end{theorem}
We also obtain some result in the case that the signature of $g$ is $(n-2,2)$:
\begin{corollary}\label{co:signaturen-22}
Let $(\mfg,H,\cG_g)$ be an almost Abelian generalised pseudo-Riemannian Lie algebra of dimension $n$ and signature $(n-2,2)$ such that $H=0$ and such that $\mfn$ is Lorentzian. Let also $f\in \End(\mfn)$ be as above. Then $(\mfg,H,\cG_g)$ is generalised Einstein for zero divergence operator if and only if $f$ is as in Theorem \ref{th:almostAbeliannondegH=0} (b).
\end{corollary}
\begin{proof}
This follows directly from the fact that by Corollary \ref{co:almostAbelianH=0} (b), a generalised pseudo-Riemannian metric $(H=0,\cG_{g})$ on an almost Abelian Lie algebra $\mfg$ such that $\mfn$ is non-degenerate may only be generalised Einstein for $\delta=0$ if also any other generalised pseudo-Riemannian metric $(H=0,\cG_{g'})$ on $\mfg$ with $g|_{\mfn}=g'|_{\mfn}$ is generalised Einstein for $\delta=0$.
\end{proof}
Next, we specialise our results to dimension four. In this case, Theorem \ref{th:4dalmostAbeliannondegHneq0} below gives a full classification of \emph{four-dimensional} almost Abelian generalised Einstein Lorentzian Lie algebras $(\mfg,H,\cG_g,\delta=0)$ with non-degenerate codimension one Abelian ideal $\mfn$ and $H\neq 0$ by showing that in this case $f^S$ cannot have a complex non-real eigenvalue. Thus, applying Theorem \ref{th:almostAbeliannondegH=0} and Theorem \ref{th:almostAbeliandeg} to dimension four and combining them with Theorem \ref{th:4dalmostAbeliannondegHneq0} we obtain the following full classification of four-dimensional almost Abelian generalised Einstein Lorentzian Lie algebras $(\mfg,H,\cG_g,\delta=0)$:
\begin{theorem}\label{th:4dalmostAbelian}
Let $(\mfg,H,\cG_g)$ be a four-dimensional almost Abelian generalised Lorentzian Lie algebra with codimension one Abelian ideal $\mfn$ and let $f\in \End(\mfn)$ be above. Then $(\mfg,H,\cG_g)$ is generalised Einstein for zero divergence operator if and only if one of the following conditions hold:
\begin{itemize}
\item[(i)] $\mfn$ is Riemannian, 
\begin{equation*}
H=0 \, ,\quad f=\diag(L_1(0,a),0) \, ,
\end{equation*}
 for some $a\in \bR$ with respect to an orthonormal basis $(e_1,e_2,e_3)$ of $\mfn$,
\item[(ii)] or $\mfn$ is Lorentzian,
\begin{equation*}
H=0 \, ,\quad f=M_3(\sigma) \, ,
\end{equation*}
for some $\sigma\in \bR$ with respect to an orthonormal basis $(e_1,e_2,e_3)$ of $\mfn$ with $g(e_1,e_1)=-1$,
\item[(iii)] or $\mfn$ is Lorentzian,
\begin{equation*}
H=0 \, ,\quad f=\diag(L_1\left(\alpha,\sqrt{\alpha^2+\frac{a^2}{2}}\right),a) \, ,
\end{equation*}
for certain $(\alpha,a)\in \bR^2\setminus \{(0,0)\}$ with respect to an orthonormal basis $(e_1,e_2,e_3)$ of $\mfn$ with $g(e_1,e_1)=-1$,
\item[(iv)] or $\mfn$ is Lorentzian,
\begin{equation*}
H=0 \, ,\quad f=\begin{pmatrix}  \frac{\epsilon}{2} &  \frac{\epsilon}{2}  & v  \\ -\frac{\epsilon}{2} &  -\frac{\epsilon}{2}  & - v \\ v & v & 0 \end{pmatrix}  \, ,
\end{equation*}
 for some $\epsilon\in\{-1,1\}$, $v\in \bR$ with respect to an orthonormal basis $(e_1,e_2,e_3)$ of $\mfn$ with $g(e_1,e_1)=-1$,
\item[(v)] or $\mfn$ is Lorentzian,
\begin{equation*}
H=0 \, ,\quad f=L_3(0) \, ,
\end{equation*}
with respect to an orthonormal basis $(e_1,e_2,e_3)$ of $\mfn$ with $g(e_1,e_1)=-1$,
\item[(vi)] or $\mfn$ is Lorentzian,
\begin{equation*}
H=a (e^{12}+e^{13})\wedge X^\flat \, , \quad f= \frac{1}{\sqrt{2}}\begin{pmatrix} 0 & -1-\frac{a^2}{2} & 0 \\ 1-\frac{a^2}{2} & 0 & 1-\frac{a^2}{2}  \\ 0 & 1+\frac{a^2}{2}  & 0 \end{pmatrix} \, ,
\end{equation*}
for some $a\in \bR^*$ with respect to an orthonormal basis $(e_1,e_2,e_3)$ of $\mfn$ with $g(e_1,e_1)=-1$,
\item[(vii)] or $\mfn$ is degenerate, 
\begin{equation*}
H=0 \, , \quad f= \begin{pmatrix} 0 & -a & 0 \\ a & 0 & 0\\  b_1 & b_2 & 0 \end{pmatrix} \, ,
\end{equation*}
for certain $a,b_1,b_2\in \bR$ with respect to a basis $(e_1,e_2,e_3)$ of $\mfn$ such that $g(e_1,e_1)=g(e_2,e_2)=g(e_3,X)=1$, $g(e_3,e_i)=0$ for $i=1,2,3$, $g(X,e_1)=g(X,e_2)=0$.
\end{itemize}
Moreover, in each of the different cases the underlying Lie algebra $\mfg$ is isomorphic to the following four-dimensional Lie algebras:
\begin{itemize}
	\item[(i)] $e(2)\oplus \bR$ (for $a\neq 0$) or $\bR^4$ (for $a=0$).
	\item[(ii)] $e(2)\oplus \bR$ (for $\sigma>0$), $A_{4,1}$ (for $\sigma=0$) or
	$e(1,1)\oplus \bR$ (for $\sigma<0$).
	\item[(iii)] $A_{4,6}^{\sqrt{2}\cos(\varphi),\sin(\varphi)}$ for some $\varphi\in \bR$ (for $a\neq 0$) or $ \mathfrak{r}'_{3,1}\oplus \bR$ (for $a=0$).
	\item[(iv)]
	$A_{4,1}$ (for $v\neq 0$) or $\mfh_3\oplus \bR$ (for $v=0$).
	\item[(v)]
	$A_{4,1}$.
	\item[(vi)]
	$A_{4,1}$ for $a\neq\pm\sqrt{2}$, $\mfh_3\oplus \bR$ for $a=\pm\sqrt{2}$.
	\item[(vii)]
	$e(2)\oplus \bR$ (for $a\neq 0$), $\mfh_3\oplus \bR$ (for $a=0$) and $(b_1,b_2)\neq (0,0)$) or $\bR^4$ (for $a=b_1=b_2=0$).
\end{itemize}
\end{theorem}
\begin{remark}
We observe from Theorem \ref{th:4dalmostAbelian} that $A_{4,1}$ and $\mfh_3\oplus \bR$ are, up to isomorphism, the only four-dimensional almost Abelian Lie algebras $\mfg$ admitting a generalised Lorentzian metric $(H,\cG_g)$ with $H\neq 0$ which is generalised Einstein for $\delta=0$.
\end{remark}
In four dimensions, Corollary \ref{co:almostAbelian4dmetriconmfn} allows to give the following classification of almost Abelian generalised Einstein pseudo-Riemannian metrics of signature $(2,2)$ with non-degenerate codimension one Abelian ideal (noting that we may, by multiplying $g$ with $-1$ if necessary, assume that $\mfn$ is Lorentzian):
\begin{corollary}\label{co:signature22}
Let $(\mfg,H,\cG_g)$ be a four-dimensional almost Abelian generalised pseudo-Riemannian Lie algebra of signature $(2,2)$ such that $\mfn$ is non-degenerate and let $f\in \End(\mfn)$ be as in Theorem \ref{th:almostAbelian}. Then $(\mfg,H,\cG_g)$ is generalised Einstein for zero divergence operator if and only if $(H,f)$ are as in Theorem \ref{th:4dalmostAbelian} (ii) -- (vi), where the chosen orthonormal basis $(e_1,e_2,e_3)$ of $\mfn$ satisfies now $g(e_1,e_1)=-g(e_2,e_2)=-g(e_3,e_3)=-\epsilon$ for some $\epsilon \{-1,1\}$.
\end{corollary}
Before we begin our investigation, we note for reference how Proposition \ref{pro:gEexplicitly} simplifies in the case of zero divergence:
\begin{corollary}\label{co:gEdelta=0}
	Let $(\mfg,H,\cG_g)$ be a generalised pseudo-Riemannian Lie algebra. Then $(\mfg,H,\cG_g,0)$ is generalised Einstein if and only if
	\begin{equation}\label{eq:gEdelta=0}
		\begin{split}
			4\, g(\ad_Y^S,\ad_Y^S)-g(\ad^*(Y),\ad^*(Y))+2\, g(H(Y,\cdot,\cdot),H(Y,\cdot,\cdot))&=0 \, ,\\
			g(H_Y,\ad^*(Z))-g(H_Z,\ad^*(Y))&=0 \, ,
		\end{split}
	\end{equation}
	for all $Y,Z\in \mfg$.	
\end{corollary}
\subsection{Endomorphisms on Lorentzian vector spaces}\label{subsec:endosLorentz}
In the following subsections, we will work with well-known normal forms for symmetric or skew-symmetric endomorphism on a Lorentzian vector space. To recall these normal forms, we first need to define some matrices:
\begin{notation}
	We define the following matrices:
	\begin{equation*}
		\begin{split}
			L_1(\alpha,\beta)&:=\begin{pmatrix} \alpha & -\beta \\ \beta & \alpha \end{pmatrix} \, ,\qquad L_2(\alpha,\epsilon):=\epsilon\, \begin{pmatrix} \tfrac{1}{2}+\alpha & \tfrac{1}{2} \\ -\tfrac{1}{2} & -\tfrac{1}{2}+\alpha \end{pmatrix} \, ,\\
			L_3(\alpha)&:=\begin{pmatrix} \alpha & -\tfrac{1}{\sqrt{2}} & 0 \\ \tfrac{1}{\sqrt{2}} & \alpha & \tfrac{1}{\sqrt{2}}  \\ 0  &  \tfrac{1}{\sqrt{2}}  & \alpha \end{pmatrix} \, , \qquad M_3(\sigma):=\begin{pmatrix} 0 & 0 & -1+\sigma \\ 0 & 0 & -1-\sigma \\ -1+\sigma & 1+\sigma & 0 \end{pmatrix} \, ,\\
			M_4(\sigma,\tau)&:=\begin{pmatrix} 0 & 0 & -1+\sigma & \tau\\ 0 & 0 & -1-\sigma & \tau \\ -1+\sigma & 1+\sigma & 0 &0 \\
				\tau & -\tau & 0 & 0
			\end{pmatrix} \, .
		\end{split}
	\end{equation*}
\end{notation}
With this notation at hand, one has the following normal forms for symmetric or skew-symmetric endomorphisms on Lorentzian vector spaces:
\begin{lemma}[\cite{R}, \cite{MN}]\label{le:canonicalforms}
	Let $(V,g)$ be an $n$-dimensional Lorentzian vector space and $f\in \End(V)$ be an endomorphism of $V$. Then:
	\begin{itemize}
		\item[(a)]
		If $f$ is symmetric with respect to $g$, then $(V,g)$ admits an orthonormal basis $(e_1,\ldots,e_n)$ with $g(e_1,e_1)=-1$ such that $f$ equals, with respect to that basis, one of the following matrices
		\begin{equation*}
			\begin{split}
				&\diag(a_1,\ldots,a_n) \, ,\qquad \diag(L_1(\alpha,\beta),b_1,\ldots b_{n-2}) \, ,\\
				&\diag(L_2(\gamma,\epsilon),c_1,\ldots,c_{n-2}) \, ,\qquad \diag(L_3(\tau),d_1,\ldots,d_{n-3}) \, ,
			\end{split}
		\end{equation*}
		for certain $a_1,\ldots,a_n,b_1,\ldots,b_{n-2},c_1,\ldots,c_{n-2},d_1,\ldots,d_{n-3}\in \bR$,
		$\alpha,\gamma,\tau\in \bR$, $\beta>0$ and $\epsilon\in \{-1,1\}$.
		\item[(b)]
		If $f$ is anti-symmetric with respect to $g$, then $(V,g)$ admits an orthonormal basis $(e_1,\ldots,e_n)$ with $g(e_1,e_1)=-1$ such that either $n=2$ and $f=\left(\begin{smallmatrix} 0 & \rho \\ \rho & 0 \end{smallmatrix}\right)$ with respect to the basis $(e_1,e_2)$, or $n$ is odd and there exists some $k\in \{0,\ldots,\left\lfloor\tfrac{n-3}{2}\right\rfloor\}$, $a_1\geq \ldots\geq a_k>0$ and some $\sigma\in \bR$ such that
		\begin{equation*}
			f=\diag(M_3(\sigma),L_1(0,a_1),\ldots,L_1(0,a_k),0,\ldots,0) \, ,
		\end{equation*}
		with respect to the basis $(e_1,\ldots,e_n)$, or $n\geq 4$ is even and there exists some $l\in \{0,\ldots,\left\lfloor\tfrac{n-4}{2}\right\rfloor\}$, $b_1\geq \ldots \geq b_l>0$ and some $\tau\in \bR$, $\nu\geq 0$ such that
		\begin{equation*}
			f=\diag(M_4(\tau,\nu),L_1(0,b_1),\ldots,L_1(0,b_l),0,\ldots,0) \, ,
		\end{equation*}
		with respect to the basis $(e_1,\ldots,e_n)$.
	\end{itemize}
\end{lemma}
It will turn out to be useful to introduce some notation for the different possible normal forms of a symmetric endomorphism on a Lorentzian vector space:
\begin{notation}
	Let us call a symmetric endomorphism $f$ of a Lorentzian vector space $(V,g)$ of \emph{first}, \emph{second}, \emph{third} or \emph{fourth type} if $(V,g)$ has an orthonormal basis $(e_1,\ldots,e_n)$ with $g(e_1,e_1)=-1$ such that with respect to that basis $f$ equals the first, second, third or fourth matrix, respectively, mentioned in Lemma \ref{le:canonicalforms} (a).
\end{notation}
\begin{remark}
	We note a symmetric endomorphism $f$ of a Lorentzian vector space is of first, second, third or fourth type if and only if, respectively, $f$ is diagonalisable (over the reals), $f$ has a complex non-real eigenvalue, the Jordan normal form of $f$ has a real Jordan block of size $2$ or the Jordan normal form of $f$ has a real Jordan block of size $3$.
\end{remark}
By a direct computation, we obtain the following useful characterisations of the condition $\tr(f^2)=0$ of a symmetric endomorphism of a Lorentzian vector space, where the parameters are as in Lemma \ref{le:canonicalforms} (a):
\begin{lemma}\label{le:trfsquare=0}
	Let $f$ be a symmetric endomorphism of a Lorentzian vector space. Then:
	\begin{itemize}
		\item 
		If $f$ is of first type, then $\tr(f^2)\geq 0$ with equality if and only if $f=0$.
		\item
		If $f$ is of second tye, then $\tr(f^2)=0$ if and only if
		\begin{equation*}
			\beta=\sqrt{\alpha^2+\tfrac{1}{2}\sum_{i=1}^{n-2}b_i^2} \,.
		\end{equation*}
		\item
		If $f$ is of third type, then $\tr(f^2)\geq 0$ with equality if and only if $\gamma=c_1=\ldots=c_{n-2}=0$.
		\item
		If $f$ is of fourth type, then $\tr(f^2)\geq 0$ with equality if and only if $\tau=d_1=\ldots=d_{n-3}=0$.
	\end{itemize}
\end{lemma}
\subsection{Generalised Einstein Lie algebras with non-degenerate codimension one ideals}
In this subsection, we consider arbitary generalised Einstein pseudo-Riemannian Lie algebras $(\mfg,H,\cG_g,\delta=0)$ which admit a non-degenerate codimension one 
ideal $\mfa$. We will apply the results of this subsection to the case where $\mfg$ has arbitrary dimension and $\mfa$ is Abelian, i.e. to the almost Abelian case, in the rest of Section \ref{sec:almostAbelian}. Furthermore, in Section \ref{sec:4dLorentzian} we will also apply some of the formulas to the case where $\mfg$ is four-dimensional and $\mfa$ is either simple or the Heisenberg Lie algebra $\mfh_3\,$.
\begin{remark}
We note that, unfortunately, we cannot transfer Theorem \ref{th:structRiemGE} to pseudo-Riemannian case even if one assumes that the commutator ideal $\mfg'$ is non-degenerate and of codimension one. This is due to the fact that the equation
\begin{equation*}
	2\,g(\ad_X^S,\ad_X^S)+ g(H(X,\cdot,\cdot),H(X,\cdot,\cdot))=0
\end{equation*}
that one still gets from \eqref{eq:GEbetaexplicitly} for $X\in \mfh=\mfg'^\perp$ does no longer imply that $\ad_X^S=0$ and $H(X,\cdot,\cdot)=0$, so the whole strategy of the proof breaks down.
\end{remark}
\begin{notation}
In what follows, we choose $X$ orthogonal to $\mfa$ with $\epsilon:=g(X,X)\in \{-1,1\}$ and set
\begin{equation*}
	f:=\ad(X)|_{\mfa}\in \End(\mfa) \, .
\end{equation*}
Moreover, we write
\begin{equation*}
	H=H'+X^\flat\wedge B \, ,
\end{equation*}
with $H'\in \Lambda^3\mfa^*$ and $B\in \Lambda^2 \mfa^*$.
\end{notation}
We first prove:
\begin{proposition}\label{pro:codim1idealnondeg}
Let $(\mfg,H,\cG_g)$ be a generalised pseudo-Riemannian Lie algebra admitting a non-degenerate codimension one ideal $\mfa$. Let $f$, $H'$ and $B$ be defined as above. Then $(\mfg,H,\cG_g,0)$ is generalised Einstein if and only if
\begin{equation*}
\begin{split}
0&=2\, g(f^S,f^S)+ g(B,B) \, ,\\
0&=2\,g(f^S,(\ad_Y^{\mfa})^S)+g(B,H'(Y,\cdot,\cdot)) \, ,\\
0&=4\, g((\ad_Y^\mfa)^S,(\ad_Y^{\mfa})^S)-g((\ad^\mfa)^*(Y),(\ad^{\mfa})^*(Y))+2\epsilon\, g([f^*,f](Y),Y)\\
&+ 2\, g(H'(Y,\cdot,\cdot),H'(Y,\cdot,\cdot))+2\epsilon\, g(B(Y,\cdot),B(Y,\cdot)) \, ,\\
0&=g((\ad^{\mfa})^*(Y),B),\\
g((\ad^{\mfa}&)^*(Y),H'_Z)-2 g(f^*(Y),B(Z,\cdot)^{\sharp})=g((\ad^{\mfa})^*(Z),H'_Y)-2 g(f^*(Z),B(Y,\cdot)^{\sharp}) \, ,
\end{split}
\end{equation*}
for all $Y,Z\in \mfa$, where we consider $B$ in the fourth equation as an element of $\End(\mfa)$.
\end{proposition}
\begin{proof}
Consider $X\in\mfa^\perp$ as in the Notation above and observe that $\ad^*(X)=0$. Inserting $X$ into the first equation of \eqref{eq:gEdelta=0}, we obtain the first equation in Proposition \ref{pro:codim1idealnondeg}. Next, writing no upper index for the corresponding operator on $\mfg$ and upper index $\mfa$ for the corresponding operator on $\mfa$, we have for $Y\in \mfa$
\begin{equation*}
\ad_Y=\begin{pmatrix} \ad_Y^{\mfa} & -f(Y)\\ 0 & 0 \end{pmatrix} \, ,
\end{equation*}
with respect to the splitting $\mfg=\mfa\oplus \spa{X}$. Hence,
\begin{equation*}
	\ad^S_Y=\begin{pmatrix} (\ad_Y^{\mfa})^S & -\frac{1}{2}\, f(Y)\\ -\frac{1}{2\epsilon}f(Y)^\flat & 0 \end{pmatrix} \, .
\end{equation*}
Thus, polarising the first equation in \eqref{eq:gEdelta=0}, inserting $X$ and $Y\in \mfa$ into that equation and using that $\ad^*(X)=0$, we arrive at the second equation in Proposition \ref{pro:codim1idealnondeg}.

Next, computing
\begin{equation*}
\begin{split}
g(\ad^*(Y)(Z),W)&=g(\ad_Z^*(Y),W)=g(Y,\ad_Z(W))=g(Y,\ad^{\mfa}_Z(W))\\
&=g((\ad^{\mfa}_Y)^*(Z),W)=g((\ad^{\mfa})^*(Y)(Z),W) \, ,\\	
g(\ad^*(Y)(Z),X)&=g(Y,\ad_Z(X))=-g(Y,f(Z))=-g(f^*(Y),Z) \, ,\\
g(\ad^*(Y)(X),Z)&=g(Y,\ad_X(Z))=g(Y,f(Z))=g(f^*(Y),Z) \, ,\\
g(\ad^*(Y)(X),X)&=0 \, ,
\end{split}
\end{equation*}
for all $Y,Z,W\in \mfa$, we obtain
\begin{equation*}
	\ad^*(Y)=\begin{pmatrix} (\ad_Y^{\mfa})^* & f^*(Y)\\ -\frac{1}{\epsilon}(f^*(Y))^\flat & 0 \end{pmatrix} \, .
\end{equation*}
Thus, inserting now $Y\in \mfa$ into the first equation in \eqref{eq:gEdelta=0}, we obtain
\begin{equation*}
\begin{split}
0&=4\, g((\ad^{\mfa}_Y)^S,(\ad^{\mfa}_Y)^S)+4\, g(\tfrac{1}{2}\,  X^\flat\otimes f(Y),\tfrac{1}{2}\, X^\flat\otimes f(Y))\\
&+4\, g(\tfrac{1}{2}\, f(Y)^\flat\otimes X,\tfrac{1}{2}\, f(Y)^\flat\otimes X)-g((\ad^\mfa)^*(Y),(\ad^{\mfa})^*(Y))\\
&-g( X^\flat\otimes f^*(Y),X^\flat\otimes f^*(Y))-g( (f^*(Y))^\flat\otimes X ,(f^*(Y))^\flat\otimes X)\\
&+2\, g(H'(Y,\cdot,\cdot),H'(Y,\cdot,\cdot))+2\epsilon\, g(B(Y,\cdot),B(Y,\cdot))\\
&=4\, g((\ad^{\mfa}_Y)^S,(\ad^{\mfa}_Y)^S)+2 \epsilon\, g(f(Y),f(Y))-g((\ad^\mfa)^*(Y),(\ad^{\mfa})^*(Y))\\
&-2\epsilon\, g(f^*(Y),f^*(Y))+2\, g(H'(Y,\cdot,\cdot),H'(Y,\cdot,\cdot))+2\epsilon\, g(B(Y,\cdot),B(Y,\cdot))\\
&=4\, g((\ad^{\mfa}_Y)^S,(\ad^{\mfa}_Y)^S)+2\epsilon\, g(f^*f(Y),Y)-2\epsilon\, g(ff^*(Y),Y)\\
&-g((\ad^\mfa)^*(Y),(\ad^{\mfa})^*(Y))+2\, g(H'(Y,\cdot,\cdot),H'(Y,\cdot,\cdot))+2\epsilon\, g(B(Y,\cdot),B(Y,\cdot))\\
&=4\,g((\ad_Y^\mfa)^S,(\ad_Y^{\mfa})^S)-g((\ad^\mfa)^*(Y),(\ad^{\mfa})^*(Y))+2\epsilon\, g([f^*,f](Y),Y)\\
&+ 2\, g(H'(Y,\cdot,\cdot),H'(Y,\cdot,\cdot))+2\epsilon\, g(B(Y,\cdot),B(Y,\cdot)) \, ,
\end{split}
\end{equation*}
i.e. the third equation in Proposition \ref{pro:codim1idealnondeg}. Next, inserting $X$ and $Y\in \mfa$ into the second equation in \eqref{eq:gEdelta=0} one arrives, due to $\ad^*(X)=0$, directly at the fourth equation in Proposition \ref{pro:codim1idealnondeg}.

Finally, let $Y,Z\in \mfa$. Noting that
\begin{equation*}
H_Y(W)=H(Y,W,\cdot)^{\sharp}=H'(Y,W,\cdot)^{\sharp}+B(Y,W)\, X=(H')_Y(W)+(B(Y,\cdot)\otimes X)(W) \, ,
\end{equation*}
for all $W\in \mfa$ and $H_Y(X)=-H(X,Y,\cdot)^{\sharp}=-\epsilon \, B(Y,\cdot)^{\sharp}$, we see that
\begin{equation*}
H_Y=\begin{pmatrix} H'_Y & -\epsilon \,B(Y,\cdot)^{\sharp} \\
	                  B(Y,\cdot)   & 0
	\end{pmatrix} \, ,
\end{equation*}
and so inserting $Y$ and $Z$ into the second equation in \eqref{eq:gEdelta=0}, we obtain
\begin{equation*}
\begin{split}	
0&=g(H'_Y,(\ad^\mfa)^*(Z))-\epsilon\,g(X^\flat\otimes B(Y,\cdot)^\sharp ,X^\flat\otimes f^*(Z))\\
&-\frac{1}{\epsilon}g(B(Y,\cdot)\otimes X, f^*(Z)^\flat\otimes X)-g(H'_Z,(\ad^\mfa)^*(Y))\\
&+\epsilon\,g(X^\flat\otimes B(Z,\cdot)^\sharp ,X^\flat\otimes f^*(Y))+\frac{1}{\epsilon}g(B(Z,\cdot)\otimes X, f^*(Y)^\flat\otimes X)\\
&=g(H'_Y,(\ad^\mfa)^*(Z))-2g(f^*(Y),B(Z,\cdot)^\sharp)\\
&-g(H'_Z,(\ad^\mfa)^*(Y))+2 g(f^*(Z),B(Y,\cdot)^\sharp)) \, ,
\end{split}
\end{equation*}
which is the final equation in Proposition \ref{pro:codim1idealnondeg} and so proves the statement.
\end{proof}
Proposition \ref{pro:codim1idealnondeg} has the following direct consequence:
\begin{corollary}\label{co:codim1fskewB=0}
Let $(\mfg,H,\cG_g)$ be a generalised pseudo-Riemannian Lie algebra admitting a non-degenerate codimension one ideal $\mfa$.
Then:
\begin{enumerate}[(a)]
	\item If $f$ is skew-symmetric and $B=0$, then $(\mfg,H,\cG_g,0)$ is generalised Einstein if and only if $(\mfa,H',\cG_{g|_{\mfa}},0)$ is generalised Einstein.
    \item If $\mfa$ is definite, then $(\mfg,H,\cG_g,0)$ is generalised Einstein if and only if $f^S=0$, $B=0$ and $(\mfa,H',\cG_{g|_{\mfa}},0)$ is generalised Einstein.
\end{enumerate}
\end{corollary}
The last corollary then implies:
\begin{corollary}\label{co:notAbeliannotdefinite}
Let $(\mfg,H,\cG_g,0)$ be an almost nilpotent generalised Einstein pseudo-Riemannian Lie algebra with definite codimension one nilpotent ideal $\mfn$. Then $\mfg$ is almost Abelian.
\end{corollary}
\begin{proof}
Since $\mfn$ is definite, Corollary \ref{co:codim1fskewB=0}(b) implies that $(\mfn,H',\cG_{g|_{\mfn}},0)$ is generalised Einstein. Hence, by Corollary \ref{co:Riemnil}, $\mfn$ is Abelian and so $\mfg$ almost Abelian.
\end{proof}
We note that in the almost Abelian case, Proposition \ref{pro:codim1idealnondeg} simplifies as follows:
\begin{corollary}\label{co:almostAbelian}
	Let $(\mfg,H,\cG_g)$ be an almost Abelian generalised pseudo-Riemannian Lie algebra with non-degenerate codimension one Abelian ideal $\mfn$. Let $f$, $H'$ and $B$ be defined as in the Notation above. Then $(\mfg,H,\cG_g,0)$ is generalised Einstein if and only if
\begin{equation}\label{eq:almostAbelian}
	\begin{split}
		0&=2\, g(f^S,f^S)+ g(B,B) \, ,\\
		0&= g(B,H'(Y,\cdot,\cdot)) \, ,\\
		g((f^*f-ff^*)(Y),Y)&=-\epsilon\, g(H'(Y,\cdot,\cdot),H'(Y,\cdot,\cdot))- g(B(Y,\cdot),B(Y,\cdot)) \, ,\\
	 g(f^*(Y),B(Z,\cdot)^{\sharp})&=g(f^*(Z),B(Y,\cdot)^{\sharp}) \, ,
	\end{split}
\end{equation}
	for all $Y,Z\in \mfn$.
\end{corollary}
We note the following consequence of Corollary \ref{co:almostAbelian} in the four-dimensional case:
\begin{corollary}\label{co:almostAbelian4d}
Let $(\mfg,H,\cG_g,\delta=0)$ be a four-dimensional almost Abelian generalised Einstein Lie algebra with non-degenerate codimension one Abelian ideal $\mfn$. Then $H'=H\vert_\mfn=0$.
\end{corollary}
\begin{proof}
We may choose an orthonormal basis $(e_1,e_2,e_3)$ of $\mfn$ and then write
\begin{equation*}
B=b_1 e^{23}+b_2 e^{31}+b_3 e^{12} \, ,\quad H'=h e^{123} \, ,
\end{equation*}
for certain $b_1,b_2,b_3,h\in \bR$. Then the second equation in \eqref{eq:almostAbelian} yields
\begin{equation*}
g(B,H'(e_i,\cdot,\cdot))=g(B, h e^{i+1\, i+2})= b_i h\, g(e^{i+1\, i+2},e^{i+1\, i+2}) \, ,
\end{equation*}
for $i=1,2,3$, where we compute the upper indices modulo three. Since $e^{i+1\, i+2}$ is not null, we have
\begin{equation*}
h b_1=0 \, ,\ h b_2=0 \, ,\ h b_3=0 \, ,
\end{equation*}
and so either $h=0$, i.e. $H'=0$, or $b_1=b_2=b_3=0$, i.e. $B=0$.

So let us assume that $B=0$ and show that then, necessarily, also $H'=0$. For this, we note first that
\begin{equation*}
g([f^*,f](Y),Y)=g(f(Y),f(Y))-g(f^*(Y),f^*(Y)) \, ,
\end{equation*}
for all $Y\in \mfn$. Hence, if $(e_1,e_2,e_3)$ is an orthonormal basis of $\mfn$ with $g(e_i,e_i)=:\epsilon_i\,$, then
\begin{equation*}
\epsilon_i\, g(H'(e_i,\cdot,\cdot),H'(e_i,\cdot,\cdot))=\epsilon_i \, h^2 g(e^{i+1\, i+2},e^{i+1\, i+2})=\epsilon_i\, \epsilon_{i+1}\,\epsilon_{i+2}\, h^2=\epsilon_1\,\epsilon_2\,\epsilon_3\, h^2 \, .
\end{equation*}
Thus, the third equation in \eqref{eq:almostAbelian} yields
\begin{equation*}
\begin{split}
0&=g(f,f)-g(f^*,f^*)=\sum_{i=1}^3 \epsilon_i\, \left(g(f(e_i),f(e_i))-g(f^*(e_i),f^*(e_i))\right)\\
&=
\sum_{i=1}^3 \epsilon_i\, g([f^*,f](e_i),e_i)=-\epsilon\, \sum_{i=1}^3 \epsilon_i\, g(H'(e_i,\cdot,\cdot),H'(e_i,\cdot,\cdot))=-\epsilon\,\sum_{i=1}^3 \epsilon_1\, \epsilon_2\,\epsilon_3\, h^2\\
&=-3\epsilon\,\epsilon_1\,\epsilon_2\, \epsilon_3\, h^2 \, .
\end{split}
\end{equation*}
Hence, $h=0$, i.e. $H'=0$.
\end{proof}
Corollary \ref{co:almostAbelian4d} implies that in four dimensions, the condition for an almost Abelian generalised pseudo-Riemannian Lie algebra $(\mfg,H,\cG_g)$ with non-degenerate codimension one Abelian ideal $\mfn$ to be generalised Einstein for zero divergence depends only on $H$ and the metric on $\mfn$:
\begin{corollary}\label{co:almostAbelian4dmetriconmfn}
Let $\mfg$ be a four-dimensional almost Abelian Lie algebra with codimension one Abelian ideal and let $g_1\,$, $g_2$ be two pseudo-Riemannian metrics on $\mfg$ with $g_1|_{\mfn}=g_2|_{\mfn}$. If $g_1$---and so also $g_2$---is non-degenerate on $\mfn$, then $(\cG_{g_1},H,\delta=0)$ is generalised Einstein if and only if $(\cG_{g_2},H,\delta=0)$ is generalised Einstein.
\end{corollary}
\begin{proof}
First of all, note that if $X_i$ is orthogonal with respect to $g_i$ to $\mfn$ with $\epsilon_i:=g(X_i,X_i)\in \{-1,1\}$ for $i=1,2$, then $X_2=\lambda\, X_1+Y$ for $\lambda\in \bR^*$ and $Y\in \mfn$. Thus, $f_2=\ad_{X_2}|_{\mfn}=\lambda\,\ad_{X_1}|_{\mfn}=\lambda\, f_1\,$. Moreover, 
\begin{equation*}
B_2=\epsilon_2\, H(X_2,\cdot,\cdot)|_{\Lambda^2 \mfn}=\epsilon_2\, \lambda\, H(X_1,\cdot,\cdot)|_{\Lambda^2 \mfn}+\epsilon_2\, H(Y,\cdot,\cdot)|_{\Lambda^2 \mfn}=\epsilon_1\,\epsilon_2\, \lambda B_1 \, ,
\end{equation*} 
due to $H_1'=H_2'=H|_{\Lambda^3 \mfn}=0$ by Corollary \ref{co:almostAbelian4d}. Now, setting $g:=g_1|_{\mfn}\,$, \eqref{eq:almostAbelian} shows that $(\cG_{g_i},H,\delta=0)$ is generalised Einstein if and only if
\begin{equation*}
\begin{split}
0&=2\, g(f_i^S,f_i^S)+g(B_i,B_i) \, ,\\
g([f_i^*,f_i](Y),Y)&=-g(B_i(Y,\cdot),B_i(Y,\cdot)) \, ,\\
g(f_i^*(Y),B_i(Z,\cdot)^{\sharp})&=g(f_i^*(Z),B_i(ZY,\cdot)^{\sharp}) \, ,
\end{split}
\end{equation*}
for all $Y,Z\in \mfn$, which yields that $(\cG_{g_1},H,\delta=0)$ is generalised Einstein if and only if $(\cG_{g_2},H,\delta=0)$ is generalised Einstein.
\end{proof}
\subsection{Almost Abelian Lorentzian case with $H=0$ and non-degenerate $\mfn$}
Here, we look at the almost Abelian case with non-degenerate $\mfn$ and $H=0$. Moreover, in the entire subsection, we let $X\in \mfg$ be orthogonal to $\mfn$ with $g(X,X)\in \{-1,1\}$ and let $f:=\ad_X|_{\mfn}\in \End(\mfn)$.

First of all, we observe that Corollary \ref{co:almostAbelian} implies the following characterisation of the generalised Einstein condition for $\delta=0$ in our situation:
\begin{corollary}\label{co:almostAbelianH=0}
	Let $(\mfg,H=0,\cG_g)$ be an almost Abelian generalised pseudo-Riemannian Lie algebra with non-degenerate codimension one Abelian ideal $\mfn$. Then:
	\begin{enumerate}[(a)]
		\item 
		$(\mfg,H=0,\cG_g,\delta=0)$ is generalised Einstein if and only if $f$ is normal and $\tr((f^S)^2)=0$.
		\item
		If $g'$ is another pseudo-Riemannian metric on $\mfg$ with $g|_{\mfn}=g'|_{\mfn}$, then $(\mfg,H=0,\cG_g,\delta=0)$ is generalised Einstein if and only if $(\mfg,H=0,\cG_{g'},\delta=0)$ is generalised Einstein.  
	\end{enumerate}
\end{corollary}
\begin{proof}
	First of all, note that $g(f^S,f^S)=\tr((f^S)^2)$. Indeed, taking $(e_1,\ldots,e_{n-1})$ an orthonormal basis of $\mfn$ with $g(e_i,e_i)=\epsilon_i\,$, we calculate
	\begin{equation*}
		\tr((f^S)^2)=\sum_{i=1}^{n-1} \epsilon_i\, g((f^S)^2(e_i),e_i)=\sum_{i=1}^{n-1} \epsilon_i\, g(f^S(e_i),f^S(e_i))=g(f^S,f^S) \, .
	\end{equation*}
	Then part (a) follows directly from Corollary \ref{co:almostAbelian}. Moreover, part (b) follows from (a) since if $X'$ is chosen orthogonally to $\mfn$ with respect to $g'$ with $g'(X',X')\in \{-1,1\}$ and $f':=\ad(X')|_{\mfn}\,$, then
	$f'=\lambda\, f$ for some $\lambda\in \bR^*$.
\end{proof}
\begin{remark}\label{re:fSdiagonalisable}
	\begin{itemize}[wide]
		\item[(i)]
		Note that if $g$ is definite on $\mfn$ or, more generally, if $f^S$ is diagonalisable, then $\tr((f^S)^2)=0$ if and only if $f^S=0$. Hence, in this case, $(\mfg,0,\mathcal{G}_g,0)$ is generalised Einstein if and only if $f$ is skew-symmetric, which implies that $g$ is flat by the computations in \cite{Mi}. 
		\item[(ii)] We note that the trace form $\tau\in \mfg^*$, $\tau(W)=\tr(\ad_W)$ for $W\in \mfg$ is given in our case by $\tau=\epsilon\tr(f)X^\flat$. Moreover, by
		 \cite[Corollary 2.30]{CK}, $(\mfg,0, \mathcal{G}_g,0)$ is generalised Einstein if and only if $g$ is a Ricci soliton satisfying
		\begin{equation*}
		\mathrm{Ric}^g+\nabla \tau=0 \, .
		\end{equation*}
         Now a short computation using the Koszul formula yields that the Levi-Civita connection $\nabla$ of $g$ is given by
         \begin{equation*}
         	\nabla_X Z=f^A(Z) \, ,\quad \nabla_X X=0 \, ,\quad \nabla_Z \tilde{Z}=\epsilon\, g(f^S(Z),\tilde{Z}) X \, ,\quad \nabla_Z X= -f^S(Z) \, ,
         \end{equation*}
         for all $Z,\tilde{Z}\in \mfn$. Hence, $\nabla \tau=-\epsilon\tr(f)\, g(f^S(\cdot),\cdot)$ so that the conditions $\tr((f^S)^2)=0$ and $f$ being normal are equivalent to
         \begin{equation*}
        \mathrm{Ric}^g=\epsilon\tr(f)\, g(f^S(\cdot),\cdot) \, .
         \end{equation*}
\end{itemize}
\end{remark}

Using the canonical forms in Lemma \ref{le:canonicalforms}, we are now able to describe all possible almost Abelian generalised Einstein Lorentzian Lie algebras with $H=0$, $\delta=0$ and with codimension one non-degenerate Abelian ideal. In order to distinguish the cases where the pseudo-Riemannian metric is flat, we need a classification of all flat pseudo-Riemannian metrics on almost Abelian Lie algebras with non-degenerate codimension one Abelian ideal:
\begin{proposition}\label{pro:almostAbelianflat}
Let $\mfg$ be an almost Abelian Lie algebra and $g$ be a pseudo-Riemannian metric on $\mfg$ such that a codimension one Abelian ideal $\mfn$ is non-degenerate. Then $g$ is flat if and only if there exists some $Y\in \mfn$ with $g(Y,Y)=0$ and some $\delta\in \{-1,1\}$ such that
\begin{equation*}
f=h+\delta\, Y^\flat\otimes Y \, ,
\end{equation*}
where $h$ is a $g$-anti-symmetric endomorphism of $\mfn$ satisfying $h(Y)=0$ and $f=\ad_X\vert_\mfn$ as explained above.
\end{proposition}
\begin{proof}
First of all, recall from Remark \ref{re:fSdiagonalisable} that the Levi-Civita connection $\nabla$ is given by 
    \begin{equation*}
	\nabla_X Z=f^A(Z) \, ,\quad \nabla_X X=0 \, ,\quad \nabla_Z \tilde{Z}=\epsilon\, g(f^S(Z),\tilde{Z}) X \, ,\quad \nabla_Z X= -f^S(Z) \, ,
\end{equation*}
for all $Z,\tilde{Z}\in \mfn$, where $\epsilon:=g(X,X)\in \{-1,1\}$. The condition for $g$ to be flat is equivalent to $[\nabla_{U},\nabla_V]=\nabla_{[U,V]}$ for all $U,V\in \mfg$. Now, we see that
\begin{equation*}
	[\nabla_{Z_1},\nabla_{Z_2}](X)=0=\nabla_{[Z_1,Z_2]}(X) \, ,
\end{equation*}
for all $Z_1,Z_2\in \mfn$. Hence, $[\nabla_{Z_1},\nabla_{Z_2}]=\nabla_{[Z_1,Z_2]}=0$ for all $Z_1,Z_2\in \mfn$ if and only if
\begin{equation*}
	[\nabla_{Z_1},\nabla_{Z_2}](W)=-\epsilon\, \left(g(f^S(Z_2),W) f^S(Z_1)- g(f^S(Z_1),W) f^S(Z_2)\right)=0 \, ,
\end{equation*}
for all $Z_1,Z_2,W\in \mfn$, i.e.  if and only if
\begin{equation*}
	f^S(Z_2)^\flat\otimes f^S(Z_1)=f^S(Z_1)^\flat\otimes f^S(Z_2) \, ,
\end{equation*}
for all $Z_1,Z_2\in \mfn$. We know from \cite{Mi} that if $f^S=0$, i.e. if $f$ is skew-symmetric, then $\mfg$ is flat. So let us assume that $f^S\neq 0$. The above equation then implies $\dim(\mathrm{im}(f^S))=1$. Hence, $f^S=\delta\, Y^\flat\otimes Y$ for some non-zero $Y\in \mfn$ and some $\delta\in \{-1,1\}$.

Then, the condition $[\nabla_X,\nabla_Z]=\nabla_{[X,Z]}=\nabla_{f(Z)}$ applied to $X$ yields
\begin{equation}\label{eq:flatcond}
		\delta\, g(Y,Z)f^A(Y)=-[\nabla_X,\nabla_Z]X=-\nabla_{f(Z)}X=\delta\, g(Y,f(Z))Y \, ,
\end{equation}
for all $Z\in \mfn$, showing that $f^A(Y)=\lambda\, Y$ for some $\lambda\in \bR$. Thus,
\begin{equation*}
\lambda g(Y,Y)=g(f^A(Y),Y)=-g(Y,f^A(Y))=-\lambda g(Y,Y) \, ,
\end{equation*}
which shows that $\lambda=0$ or $Y$ is null. We show that, in fact, both conditions have to hold. For this, assume first that $\lambda=0$, i.e. $f^A(Y)=0$. Then Equation \eqref{eq:flatcond} implies $0=g(Y,f(Z))=g(Y,f^S(Z))=\delta\, g(Y, g(Y,Z) Y)=\delta\, g(Y,Y)\, g(Y,Z)$ for all $Z\in\mfn$ and so $g(Y,Y)=0$, i.e. $Y$ is null. Next, assume that $Y$ is null. Then Equation \eqref{eq:flatcond} yields
\begin{equation*}
\begin{split}
		\delta \lambda\, g(Y,Z) Y&=\delta\, g(Y,Z)f^A(Y)=\delta\, g(Y,f(Z))Y=\delta\, g(Y,f^A(Z)+\delta g(Y,Z)Y)Y\\
		&=\delta\, g(Y,f^A(Z))Y=-\delta\, g(f^A(Y),Z) Y=-\delta \lambda\, g(Y,Z) Y \, ,
\end{split}
\end{equation*}
and so also here $\lambda=0$, i.e. $f^A(Y)=0$. Now, assuming that $f^A(Y)=0$ and that $Y$ is null, a short computation shows that also $[\nabla_X,\nabla_Z]\tilde{Z}=\nabla_{f(Z)}\tilde{Z}$ holds and so $g$ is flat in that case. This proves the statement.
\end{proof}
We can now prove the main result of this subsection:
\begin{theorem}\label{th:almostAbeliannondegH=0}
Let $(\mfg,H=0,\mathcal{G}_g,\delta=0)$ be an $n$-dimensional, $n\geq 3$, almost Abelian generalised Lorentzian Lie algebra with \emph{non-degenerate} codimension one Abelian ideal $\mfn$. Then:
\begin{itemize}
	\item[(a)] $(\mfg,H=0,\mathcal{G}_g,\delta=0)$ is generalised Einstein with positive definite $\mfn$ if and only if there exists some orthonormal basis $(e_1,\ldots,e_{n-1})$ and some $k\in \left\{0,\ldots,\left\lfloor \tfrac{n-1}{2}\right\rfloor\right\}$ and $a_1\geq \ldots\geq a_k>0$ such that
	\begin{equation*}
	f=\diag(L_1(0,a_1),\ldots,L_1(0,a_k),0,\ldots,0) \, ,
	\end{equation*}
	with respect to $(e_1,\ldots,e_{n-1})$. In this case, $g$ is flat.
	\item[(b)] $(\mfg,H=0,\mathcal{G}_g,\delta=0)$ is generalised Einstein with Lorentzian $\mfn$ if and only if there exists some orthonormal basis $(e_1,\ldots,e_{n-1})$ of $\mfn$ with $g(e_1,e_1)=-1$ such that with respect to that basis, $f$ equals one of the following matrices:
	\begin{itemize}
	\item[(i)]
	$n=3$ and
	\begin{equation*}
	f=\begin{pmatrix} 0 & \rho \\ \rho & 0 \end{pmatrix} \, ,
	\end{equation*}
	for some $\rho\in \bR$,
	\item[(ii)]
	or $n$ is even and
		\begin{equation*}
	f=\diag(M_3(\sigma),L_1(0,a_1),\ldots,L_1(0,a_k),0,\ldots,0) \, ,
	\end{equation*}
	for some $\sigma\in \bR$, some $k\in \{0,\ldots,\tfrac{n-4}{2}\}$ and certain $a_1\geq \ldots\geq a_k>0$,
	\item[(iii)]
	or $n\geq 5$ is odd and
    \begin{equation*}
    	f=\diag(M_4(\sigma,\tau),L_1(0,a_1),\ldots,L_1(0,a_k),0,\ldots,0) \, .
    \end{equation*}
    for some $\sigma\in \bR$, $\tau\geq 0$, some $k\in \{0,\ldots,\tfrac{n-5}{2}\}$ and certain $a_1\geq \ldots \geq a_k>0$,
    \item[(iv)]
    or
    \begin{equation*}
    f=\diag(L_1(\alpha,\beta),L_1(a_1,b_1),\ldots,L_1(a_k,b_k),c_1,\ldots,c_s) \, ,
    \end{equation*}
    for some $\alpha\in \bR$, some $k\in \left\{0,\ldots,\left\lfloor \tfrac{n-3}{2}\right\rfloor\right\}$
and $s\in \left\{0,\ldots,n-3\right\}$
with $2k+s=n-3$, and certain $a_1,\ldots,a_k\in \bR$, $b_1\geq \ldots\geq b_k>0$, $c_1\geq\ldots\geq c_s$
with $\beta=+\sqrt{\alpha^2+\sum_{i=1}^k a_i^2+\sum_{j=1}^s \tfrac{c_j^2}{2}}\,$,
\item[(v)]
or
\begin{equation*}
f=\begin{pmatrix} \tfrac{\epsilon}{2} & \tfrac{\epsilon}{2} & v^t \\ -\tfrac{\epsilon}{2} & -\tfrac{\epsilon}{2} & -v^t \\ v & v & \diag(L_1(0,a_1),\ldots,L_1(0,a_k),0,\ldots,0)
     \end{pmatrix} \, ,
 \end{equation*}
  for $\epsilon \in \{-1,1\}$, some $v\in \bR^{n-3}$, some $k\in \left\{0,\ldots,\left\lfloor \tfrac{n-3}{2}\right\rfloor\right\}$ and certain $a_1\geq \ldots\geq a_k>0$,
\item[(vi)]
or $n\geq 4$ and
\begin{equation*}
	f=\begin{pmatrix} 0 & -\tfrac{1}{\sqrt{2}} & 0 & v^t \\
		\tfrac{1}{\sqrt{2}} & 0 & \tfrac{1}{\sqrt{2}} &0\\
		0 & \tfrac{1}{\sqrt{2}} & 0 & -v^t \\
		v & 0 & v & \diag(L_1(0,a_1),\ldots,L_1(0,a_k),0,\ldots,0)
	\end{pmatrix}  \, ,            
\end{equation*}
  for some $v\in \bR^{n-4}$, some $k\in \left\{0,\ldots,\left\lfloor \tfrac{n-4}{2}\right\rfloor\right\}$ and certain $a_1\geq \ldots\geq a_k>0$.
	\end{itemize}
Moreover, $g$ is flat if and only if $f$ is as in the cases (i), (ii), (iii) or (v).
\end{itemize}
\end{theorem}
\begin{proof}
\begin{itemize}[wide]
	\item[(a)]
	By Remark \ref{re:fSdiagonalisable}, $f$ is a skew-symmetric endomorphism on the positive definite subspace $\mfn$ and so of the claimed form.
	\item[(b)]
By Corollary \ref{co:almostAbelianH=0},  $(\mfg,H=0,\mathcal{G}_g,\delta=0)$ is generalised Einstein if and only if $\tr((f^S)^2)=0$ and $f$ is normal, the latter condition being equivalent to $[f^S,f^A]=0$. Now we distinguish different cases according to the type of $f^S$ and use the associated orthonormal basis $(e_1,\ldots,e_{n-1})$ of $\mfn$ with $g(e_1,e_1)=-1$ with respect to which $f^S$ has the corresponding normal form:
\begin{itemize}[$\bullet$,wide]
	\item If $f^S$ is of first type, then Lemma \ref{le:trfsquare=0} implies $f^S=0$, i.e. that $f$ is skew-symmetric. Hence, by Lemma \ref{le:canonicalforms} (b), there exists a (maybe different) orthonormal basis $(e_1,\ldots,e_{n-1})$ of $\mfn$ such that $f$ is as in case (i), (ii) or (iii).
   \item
   If $f^S$ is of the second type, i.e. $f^S=\diag(L_1(\alpha,\beta),b_1,\ldots,b_{n-3})$, then the condition $\beta\neq 0$ and $[f^S,f^A]=0$ imply that $f^A$ has to preserve both $U_1:=\spa{e_1,e_2}$ and 
$U_2:=\spa{e_3,\ldots,e_{n-1}}$. Since $f|_{U_2}$ is normal and $U_2$ is positive definite, $f|_{U_2}$ is diagonalisable over the complex numbers. Consequently, there is an 
orthonormal basis of $U_2\,$, which we name again $(e_3,\ldots,e_{n-1})$, some $k\in \left\{0,\ldots,\left\lfloor \tfrac{n-3}{2}\right\rfloor\right\}$, some $s\in \left\{0,\ldots,n-3\right\}$
with $2k+s=n-3$ and $a_1,\ldots,a_k\in \bR$, $b_1\geq \ldots\geq b_k>0$, $c_1\geq\ldots\geq c_s$ such that
\begin{equation*}
f|_{U_2}=\diag(L_1(a_1,b_1),\ldots, L_1(a_k,b_k),c_1,\ldots,c_s) \, .
\end{equation*}
Now we have $f^A|_{U_1}=\left(\begin{smallmatrix} 0 & \rho \\ \rho & 0\end{smallmatrix}\right)$ for some $\rho\in \bR$ and so the condition $[f^A,f^S]=0$ on $U_1$ and $f^S|_{U_1}=L_1(\alpha,\beta)$ with $\beta\neq 0$ forces $\rho=0$. Thus, $f|_{U_1}=f^S|_{U_1}$ and
\begin{equation*}
\begin{split}
f&=\diag(L_1(\alpha,\beta),L_1(a_1,b_1),\ldots,L_1(a_k,b_k),c_1,\ldots,c_s) \, ,\\
f^S&=\diag(L_1(\alpha,\beta),a_1,a_1,\ldots,a_k,a_k,c_1,\ldots,c_s) \, .
\end{split}
\end{equation*}
Now, by Lemma \ref{le:trfsquare=0}, the condition $\tr((f^S)^2)=0$ is equivalent to $\beta^2=\alpha^2+\sum_{i=1}^k a_i^2+\sum_{j=1}^s \tfrac{c_j^2}{2}$. 

Next, we show that $g$ is never flat. For this, assume by contradiction that $g$ was flat. Then, by Proposition \ref{pro:almostAbelianflat}, $f^S$ has to be of the form $f^S=\delta\, v^\flat\otimes v$ for some null vector $v\in \spa{e_1,\ldots,e_{n-1}}$ and some $\delta\in \{-1,1\}$. Consequently, $(f^S)^2=0$, and so $L_1(\alpha,\beta)^2=0$. However, the latter equality implies $\alpha=\beta=0$, which contradicts the assumption $\beta>0$. Hence, $g$ is never flat in this case.
\item
If $f^S$ is of the third type, Lemma \ref{le:trfsquare=0} implies $f^S=\diag(L_2(0,\epsilon),0,\ldots,0)$ with respect to the basis $(e_1,\ldots,e_{n-1})$. Since $f^A$ commutes with $f^S$, $f^A$ has to preserve both $\ker(f^S)=\spa{e_1-e_2,e_3,\ldots,e_{n-1}}$ and $\mathrm{im}(f^S)=\spa{e_1-e_2}$. 
Consequently,
$f^A(e_1)=\rho\, e_2 + v$ and $f^A(e_2)=\rho\, e_1+v$ for $\rho\in \bR$ and some $v\in U:=\spa{e_3,\ldots,e_{n-1}}$. As a result,
\begin{align*}
\epsilon\tfrac{\rho}{2}(e_1-e_2)&=f^S(\rho\, e_2+v)=f^S(f^A(e_1))\\
&=f^A(f^S(e_1))=\epsilon\tfrac{1}{2} f^A(e_1-e_2)=-\epsilon\tfrac{\rho}{2}(e_1-e_2) \, ,
\end{align*}
i.e. $\rho=0$. Thus,
\begin{equation*}
f^A=\begin{pmatrix} 0 & 0 & v^t \\ 0 & 0 & -v^t \\ v & v & \tilde{f}^A
     \end{pmatrix} \, ,
\end{equation*}
with respect to $\spa{e_1}\oplus \spa{e_2}\oplus U$ for some anti-symmetric $\tilde{f}^A \in \End(U)$. Since $U$ is positive definite, we may choose an appropriate orthonormal basis such that
\begin{equation*}
\tilde{f}^A=\diag(L_1(0,a_1),\ldots,L_1(0,a_k),0,\ldots,0) \, ,
\end{equation*}
for some $k\in \left\{0,\ldots,\left\lfloor \tfrac{n-3}{2}\right\rfloor\right\}$ and certain $a_1\geq \ldots\geq a_k>0$. Thus,
\begin{equation*}
f=\begin{pmatrix} \tfrac{\epsilon}{2} & \tfrac{\epsilon}{2} & v^t \\ -\tfrac{\epsilon}{2} & -\tfrac{\epsilon}{2} & -v^t \\ v & v & \diag(L_1(0,a_1),\ldots,L_1(0,a_k),0,\ldots,0)
     \end{pmatrix} \, ,
\end{equation*}
As $f^S=-\epsilon\, u^\flat\otimes u$ for the null vector $u:=\frac{1}{\sqrt{2}}\left(e_1-e_2\right)$, and $f^A(u)=0$, $g$ is flat by Proposition \ref{pro:almostAbelianflat}.
\item
If $f^S$ is of the fourth type, Lemma \ref{le:trfsquare=0} implies $f^S=\diag(L_3(0),0,\ldots,0)$ with respect to the basis $(e_1,\ldots,e_{n-1})$. Here, we set $U:=\spa{e_4,\ldots,e_{n-1}}$ and
write
\begin{equation*}
f^A=\begin{pmatrix} 0 & \rho_1 & \rho_2 & v^t \\
                    \rho_1 & 0 & -\rho_3 & -w^t \\
                    \rho_2 & \rho_3 & 0 & -z^t \\
                    v & w & z & \tilde{f}^A
    \end{pmatrix}       \, ,       
\end{equation*}
for $\rho_1,\rho_2,\rho_3\in \bR$, $v,w,z\in \bR^{n-4}$ and $\tilde{f}^A$ being skew-symmetric on the positive definite subspace $U$. Now, arguing as before, $f^A$ has to preserve both $\mathrm{im}(f^S)=\spa{e_1-e_3,e_2}$ and $\ker(f^S)=\spa{e_1-e_3,e_4,\ldots,e_{n-1}}$, and so also $\mathrm{im}(f^S)\cap \ker(f^S) =\spa{e_1-e_3}$. Nevertheless, $f^A$ preserves $\spa{e_1-e_3}$ if and only if $z=v$ and $\rho_3=-\rho_1$. Note that then $f^A$ preserves $\ker(f^S)=\spa{e_1-e_3,e_4,\ldots,e_{n-1}}$ if and only if  $w=0$. Under these assumptions, $\mathrm{im}(f^S)=\spa{e_1-e_3,e_2}$ is also preserved. Now
\begin{equation*}
\begin{split}
-\tfrac{\rho_1}{\sqrt{2}}e_1+\tfrac{\rho_2}{\sqrt{2}}e_2+\tfrac{\rho_1}{\sqrt{2}}e_3&=f^S(\rho_1 e_2+\rho_2 e_3+v)=f^S(f^A(e_1))=f^A(f^S(e_1))\\
&=f^A(\tfrac{1}{\sqrt{2}}e_2)=\tfrac{\rho_1}{\sqrt{2}} (e_1-e_3),
\end{split}
\end{equation*}
and so $\rho_1=0$, $\rho_2=0$, which also gives $\rho_3=-\rho_1=0$. Hence, by choosing an appropriate orthonormal basis of $U$, we have
\begin{equation*}
	f=\begin{pmatrix} 0 & -\tfrac{1}{\sqrt{2}} & 0 & v^t \\
		\tfrac{1}{\sqrt{2}} & 0 & \tfrac{1}{\sqrt{2}} &0\\
		0 & \tfrac{1}{\sqrt{2}} & 0 & -v^t \\
		v & 0 & v & \diag(L_1(0,a_1),\ldots,L_1(0,a_k),0,\ldots,0)
	\end{pmatrix},            
\end{equation*}
where $k\in \left\{0,\ldots,,\left\lfloor\tfrac{n-4}{2}\right\rfloor\right\}$ and $a_1\geq \ldots\geq a_k\,$.
\end{itemize}
By the same reasoning as in the second case, $g$ being flat implies $(f^S)^2=0$. However, a direct computation shows
$(f^S)^2\neq 0$ and we conclude that $g$ is never flat in this case.
\end{itemize}
\end{proof}
We are able to obtain some results in the non-zero divergence case as well:
\begin{corollary}\label{co:almostAbeliandeltaneq0}
Let $(\mfg,H=0,\cG_g)$ be one of the generalised Lorentzian Lie algebras in Theorem \ref{th:almostAbeliannondegH=0}	(a), (b) (i), (b) (ii) or (b) (iii) with the corresponding endomorphism $f\in \End(\mfn)$ of the codimension one Abelian ideal $\mfn$. Moreover, let $\delta\in E^*$ be arbitrary. Then $(\mfg,H=0,\cG_g,\delta)$ is generalised Einstein if and only if $\delta(\im(f)\oplus \im(f)^\flat)=0$.
\end{corollary}
\begin{proof}
We note first that in the relevant cases of Theorem \ref{th:almostAbeliannondegH=0}, $f$ is a skew-symmetric endomorphism of the codimension one Abelian ideal. We now show that $\ad_Z^*(Z)\in \mfg'$ for all $Z\in \mfg=\mfn\oplus \bR\cdot X$, where the latter decomposition is as vector spaces. For this, write $Z=Y+\lambda X$ for some $Y\in \mfn$ and some $\lambda\in \bR$. Then, using that $\ad^*(X)=0$ and that $f$ is skew-symmetric, we get
\begin{equation*}
\ad_Z^*(Z)=\ad_Y^*(Y)+\lambda \ad_X^*(Y)=\ad_Y^*(Y)+\lambda f^*(Y)=\ad_Y^*(Y)-\lambda f(Y) \, .
\end{equation*}
However, $\ad_Y^*(Y)=0$ since $g(\ad_Y^*(Y),W)=g(Y,[Y,W])=0$ for all $W\in \mfn$ and
\begin{equation*}
g(\ad_Y^*(Y),X)=g(Y,\ad_Y(X))=-g(Y,f(Y))=0	 \, ,
\end{equation*}
since $f$ is skew-symmetric. Now $\mfg'=\mathrm{im}(f)$ and so $\ad_Z^*(Z)\in \mfg'$ for all $Z\in \mfg$. Hence, by Corollary \ref{co:gEdeltanonzero} (a), $(\mfg,H=0,\cG_g,\delta)$ is generalised Einstein if and only if
\begin{equation*}
 0=\delta(\mfg'\oplus (\mfg')^\flat)=\delta(\im(f)\oplus \im(f)^\flat) \, .
\end{equation*}
\end{proof}
\subsection{Almost Abelian Lorentzian case with $H\neq 0$ and non-degenerate $\mfn$}
We consider again almost Abelian generalised Einstein Lorentzian Lie algebras $(\mfg,H,\cG_g,0)$ with non-degenerate codimension one Abelian ideal $\mfn$. However, we now assume $H\neq 0$. We obtain a full classification in arbitrary dimensions under the assumption that $f^S$ is not of second type. Moreover, we show that in dimension four $f^S$ cannot be of second type and so we get a full classification without any extra assumptions in that case.

First of all, we show that in any dimension, any such example of a generalised Einstein Lie algebra must have Lorentzian $\mfn$ and $f^S$ cannot be diagonalisable:
\begin{lemma}\label{le:fSnotdiag}
Let $(\mfg,H,\cG_g,0)$ be an almost Abelian generalised Einstein Lorentzian Lie algebra such that the codimension one Abelian ideal $\mfn$ is non-degenerate and such that $H\neq 0$. Then $\mfn$ has Lorentzian signature and $f^S$ is not of first type.
\end{lemma}
\begin{proof}
We prove the Lemma by contradiction. Assume that either $\mfn$ has Riemannian signature or $\mfn$ has Lorentzian signature and $f^S$ is not of first type.

If $\mfn$ has Riemannian signature, then the first equation in \eqref{eq:almostAbelian} shows that $f^S=0$ and $B=0$. Hence, $f$ is skew-symmetric and so normal. Thus, the third equation in \eqref{eq:almostAbelian} implies
\begin{equation*}
g(H'(Y,\cdot,\cdot),H'(Y,\cdot,\cdot))=0 \, ,
\end{equation*}
for all $Y\in \mfn$, which gives $H'=0$ due to $g$ being Riemannian on $\mfn$. Hence, $H=0$ in contradiction to our assumptions.

Next, assume that $\mfn$ has Lorentzian signature but $f^S$ is of first type, i.e. $f^S=\diag(a_1,\ldots,a_{n-1})$ for certain $a_1,\ldots,a_{n-1}\in\bR$ with respect to an orthonormal basis $(e_1,\ldots,e_{n-1})$ of $\mfn$ with $g(e_1,e_1)=-1$. Then $\epsilon=g(X,X)=1$ and $e_1\hook H'$ and $e_1\hook B$ are both forms on the positive definite subspace $V:=\spa{e_2,\ldots,e_{n-1}}$. Since
$[f^*,f]=-2 [f^A,f^S]$, we have
\begin{equation*}
\begin{split}
g((f^*f-ff^*)(e_i),e_i)&=-2 g(f^A(f^S(e_i)),e_i)+2 g(f^S(f^A(e_i)),e_i)\\
&=4 g(f^S(e_i),f^A(e_i))=4 a_i\, g(e_i,f^A(e_i))=0 \, ,
\end{split}
\end{equation*}
for all $i=1,\ldots,n-1$. Inserting $e_1$ into the third equation in \eqref{eq:almostAbelian} yields
\begin{equation*}
g(H'(e_1,\cdot,\cdot),H'(e_1,\cdot,\cdot))+g(B(e_1,\cdot),B(e_1,\cdot))=0 \, ,
\end{equation*}
and so, since $V$ was positive definite, that $H'(e_1,\cdot,\cdot)=0$ and $B(e_1,\cdot)=0$. However, then $H'$ and $B$ themselves are forms on the positive definite subspace $V$. Hence, inserting now $e_i$ for $i=2,\ldots,n-1$ into the third equation in \eqref{eq:almostAbelian} gives
\begin{equation*}
	g(H'(e_i,\cdot,\cdot),H'(e_i,\cdot,\cdot))+g(B(e_i,\cdot),B(e_i,\cdot))=0 \, ,
\end{equation*}
and so $H'(e_i,\cdot,\cdot)=0$ and $B(e_i,\cdot)=0$ for all $i=2,\ldots,n-1$, which means that $H=H'+B=0$, a contradiction.
\end{proof}
Let us briefly comment on the case where $\mfg$ is Abelian. In this situation, $f^S$ is of first type and so Lemma \ref{le:fSnotdiag} implies that any generalised Einstein Lie algebra of the form $(\bR^n,H,\cG_g,0)$ with Lorentzian $g$ must satisfy $H=0$. This turns out to be still true when $n=4$ and $g$ has split signature:
\begin{corollary}\label{co:AbelianH=0}
	Let $(H,\mathcal{G}_g)$ be a generalised Einstein metric on the Abelian Lie algebra $\mfg=\bR^n$ with divergence operator $\delta=0$. If either $g$ has Lorentzian signature or the dimension $n$ is less or equal to four, we must have $H=0$.
\end{corollary}
\begin{proof}
The case where $g$ is Lorentzian is clear by our discussion above. Up to an overall sign, a metric in dimension 1, 2 or 3 is either Riemannian or Lorentzian, and the Riemannian case was studied in Corollary \ref{co:Riemnil}. Therefore, it suffices to consider the case $n=4$ with split signature. Choose some non-degenerate codimension one subspace $\mfn$ in $\mfg$. Then, $\mfn$ is an Abelian ideal and writing, as before, $H=H'+X^\flat\wedge B$ for $H'\in \Lambda^3 \mfn^*$, $B\in \Lambda^2 \mfn^*$ with $X$ orthogonal to $\mfn$ and $g(X,X)\in \{-1,1\}$, Corollary \ref{co:almostAbelian4d} yields that $H'=0$. Then, the third equation in \eqref{eq:almostAbelian} is equivalent to
    \begin{equation*}
    g(B(Y,\cdot), B(Y,\cdot))=0 \, ,
    \end{equation*}
    for all $Y\in \mfn$, which implies $B=0$, and so $H=0$.
\end{proof}
Note that there are Abelian generalised Einstein pseudo-Riemannian Lie algebras $(\bR^n,H,\mathcal{G}_g,0)$ with $H\neq 0$. In fact, such examples already occur in dimension $n=5$ when $g$ has signature $(3,2)$:
\begin{example}\label{ex:abeliansignature32}
Let $(e_1,\ldots,e_5)$ be an orthonormal basis of $(\bR^5,g)$ such that $g(e_i,e_i)=1$ for $i=1,2,3$ and $g(e_4,e_4)=g(e_5,e_5)=-1$. Consider then
   \begin{equation*}
   	H:=e^1\wedge (e^{23}+e^{24}+e^{35}+e^{45}) \, .
   \end{equation*}
   We have that
   \begin{equation*}
   \begin{split}
   e_1\hook H&=e^{23}+e^{24}+e^{35}+e^{45} \, ,\qquad e_2\hook H=-e^{13}-e^{14} \, ,\qquad e_3\hook H=e^{12}-e^{15} \, ,\\
   e_4\hook H&=e^{12}-e^{15} \, ,\qquad e_5\hook H=e^{13}+e^{14} \, ,
   \end{split}
\end{equation*}
are all null and orthogonal to each other, which shows that $X\hook H$ is null for any $X\in \bR^5$. Thus, by \eqref{eq:gEdelta=0}, $(\bR^5,H,\cG_g,0)$ is generalised Einstein.
\end{example}
We now continue discussing the case of almost Abelian generalised Einstein Lorentzian Lie algebras $(\mfg,H,\cG_g,0)$ with non-degenerate $\mfn$ and $H\neq 0$. Recall that in Lemma \ref{le:fSnotdiag} we have showed that $f^S$ cannot be diagonalisable, i.e of first type in Lemma \ref{le:canonicalforms} (a). We now exclude also exclude the case where $f^S$ is of third type:
\begin{lemma}\label{le:fSnot3rd}
Let $(\mfg,H,\cG_g,0)$ be an almost Abelian generalised Einstein Lorentzian Lie algebra with non-degenerate Lorentzian codimension one Abelian ideal $\mfn$ and $H\neq 0$. Then $f^S$ is not of third type. 
\end{lemma}
\begin{proof}
	We argue by contradiction and assume that 
	\begin{equation*}
		f^S=\diag(L_2(\alpha,\epsilon),a_1,\ldots,a_{n-3}) \, ,
	\end{equation*}
	for some $\alpha,a_1,\ldots,a_{n-3}\in \bR$ and some $\epsilon\in\{-1,1\}$ with respect to some orthonormal basis $(e_1,\ldots,e_{n-1})$ of $\mfn$ with $-g(e_1,e_1)=g(e_2,e_2)=\cdots=g(e_{n-1},e_{n-1})=1$.	Let us set $U:=\spa{e_3,\ldots,e_{n-1}}$. Note that $U$ is a positive definite subspace, so we will write $\left\|T\right\|^2$ instead of $g(T,T)$ for tensors on $U$. We write
	\begin{equation*}
		\begin{split}
			H'&=e^1\wedge \eta_1+e^2\wedge \eta_2+e^{12}\wedge \gamma+H_0' \, ,\\
			B&=e^1\wedge \beta_1+e^2\wedge \beta_2+b\, e^{12}+B_0 \, ,
		\end{split}
	\end{equation*}
	for $H_0'\in \Lambda^3 U^*$, $\eta_1,\eta_2,B_0\in \Lambda^2 U^*$, $\beta_1,\beta_2,\gamma\in U^*$ and $b\in \bR$. Then, the first equation in \eqref{eq:almostAbelian} is explicitly given  by
	\begin{equation}\label{eq:3rdcase1}
		\begin{split}
			0&=2\,g(f^S,f^S)+g(B,B)=2\,\tr((f^S)^2)-\left\|\beta_1\right\|^2+\left\|\beta_2\right\|^2-b^2+\left\|B_0\right\|^2\\
			&=4\alpha^2+2\sum_{i=1}^{n-3} a_i^2-\left\|\beta_1\right\|^2+\left\|\beta_2\right\|^2-b^2+\left\|B_0\right\|^2 \, .
		\end{split}
	\end{equation}
	Moreover, if $f^A$ denotes the anti-symmetric part of $f$ and $\rho:=(f^A)_{12}=(f^A)_{21}$, then one computes that
	the upper $2\times 2$-block $M$ of the symmetric two-tensor $g([f^*,f](\cdot),\cdot)$ is given by
	\begin{equation*}
		M:=-2\epsilon\rho\, \begin{pmatrix} 1& 1\\ 1 & 1 \end{pmatrix} \, .
	\end{equation*}
	Now
	\begin{equation*}
		\begin{split}
			N&:=(g(H'(e_i,\cdot,\cdot),H'(e_j,\cdot,\cdot))+g(B(e_i,\cdot),B(e_i,\cdot)))_{i,j=1,2}\\
			&=
			\begin{pmatrix} \left\|\eta_1\right\|^2+\left\|\gamma\right\|^2+ \left\|\beta_1\right\|^2+b^2 &  g(\eta_1,\eta_2) +g(\beta_1,\beta_2)\\ g(\eta_1,\eta_2) +g(\beta_1,\beta_2) &  \left\|\eta_2\right\|^2-\left\|\gamma\right\|^2+ \left\|\beta_2\right\|^2-b^2  \end{pmatrix} \, .
		\end{split}
	\end{equation*}
	By the third equation in \eqref{eq:almostAbelian}, we have $M=-N$ so, in particular,
	\begin{equation*}
		\left\|\eta_1\right\|^2-\left\|\eta_2\right\|^2+\left\|\beta_1\right\|^2-\left\|\beta_2\right\|^2+2\left\|\gamma\right\|^2+ 2 b^2=0 \, .
	\end{equation*}
	Adding this equation twice to \eqref{eq:3rdcase1}, we arrive at
	\begin{equation}\label{eq:3rdcase2}
		0=4\alpha^2+2\sum_{i=1}^{n-3} a_i^2+\left\|\beta_1\right\|^2-\left\|\beta_2\right\|^2+2 \left\|\eta_1\right\|^2-2 \left\|\eta_2\right\|^2+4 \left\|\gamma\right\|^2+3b^2+\left\|B_0\right\|^2 \, .
	\end{equation}
	Now since $f^S$ is diagonal on $U$ for our choice of basis, one finds $g([f^*,f](e_i),e_i)=0$ for all $i=3,\ldots,n-1$ and so, by the third equation in \eqref{eq:almostAbelian}, we have
	\begin{equation*}
		\begin{split}
			0&=\sum_{i=3}^{n-1}\left( g(H'(e_i,\cdot,\cdot),H'(e_i,\cdot,\cdot))+g(B(e_i,\cdot),B(e_i,\cdot))\right)\\
			&=\sum_{i=3}^{n-1} -\left\|e_i\hook \eta_1\right\|^2+\left\|e_i\hook \eta_2\right\|^2-\left\|e_i\hook \gamma\right\|^2+\left\|e_i\hook H_0'\right\|^2-\left\|e_i\hook \beta_1\right\|^2+\left\|e_i\hook \beta_2\right\|^2\\
			&+\left\|e_i\hook B_0\right\|^2=-2\left\|\eta_1\right\|^2+2\left\|\eta_2\right\|^2-\left\|\gamma\right\|^2+3\left\|H'_0\right\|^2-\left\|\beta_1\right\|^2+\left\|\beta_2\right\|^2+2\left\|B_0\right\|^2 \, .
		\end{split}
	\end{equation*}
	Adding this equation to \eqref{eq:3rdcase2}, we get
	\begin{equation*}
		0=4\alpha^2+2\sum_{i=1}^{n-3} a_i^2+3 \left\|\gamma\right\|^2+3b^2+3\left\|B_0\right\|^2+3\left\|H_0'\right\|^2 \, ,
	\end{equation*}
	and so that $\alpha=a_1=\ldots=a_{n-3}=b=0$, $\gamma=0$, $B_0=0$ and $H_0'=0$. Now, since $M=-N$ and all entries in $M$ are equal, we get
	\begin{equation*}
		\left\|\eta_1\right\|^2+\left\|\beta_1\right\|^2=\left\|\eta_2\right\|^2+\left\|\beta_2\right\|^2=g(\eta_1,\eta_2)+g(\beta_1,\beta_2) \, ,
	\end{equation*}
	and so 
	\begin{equation*}
		\left\|\eta_1-\eta_2\right\|^2+\left\|\beta_1-\beta_2\right\|^2=\left\|\eta_1\right\|^2+\left\|\beta_1\right\|^2+\left\|\eta_2\right\|^2+\left\|\beta_2\right\|^2-2 g(\eta_1,\eta_2)-2g(\beta_1,\beta_2)=0 \, .
	\end{equation*}
	Hence, $\eta:=\eta_1=\eta_2$ and $\beta:=\beta_1=\beta_2$ and so
	\begin{equation*}
		H'=(e^1+e^2)\wedge \eta \, ,\quad B=(e^1+e^2)\wedge \beta \, .
	\end{equation*}
	Moreover,
	\begin{equation*}
		f^*(e_1)=\frac{\epsilon}{2}e_1-\left(\frac{\epsilon}{2}+\rho\right)e_2+u_1 \, ,\qquad
		f^*(e_2)=\left(\frac{\epsilon}{2}-\rho\right)e_1-\frac{\epsilon}{2}e_2+u_2 \, ,
	\end{equation*}
	for certain $u_1,u_2\in U$. Moreover, $f^*(u)=-f(u)$ for any $u\in U$. Writing $f(u)=h(u)+k(u)$ for linear maps $h:U\rightarrow \spa{e_1,e_2}$ and $k:U\rightarrow U$, we see
	that $k$ is skew-symmetric and that the fourth equation in \eqref{eq:almostAbelian} yields
	\begin{equation*}
		\begin{split}
		        (\beta\circ k)(u)&=g(k(u),\beta^{\sharp}) =g(f(u),\beta^{\sharp})=-g(f^*(u),\beta^\sharp)\\
		        &=-g(f^*(u),B(e_1,\cdot)^\sharp)=-g(f^*(e_1),B(u,\cdot)^\sharp)\\
			&=\beta(u)\, g\left(\frac{\epsilon}{2}e_1-\left(\frac{\epsilon}{2}+\rho\right)e_2,-e_1+e_2\right)=-\rho\beta(u) \, , \\
		\end{split}
	\end{equation*}
	for all $u\in U$. Thus, $\beta\circ k=-\rho\, \beta$. If $\beta\neq 0$, the one-form $\beta$ is an eigenvector with real eigenvalue $\rho$ for the skew-symmetric endomorphism $U^*\ni \alpha\mapsto \alpha\circ k\in U^*$, and so $\rho=0$ as skew-symmetric endomorphisms have only imaginary eigenvalues. Thus, we have either $\beta=0$ or $\rho=0$. Since $M=-N$ implies
	\begin{equation*}
		\rho=\epsilon\,\frac{\left\|\beta\right\|^2+\left\|\eta\right\|^2}{2} \, ,
	\end{equation*}
	the condition $\rho=0$ yields $\beta=0$, $\eta=0$ and so $H=0$, a contradiction. 
	Hence, we can discard that case and assume that $\rho\neq 0$ and $\beta=0$ in the following. Now $H=(e^1+e^2)\wedge \eta$ and one observes that $f.(e^1+e^2)=-(e^1+e^2)\circ f=-\rho\, (e^1+e^2)$, so that the closure condition for $H$ reads 
	\begin{equation*}
	 0=\dd H=X^\flat\wedge f. ((e^1+e^2)\wedge \eta)= (e^1+e^2)\wedge (-\rho\,\eta+f.\eta)\wedge X^\flat \, ,
	\end{equation*}
where we recall that $X$ was orthogonal to $\mfn$ with $g(X,X)=1$. Now we see that the closure condition is equivalent to 
\begin{equation*}
f.\eta=\rho\, \eta=\frac{\epsilon}{2} \left\|\eta\right\|^2 \eta \, .
\end{equation*}
As $f$ acts skew-symmetrically on $U$, and so also on $\Lambda^2 U^*$, we have
\begin{equation*}
0=g(f.\eta,\eta)=\frac{\epsilon}{2} \left\|\eta\right\|^4 \, .
\end{equation*}
Thus, $\eta=0$ and so $H=0$, again a contradiction. This finishes the proof.
\end{proof}
We are now able to show:
\begin{theorem}\label{th:almostAbeliannondegHneq0}
Let $(\mfg,H,\cG_g)$ be an almost Abelian generalised Lorentzian Lie algebra with non-degenerate codimension Abelian ideal $\mfn$ with $H\neq 0$ and assume that $f^S$ is not of second type.

Then $(\mfg,H,\cG_g,\delta=0)$ is generalised Einstein if and only if there exists an orthonormal basis of $\mfn$ (by an abuse of notation again denoted by) $(e_1,\ldots,e_{n-1})$ such that $g(e_1,e_1)=-g(e_i,e_i)=-1$ for all $i=2,\ldots,n-1$ and there exists some $k\in \left\{0,\ldots,\left\lfloor \frac{n-4}{2}\right\rfloor \right\}$, certain $c_1\geq \ldots \geq c_k>0$ and $u\in \bR^{n-4}$ such that when we set $U_1:=\spa{e_4,\ldots,e_{4+2k-1}}$ and $U_2:=\spa{e_{4+2k},\ldots,e_{n-1}}$ there exist $b\in \bR$, $\beta,\nu\in U_2^*$, $\tau_1\in \Lambda^2 U_1^*$ and $\tau_2\in \Lambda^2 U_2^*$ with at least one of $b,\beta,\nu,\tau_1,\tau_2$ being non-zero, such that
\begin{equation*}
\tau_1\in [[\Lambda^{1,1}U_1^*]] \, ,
\end{equation*}
 with respect to the almost complex structure $J=\diag(L_1(0,1),\ldots,L_1(0,1))$ on $U_1$ and
 \begin{equation*}
\tau_1(Y,Z)=0 \, ,
 \end{equation*}
whenever $Y\in \spa{e_{4+2i-2},e_{4+2i-1}}$ and $Z\in \spa{e_{4+2j-2},e_{4+2j-1}}$ for $i,j\in \{1,\ldots,k\}$ satisfying $c_i\neq c_j\,$, such that
\begin{equation*}
H=(e^1+e^3)\wedge (-e^2\wedge (\nu+b\, X^\flat)+\beta\wedge X^\flat+\tau_1+\tau_2) \, ,
\end{equation*}
and such that
\begin{equation*}
f=\begin{pmatrix} 0 & -\frac{1+\tilde\rho}{\sqrt{2}} & 0 & u^t \\
\frac{1-\tilde\rho}{\sqrt{2}} & 0 & \frac{1-\tilde\rho}{\sqrt{2}} & 0 \\
0 & \frac{1+\tilde\rho}{\sqrt{2}}
& 0 & -u^t \\
u & 0 & u & \diag(L_1(0,c_1),\ldots,L_1(0,c_k),0,\ldots,0)
\end{pmatrix} \, ,
\end{equation*}
for 
\begin{equation*}
\tilde{\rho}:=\frac{1}{2}\left(b^2+\left\|\beta\right\|^2+\left\|\nu\right\|^2+\left\|\tau_1\right\|^2+\left\|\tau_2\right\|^2\right) \, .
\end{equation*}
\end{theorem}
\begin{proof}
Note first that Lemma \ref{le:fSnotdiag} and Lemma \ref{le:fSnot3rd} show that for an almost Abelian generalised Einstein Lorentzian Lie algebra $(\mfg,H,\cG_g,0)$ with non-degenerate codimension one Abelian ideal $\mfn$, $\mfn$ has to be Lorentzian and the symmetric part $f^S$ of $f$ cannot be of the first or third canonical form in Lemma \ref{le:canonicalforms} (a). So if $f^S$ is not of the second type, it has to be of the fourth type in Lemma \ref{le:canonicalforms} (a):
	\begin{equation*}
		f^S=\diag(L_3(\alpha),a_1,\ldots,a_{n-4}) \, ,
	\end{equation*}
	for certain $\alpha,a_1,\ldots,a_{n-4}\in \bR$, with respect to an orthonormal basis $(e_1,\ldots,e_{n-1})$ of $\mfn$ such that $g(e_1,e_1)=-g(e_i,e_i)=-1$ for all $i=2,\ldots,n-1$. We work in this basis from now on. As a first step, we write $f^A$ as a block-matrix with respect to the splitting $\spa{e_1,e_2,e_3}\oplus \spa{e_4,\ldots,e_{n-1}}$,
	\begin{equation*}
	f^A=\begin{pmatrix} A_1 & A_2 \\ A_3 & A_4 \end{pmatrix} \, ,
	\end{equation*}
   for $A_1\in \bR^{3\times 3}$, $A_2\in \bR^{3\times (n-4)}$, $A_3\in \bR^{(n-4)\times 3}$ and $A_4\in \bR^{(n-4)\times (n-4)}$. Furthermore, we write
  \begin{equation}\label{eq:A3}
   A_3=\begin{pmatrix} u & v & w \end{pmatrix} \, ,
\end{equation}
   for $u,v,w\in \bR^{n-4}$ and note that the skew-symmetry of $f^A$ yields
    \begin{equation*}
   	A_2=\begin{pmatrix} u^t \\ -v^t \\  -w^t \end{pmatrix} \, .
   \end{equation*}
Moreover, $A_1$ may explicitly be written as
	\begin{equation*}
		f^A=\begin{pmatrix} 0 & \rho_1 & \rho_2 \\ \rho_1 & 0 & \rho_3 \\ \rho_2 & -\rho_3 & 0 \end{pmatrix} \, ,
	\end{equation*}
	for certain $\rho_1,\rho_2,\rho_3\in \bR$. Then one computes that
	\begin{equation*}
		M:=(g([f^*,f](e_i),e_j))_{i,j=1,2,3}=\sqrt{2}\begin{pmatrix}
			2\, \rho_1 & \rho_2 & \rho_1+\rho_3 \\
			\rho_2 & 2\, (\rho_1-\rho_3) & \rho_2 \\
			\rho_1+\rho_3 & \rho_2 & 2\, \rho_3
		\end{pmatrix} \, .
	\end{equation*}
	Furthermore, we set $U:=\spa{e_4,\ldots,e_{n-1}}$, note that $U$ is a positive definite subspace of $\mfn$ and write
	\begin{equation*}
	\begin{split}
		B&=b_1\, e^{23}+b_2\, e^{31}+b_3\, e^{12}+\sum_{i=1}^3 e^i\wedge \beta_i+ B_0 \, ,\\
		H'&=h\, e^{123}+e^{23}\wedge \nu_1+e^{31}\wedge \nu_2+e^{12}\wedge \nu_3+\sum_{i=1}^3 e^i\wedge \tau_i +H'_0 \, ,
	\end{split}
	\end{equation*}
	with $b_1,b_2,b_3,h\in \bR$, $\beta_1,\beta_2,\beta_3,\nu_1,\nu_2,\nu_3\in U^*$, $\tau_1,\tau_2,\tau_3,B_0\in \Lambda^2 U^*$ and $H'_0\in \Lambda^3 U^*$.
	Then a short computation yields that the matrix
	\begin{equation*}
			N:=\left(g(H'(e_i,\cdot,\cdot),H'(e_j,\cdot,\cdot))+g(B(e_i,\cdot,\cdot),B(e_j,\cdot,\cdot))\right)_{i,j=1,2,3}
   \end{equation*}
has the following entries:
\begin{equation*}
\begin{split}
N_{11}&=h^2+b_2^2+b_3^2+\left\|\beta_1\right\|^2+\left\|\nu_2\right\|^2+\left\|\nu_3\right\|^2+\left\|\tau_1\right\|^2 \, ,\\
N_{22}&=-h^2+b_1^2-b_3^2+\left\|\beta_2\right\|^2-\left\|\nu_3\right\|^2+\left\|\nu_1\right\|^2+\left\|\tau_2\right\|^2 \, ,\\
N_{33}&=-h^2+b_1^2-b_2^2+\left\|\beta_3\right\|^2-\left\|\nu_2\right\|^2+\left\|\nu_1\right\|^2+\left\|\tau_3\right\|^2 \, ,\\
N_{12}&=N_{21}=-g(\nu_1,\nu_2)+g(\tau_1,\tau_2)-b_1 b_2+g(\beta_1,\beta_2) \, ,\\
N_{13}&=N_{31}=-g(\nu_1,\nu_3)+g(\tau_1,\tau_3)-b_1 b_3+g(\beta_1,\beta_3) \, ,\\
N_{23}&=N_{32}=g(\nu_2,\nu_3)+g(\tau_2,\tau_3)+b_2 b_3+g(\beta_2,\beta_3) \, .
\end{split}
\end{equation*}
	We note that we must have $M=-N$ due to the third equation in \eqref{eq:almostAbelian}. Hence, $2M_{13}=M_{11}+M_{33}$ implies
	\begin{equation*}
	\begin{split}
	-2 g(\nu_1,\nu_3)+2 g(\tau_1,\tau_3)-2 b_1 b_3+2 g(\beta_1,\beta_3)&=2 N_{13}=N_{11}+N_{33}\\
	&=b_1^2+b_3^2+\left\|\beta_1\right\|^2+\left\|\beta_3\right\|^2\\
	&+\left\|\nu_1\right\|^2+\left\|\nu_3\right\|^2+\left\|\tau_1\right\|^2+\left\|\tau_3\right\|^2 \, ,
	\end{split}	
	\end{equation*}
    which is equivalent to
    \begin{equation*}
   0=\left\|\nu_1+\nu_3\right\|^2+ \left\|\tau_1-\tau_3\right\|^2+(b_1+b_3)^2+\left\|\beta_1-\beta_3\right\|^2 \, ,
    \end{equation*}
    and so to
    \begin{equation}\label{eq:fourthcase1}
    \nu:=\nu_1=-\nu_3 \, ,\quad \tau:=\tau_1=\tau_3 \, ,\quad b:=b_1=-b_3 \, ,\quad \beta:=\beta_1=\beta_3 \, .
    \end{equation}
 In addition, $M_{22}=M_{11}-M_{33}$ and so
	\begin{equation}\label{eq:fourthcase2}
		-h^2+\left\|\beta_2\right\|^2+\left\|\tau_2\right\|^2=N_{22}=N_{11}-N_{33}=2\left(h^2+b_2^2+ \left\|\nu_2\right\|^2\right) \, .
	\end{equation}
 Moroever, as $f^S$ is diagonal on $U$, we see that $g([f^*,f](e_i),e_i)=0$ for all $i=4,\ldots,n-1$. Consequently, the third equation in \eqref{eq:almostAbelian} gives, using already \eqref{eq:fourthcase1},
 \begin{equation*}
 \begin{split}
 0&=g(H'(e_i,\cdot,\cdot),H'(e_i,\cdot,\cdot))+g(B(e_i,\cdot),B(e_i,\cdot))\\
&=\beta_2^2(e_i)+\left\|B_0(e_i,\cdot)\right\|^2-\left\|\nu_2(e_i,\cdot)\right\|^2+\left\|\tau_2(e_i,\cdot)\right\|^2+\left\|H'_0(e_i,\cdot,\cdot)\right\|^2 \, .
 \end{split}
 \end{equation*}
Summing this up over $i=4,\ldots,n-1$, we arrive at
 \begin{equation}\label{eq:fourthcase3}
		0=\left\|\beta_2\right\|^2+2\left\|B_0\right\|^2-2\left\|\nu_2\right\|^2+2\left\|\tau_2\right\|^2+3\left\|H'_0\right\|^2 \, .
\end{equation}
Subtracting \eqref{eq:fourthcase2} from \eqref{eq:fourthcase3}, we get
\begin{equation*}
h^2+2\left\|B_0\right\|^2-2\left\|\nu_2\right\|^2+\left\|\tau_2\right\|^2+3\left\|H'_0\right\|^2=-2\, h^2-2\, b_2^2- 2 \left\|\nu_2\right\|^2 \, ,
\end{equation*}
and so
\begin{equation*}
0=3\,h^2+2\, b_2^2+2\left\|B_0\right\|^2+\left\|\tau_2\right\|^2+3\left\|H'_0\right\|^2 \, .
\end{equation*}
Thus, $h=b_2=0$, $B_0=\tau_2=0$ and $H_0'=0$. As a result, \eqref{eq:fourthcase2} reduces to
\begin{equation}\label{eq:fourthcase4}
\left\|\beta_2\right\|^2=2\left\|\nu_2\right\|^2 \, ,
\end{equation}
and we also have that
\begin{equation*}
B=\left(e^1+e^3\right)\wedge (-b e^2+\beta)+e^2\wedge \beta_2 \, ,
\end{equation*}
which implies $g(B,B)=\left\|\beta_2\right\|^2$. Therefore, the first equation in \eqref{eq:almostAbelian} becomes
	\begin{equation*}
		0=2 g(f^S,f^S)+g(B,B)=2\tr((f^S)^2)+g(B,B)=2\tr((f^S)^2)+\left\|\beta_2\right\|^2 \, .
	\end{equation*}
Since $\tr((f^S)^2)\geq 0$ in this case by Lemma \ref{le:trfsquare=0}, we must have $\tr((f^S)^2)=0$ and $\beta_2=0$. Now, by Lemma \ref{le:trfsquare=0} this implies $\alpha=a_1=\ldots=a_{n-4}=0$. In particular, $f^S|_U=0$.
Moreover, $\nu_2=0$ by \eqref{eq:fourthcase4}, which implies
\begin{equation*}
	B=\left(e^1+e^3\right)\wedge (-b e^2+\beta) \, ,\quad  H'=(e^1+e^3)\wedge (-e^2\wedge \nu+\tau) \, .
\end{equation*}	
Hence, $0=N_{12}=-M_{12}=-\sqrt{2}\, \rho_2$ gives $\rho_2=0$, and
\begin{equation*}
\rho:=\rho_1=-\frac{N_{11}}{2\sqrt{2}}=-\frac{b^2+\left\|\beta\right\|^2+\left\|\nu\right\|^2+\left\|\tau\right\|^2}{2\sqrt{2}}=-\frac{N_{33}}{2\sqrt{2}}=\rho_3 \, .
\end{equation*}
Note that the statement of Theorem is given in terms of $\tilde{\rho}:=-\sqrt{2}\rho$, but for convenience we will use $\rho$ for the remainder of the proof.

Observing that $Y\hook H'$ is orthogonal to $Z\hook H'$ and $Y\hook B$ is orthogonal to $Z\hook B$ for all $Y\in \spa{e_1,e_2,e_3}$ and $Z\in U$, (the polarisation of) the third equation in \eqref{eq:almostAbelian} yields
\begin{equation*}
2\,g([f^S,f^A](Y),Z)=g([f^*,f](Y),Z)=0 \, ,
\end{equation*}
for all $Y\in \spa{e_1,e_2,e_3}$ and all $Z\in U$. Hence, $[f^S,f^A]$ preserves the subspace $\spa{e_1,e_2,e_3}$. Now
\begin{equation*}
[f^S,f^A](e_2)=f^S\left(\rho (e_1-e_3)+v \right)-f^A\left(-\frac{1}{\sqrt{2}}(e_1-e_3)\right)=\frac{1}{\sqrt{2}} (u-w) \, ,
\end{equation*}
which is in $\spa{e_1,e_2,e_3}$ only if $w=u$. Moreover,
\begin{equation*}
\begin{split}
	[f^S,f^A](e_1)&=f^S\left(\rho e_2+u\right)-f^A\left(\frac{1}{\sqrt{2}}\, e_2\right)\\
	&=-\frac{\rho}{\sqrt{2}}(e_1-e_3)-\frac{\rho}{\sqrt{2}} (e_1-e_3)-\frac{1}{\sqrt{2}}\, v=-\sqrt{2}\rho (e_1-e_3)-\frac{1}{\sqrt{2}}\, v  \, ,
\end{split}
\end{equation*}
and the condition that this lies in $\spa{e_1,e_2,_3}$ yields $v=0$.

Next, since $A_4$ is skew-symetric on the positive definite subspace $U$, there is an orthonormal basis $(e_4,\ldots,e_{n-1})$ of $U$, some $k\in \left\{0,\ldots,\left\lfloor \frac{n-4}{2}\right\rfloor \right\}$ and certain $c_1\geq \ldots \geq c_k>0$ such that
\begin{equation*}
A_4=\diag(L_1(0,c_1),\ldots,L_1(0,c_k),0,\ldots,0) \, ,
\end{equation*}
with respect to $(e_4,\ldots,e_{n-1})$. We set $U_1:=\spa{e_4,\ldots,e_{4+2k-1}}$ and $U_2:=\spa{e_{4+2k},\ldots,e_{n-1}}$ and use now the equation $\dd H=0$ to get some more information on $f$ and $H$. For this, note first that
\begin{equation*}
\dd H=X^\flat\wedge f.H' \, .
\end{equation*}
Using that $f.(e^1+e^3)=0$ and $f.e^2\in \spa{e^1+e^3}$, we find that
\begin{equation*}
\dd H=X^\flat\wedge (e^1+e^3)\wedge (-e^2\wedge f.\nu+f.\tau) \, .
\end{equation*}
Note also that, for a one-form $\gamma\in U^*\cong (\bR^{n-4})^*$, we have $f.\gamma=-\gamma(u)\, (e^1+e^3)+A_4.\gamma$. We thus see that $\dd H=0$ if and only if
\begin{equation*}
A_4.\nu=0 \, ,\qquad A_4.\tau=0 \, .
\end{equation*}
Since $0=A_4.\nu=-\nu\circ A_4$ and $\mathrm{im}(A_4)=U_1\,$, the first equation is equivalent to $\nu(U_1)=0$, i.e. to $\nu\in U_2^*$. Moreover, $A_4.\tau=0$ is equivalent to the $U_1^*\wedge U_2^*$-part of $\tau$ vanishing---that is $\tau=\tau_1+\tau_2$ with $\tau_i\in \Lambda^2 U_i^*$ for $i=1,2$---and to $\tilde{A}.\tau_1=0$, where $\tilde{A}:=\diag(L_1(0,c_1),\ldots,L_1(0,c_k))$. Let us analyse this last condition further. Let
\begin{equation*}
J:=\diag(L_1(0,1),\ldots,L_1(0,1))	
\end{equation*}
be the natural almost complex structure on $U_1\,$. Note that on $\spa{e_{4+2i-2},e_{4+2i-1}}$ for some $i\in \{1,\ldots,k\}$, we have $\tilde{A}=c_i J$ and so $\tilde{A}^2=-c_i^2 \id$ on that subspace. Hence, if $v\in \spa{e_{4+2i-2},e_{4+2i-1}}$ and $w\in  \spa{e_{4+2j-2},e_{4+2j-1}}$ for $i,j\in \{1,\ldots,k\}$, using that $\tilde{A}.\tau_1=0$ we find
\begin{equation*}
\tau_1(v,w)=-\frac{1}{c_i^2} \tau_1(\tilde{A}^2 v,w)=-\frac{1}{c_i^2} \tau_1(v,\tilde{A}^2 w)=\frac{c_j^2}{c_i^2}\, \tau_1(v,w) \, ,
\end{equation*}
which shows, due to $c_i,c_j>0$, that either $c_i=c_j$ or $\tau_1(v,w)=0$. However, if $c_i=c_j\,$, we see that $\tau_1(Jv,w)=-\tau_1(v,Jw)$. Hence, $\tau_1$ is of type $(1,1)$ with respect to $J$.

Conversely, assuming all the conditions that we have derived so far we have, in fact, that $\tilde{A}.\tau_1=0$ and so $H$ is closed. We also have that the first, second and third equations in \eqref{eq:almostAbelian} are satisfied.

Therefore, it only remains to study the last equation in \eqref{eq:almostAbelian}. For this, note that
\begin{align*}
B(e_1,\cdot)^\sharp=B(e_3,\cdot)^\sharp&=-b\, e_2+\beta^\sharp \, , & B(e_2,\cdot)^\sharp&=-b (e_1-e_3) \, , & \\
B(Y,\cdot)^\sharp&=\beta(Y)\, (e_1-e_3) \, , &&&
\end{align*}
for all $Y\in U$ and that
\begin{equation*}
\begin{split}
f^*(e_1)&=f^*(e_3)=\left(\frac{1}{\sqrt{2}}-\rho\right)\, e_2-u \, ,\quad f^*(e_2)=-\left(\frac{1}{\sqrt{2}}+\rho\right) (e_1-e_3) \, ,\\
f^*(Y)&=-u^t Y (e_1-e_3)-A_4\, Y  \, ,
\end{split}
\end{equation*}
for all $Y\in U\cong \bR^{n-4}$. Using these formulas, one sees that the last equation in \eqref{eq:almostAbelian} is satisfied if and only if $\beta(A_4 Y)=g(\beta^\sharp,A_4 Y)=0$ for all $Y\in U$, i.e. if and only if $\beta(U_1)=\beta(\mathrm{im}(A_4))=0$, i.e. if and only if $\beta\in U_2^*$. This finishes the proof.
\end{proof}
Thus, the only remaining case is that where $f^S$ is of second type, namely $f^S=\diag(L_2(\alpha,\beta),a_1,\ldots,a_{n-3})$ for $\alpha,a_1,\ldots,a_{n-3}\in \bR$ and some $\beta>0$ with respect to an appropriate orthonormal basis of $\mfn$. This situation seems to be too hard to be investigated in total generality in all dimensions. The problem stems from the fact that in this case---as described in Lemma \ref{le:trfsquare=0}---the condition $\tr((f^S)^2)=0$ does not imply that all involved parameters are zero, in contrast with other cases. However, in dimension four this case can be excluded and we achieve the following classification:
\begin{theorem}\label{th:4dalmostAbeliannondegHneq0}
	Let $(\mfg,H,\cG_g)$ be a four-dimensional almost Abelian generalised Lorentzian Lie algebra with non-degenerate codimension one Abelian ideal $\mfn$ and $H\neq 0$. Then $(\mfg,H,\cG_g,0)$ is generalised Einstein if and only if there exists $b\in \bR^*$ and an orthonormal basis $(e_1,e_2,e_3,e_4)$ of $\mfg$ with $g(e_1,e_1)=-g(e_2,e_2)=-g(e_3,e_3)=-g(e_4,e_4)=-1$ such that
	\begin{equation*}
	H=b\, (e^1+e^3)\wedge e^{24} \, ,\qquad
	f=\frac{1}{\sqrt{2}}\begin{pmatrix} 0 & -1-\frac{b^2}{2} & 0 \\ 1-\frac{b^2}{2} & 0 & 1-\frac{b^2}{2} \\ 0 & 1+\frac{b^2}{2} & 0
	\end{pmatrix} \, ,
	\end{equation*}
	with respect to $(e_1,e_2,e_3,e_4)$, where $\mfn=\spa{e_1,e_2,e_3}$ and $f=\ad\vert_{e_4}$. Note that in this case, $\mfg\cong A_{4,1}$ for $b\neq\pm\sqrt{2}$ and $\mfg\cong\mfh_3\oplus \bR$ for $b=\pm\sqrt{2}$.
\end{theorem}
\begin{proof}
Using Theorem \ref{th:almostAbeliannondegHneq0}, we are left with showing that under the assumption $\dim(\mfg)=4$, $f^S$ cannot be of second type. To show this, we argue by contradiction and assume that
	\begin{equation*}
		f^S=\diag(L_1(\alpha,\beta),d)	 \, ,
	\end{equation*}
	for certain $\alpha,d\in \bR$ and $\beta>0$. We can also write
	\begin{equation*}
		f^A=\begin{pmatrix} 0 & \rho & \tau_1 \\ \rho & 0 & \tau_2 \\ \tau_1 & -\tau_2 & 0 \end{pmatrix} \, ,
	\end{equation*}
	for certain $\rho,\tau_1,\tau_2\in\bR$. Thus,
	\begin{equation*}
		(g(f^*f-ff^*(e_i),e_j))_{i,j=1,2}=4\, \beta\rho\, I_2  \, .
	\end{equation*}
     Now, by Corollary \ref{co:almostAbelian4d}, we have $H'=0$. We write $B=b_1\, e^{23}+b_2 e^{31}+b_3 e^{12}$ with $b_1,b_2,b_3\in \bR$ and observe that $[f^*,f]=2\, [f^S,f^A]$. Then
   \begin{equation*}
   	g([f^*,f](e_3),e_3)=2 g([f^S,f^A](e_3),e_3)=4 g(f^S(e_3),f^A(e_3))=0 \, ,
   \end{equation*}
so the third equation in \eqref{eq:almostAbelian} yields
   \begin{equation*}
   	0=g(B(e_3,\cdot),B(e_3,\cdot))=b_1^2-b_2^2 \, ,
   \end{equation*}
   that is $b_1^2=b_2^2$. Moreover, the third equation in \eqref{eq:almostAbelian} also gives
   	\begin{equation*} 
   	    \begin{split}	
   		 4\, \beta\rho\, I_2&=	(g(f^*f-ff^*(e_i),e_j))_{i,j=1,2}=-(g(B(e_i,\cdot),B(e_j,\cdot)))_{i,j=1,2}\\
   		 &=\begin{pmatrix} -b_2^2-b_3^2 & b_1 \, b_2 \\ b_1\, b_2 & -b_1^2+b_3^2 \end{pmatrix} \, .
   		 \end{split}
   \end{equation*}
   Hence, $b_1 b_2=0$ and so, due to $b_1^2=b_2^2$, we have $b_1=b_2=0$. Then, $-b_3^2= 4\, \beta\rho=b_3^2$ yields also $b_3=0$ and so $B=0$. Thus $H=0$, a contradiction. Hence, $(\mfg,H,\cG_g,0)$ cannot be generalised Einstein.
\end{proof}
We end this section by noting that there are examples of almost Abelian generalised Einstein Lorentzian Lie algebras $(\mfg,H,\cG_g,0)$ with $\mfn$ being non-degenerate and $H\neq 0$ for which $f^S$ is of second type already in dimension five:
\begin{example}\label{ex:5dalmostAbelian}
	Let $(\mfg,g)$ be the five dimensional almost Abelian Lorentzian Lie algebra for which there exists an orthonormal basis $(e_1,\ldots,e_5)$ with $g(e_1,e_1)=-g(e_i,e_i)=-1$ for $i=2,\ldots,5$ such that $\mfn=\spa{e_1,\ldots,e_4}$ and such that $f=\ad_{e_5}|_{\mfn}$ is given by
	\begin{equation*}
		f=\left(\begin{smallmatrix} 0 & -2 & 0 & 0 \\ 0 & 0 & 0 & 0 \\ 0 & 0 & 1 & 1 \\ 0 & 0 & -1 & -1 \end{smallmatrix}\right) \, .
	\end{equation*}
Moreover, let
	\begin{equation*}
		H=\sqrt{2}\left(e^{135}+e^{145}-e^{235}+e^{245}\right) \, .
	\end{equation*}
We claim that $(\mfg,H,\cG_g,0)$ is generalised Einstein. For this, we note first that $H'=0$, that
\begin{equation*}
	B=\sqrt{2} \left(e^{13}+e^{14}-e^{23}+e^{24}\right) \, ,
\end{equation*}
and that
	\begin{equation*}
		f^S=\left(\begin{smallmatrix} 0 & -1 & 0 & 0 \\ 1 & 0 & 0 & 0 \\ 0 & 0 & 1 & 0 \\ 0 & 0 & 0 & -1 \end{smallmatrix}\right) \, .
	\end{equation*}
We see that $f^S$ is of second type. Moreover, $g(B,B)=0$ and so
\begin{equation*}
2g(f^S,f^S)+g(B,B)=\tr((f^S)^2)=0 \, ,
\end{equation*}
i.e. the first equation in \eqref{eq:almostAbelian} is satisfied. The second equation in \eqref{eq:almostAbelian} is also satisfied since $H'=0$.

Next, computing that
\begin{equation*}
g([f^*,f](e_i),e_j)_{i,j=1,\ldots,4}=\left(\begin{smallmatrix}
	-4 & 0 & 0 & 0 \ \\
	0& -4 & 0 & 0 \\
	0	& 0 &0  & 4\\
	0 & 0 & 4 & 0
\end{smallmatrix}\right) \, ,
\end{equation*}
and that
\begin{equation*}
	\begin{split}
	-(g(H'(e_i,\cdot,\cdot),H'(e_j,\cdot,\cdot))&+g(B(e_i,\cdot),B(e_j,\cdot))_{i,j=1,\ldots,4}\\
	&=-g(B(e_i,\cdot),B(e_j,\cdot))_{i,j=1,\ldots,4}=\left(\begin{smallmatrix}
		-4 & 0 & 0 & 0 \ \\
	0& -4 & 0 & 0 \\
	0	& 0 &0  & 4\\
		0 & 0 & 4 & 0
	\end{smallmatrix}\right) \, ,
\end{split}
\end{equation*}
we see that also the third equation in \eqref{eq:almostAbelian} holds. Finally, noting that
\begin{equation*}
f^*=\left(\begin{smallmatrix} 0 & 0 & 0 & 0 \\ 2 & 0 & 0 & 0 \\ 0 & 0 & 1 & -1 \\ 0 & 0 & 1 & -1 \end{smallmatrix}\right) \, ,
\end{equation*}
and that
\begin{equation*}
\begin{split}
B(e_1,\cdot)^{\sharp}&=\sqrt{2}(e_3+e_4) \, ,\qquad B(e_2,\cdot)^{\sharp}=\sqrt{2}(-e_3+e_4) \, ,\\ B(e_3,\cdot)^{\sharp}&=\sqrt{2}(e_1+e_2) \, ,\qquad B(e_4,\cdot)^\sharp=\sqrt{2}(e_1-e_2) \, ,
\end{split}
\end{equation*}
one checks that
\begin{equation*}
g(f^*(e_i),B(e_j,\cdot)^\sharp)=g(f^*(e_j),B(e_i,\cdot)^\sharp) \, ,
\end{equation*}
for all $i,j=1,\ldots,4$. Hence, also the fourth equation in \eqref{eq:almostAbelian} is satisfied and so
$(\mfg,H,\cG_g,0)$ is, indeed, generalised Einstein.
\end{example}

\subsection{Almost Abelian Lorentzian case with degenerate $\mfn$}
We now turn our attention to the case where $\mfn$ is degenerate with respect to the pseudo-Riemannian metric $g$. As $g$ is Lorentzian and $\mfn$ is degenerate, it will turn out to be useful to work with the following basis:
\begin{definition}\label{def:Wittbasis}
Let $(\mfg,g)$ be an almost Abelian Lorentzian Lie algebra with degenerate codimension one Abelian ideal $\mfn$. A basis $(e_1,\ldots,e_n)$ of $\mfg$ is called a \emph{Witt basis} if $(e_1,\ldots,e_{n-1})$ is a basis of $\mfn$ and
\begin{equation*}
\begin{split}
g(e_i,e_j)&=\delta_{ij} \, ,\quad g(e_{n-1},e_i)=g(e_n,e_i)=0 \, ,\quad g(e_{n-1},e_{n-1})=g(e_n,e_n)=0 \, ,\\
g(e_{n-1},e_n)&=1 \, ,
\end{split}
\end{equation*}
for all $i,j=1,\ldots,n-2$.
\end{definition}
We may now show:
\begin{theorem}\label{th:almostAbeliandeg}
Let $(\mfg,H,\cG_g)$ be an $n$-dimensional almost Abelian generalised Lorentzian Lie algebra with degenerate codimension one Abelian ideal $\mfn$. Moreover, let $(e_1,\ldots,e_{n-2},Y,X)$ be a Witt basis, set $f:=\ad_X|_{\mfn}\in \End(\mfn)$ and $U:=\spa{e_1,\ldots,e_{n-2}}$. Then $(\mfg,H,\cG_g,0)$ is generalised Einstein if and only $H=0$ and there exists some $k\in \left\{0,\ldots,\left\lfloor\tfrac{n-2}{2}\right\rfloor\right\}$ such that
\begin{equation*}
f=\begin{pmatrix} \diag(L_1(0,c_1),\ldots, L_1(0,c_k),0,\ldots,0) & 0 \\ \alpha & 0 \end{pmatrix} \, ,
\end{equation*}
for some $\alpha\in U^*$ and certain $c_1\geq \ldots\geq c_k>0$ with respect to the splitting $\mfn=U\oplus \spa{Y}$.

Moreover, $g$ is then flat.
\end{theorem}
\begin{proof}
Observing that $U$ is positive definite, we write $\left\|T\right\|^2$ instead of $g(T,T)$ for any tensor $T$ on $U$ in what follows.

Next, we describe how to compute scalar products and adjoints of linear endomorphisms of $\mfg$ according to the splitting $\mfg=U\oplus \spa{Y}\oplus \spa{X}$. For this, let $F\in \End(\mfg)$ and decompose $F$ as
\begin{equation*}
F=\begin{pmatrix} \tilde{F} & v & w \\ \alpha & a & b \\ \beta & c & d
\end{pmatrix} \, ,
\end{equation*}
with $\tilde{F}\in \End(U)$, $v,w\in U$, $\alpha,\beta\in U^*$ and $a,b,c,d\in \bR$. Then
\begin{equation}\label{eq:norm}
g(F,F)=\left\|\tilde{F}\right\|^2+2\, g(\alpha,\beta)+2\, g(v,w)+2\, (ad+bc) \, .
\end{equation}
Moreover, 
\begin{equation}\label{eq:adjoint}
F^*=\begin{pmatrix} \tilde{F}^* & \beta^\sharp &\alpha^\sharp \\ w^\flat & d & b \\ v^\flat & c & a
\end{pmatrix} \, .
\end{equation}	
Now, decompose first
$f$ as
\begin{equation*}
	f=\begin{pmatrix} \tilde{f} & v \\ \alpha & a \end{pmatrix} \, ,
\end{equation*}
with $\tilde{f}\in \End(U)$, $v\in U$, $\alpha\in U^*$ and $a\in \bR$ according to the splitting $\mfn=U\oplus \spa{Y}$. Then
\begin{equation*}
\ad_X=\begin{pmatrix} \tilde{f} & v & 0 \\ \alpha & a & 0 \\ 0 & 0 & 0 \end{pmatrix} \, ,
\end{equation*}	
and by \eqref{eq:adjoint}
\begin{equation*}
	\ad_X^*=\begin{pmatrix} \tilde{f}^* & 0 & \alpha^\sharp \\ 0 & 0 & 0 \\ v^\flat & 0 & a \end{pmatrix} \, .
\end{equation*}	
Thus,
\begin{equation*}
	\ad_X^S=\begin{pmatrix} \tilde{f}^S & \tfrac{1}{2}\,v & \tfrac{1}{2}\, \alpha^\sharp \\ \tfrac{1}{2}\, \alpha & \frac{a}{2} & 0 \\ \tfrac{1}{2}\, v^\flat & 0 & \frac{a}{2} \end{pmatrix} \, .
\end{equation*}	
and so 
\begin{equation}\label{eq:adXSnorm}
4\, g(\ad_X^S,\ad_X^S)=4\, \left\| \tilde{f}^S \right\|^2+4\, \alpha(v)+ 2\, a^2 \, ,
\end{equation}
by \eqref{eq:norm}. Furthermore,
\begin{equation*}
\ad_Y=\begin{pmatrix} 0 & 0 & -v \\ 0 & 0 & -a \\ 0 & 0 & 0 \end{pmatrix} \, ,
\end{equation*}
so we get 
\begin{equation*}
	\ad_Y^S= \begin{pmatrix} 0 & 0 & -\tfrac{1}{2}\,v \\ -\tfrac{1}{2}\, v^\flat & 0 & -a \\ 0 & 0 & 0 \end{pmatrix} \, ,
\end{equation*}
using \eqref{eq:adjoint}. This implies $g(\ad_Y^S,\ad_Y^S)=0$ by \eqref{eq:norm}. Moreover, $Y\in \mfn^{\perp}\subseteq (\mfg')^{\perp}$ and so $\ad^*(Y)=0$. Hence, writing
\begin{equation*}
H=H_0+X^\flat\wedge B_1+Y^\flat\wedge B_2+X^\flat\wedge Y^\flat\wedge \beta \, ,
\end{equation*}
for $H_0\in \Lambda^3 U^*$, $B_1,B_2\in \Lambda^2 U^*$ and $\beta\in  U^*$, we have
\begin{equation*}
Y\hook H=B_1+Y^\flat\wedge \beta \, ,
\end{equation*}
and inserting $Y$ into the first equation in \eqref{eq:gEdelta=0} we find
\begin{equation*}
0=4g(\ad_Y^S,\ad_Y^S)-g(\ad^*(Y),\ad^*(Y))+2 g(H(Y,\cdot,\cdot),H(Y,\cdot,\cdot))=2 \left\|B_1\right\|^2 \, ,
\end{equation*}
i.e. $B_1=0$. On the other hand, we have
\begin{equation*}
	X\hook H=B_2-X^\flat\wedge \beta \, .
\end{equation*}
Now, inserting $X$ and $Y$ into the polarisation of the first equation in \eqref{eq:gEdelta=0} and using the polarisation of \eqref{eq:norm} we obtain
\begin{equation*}
0=2 g(\ad_X^S,\ad_Y^S)+g(H(X,\cdot,\cdot),H(Y,\cdot,\cdot))=-\left\|v\right\|^2-\left\|\beta\right\|^2 \, ,
\end{equation*} 
i.e. $v=0$ and $\beta=0$. Now observe that
\begin{equation*}
\ad^*(X)=\begin{pmatrix} 0 & 0 & \alpha^\sharp \\ -\alpha & -a & 0 \\ 0 & 0 & a \end{pmatrix} \, .
\end{equation*}
Thus, \eqref{eq:norm} and \eqref{eq:adXSnorm} imply
\begin{equation*}
\begin{split}
0&=4\, g(\ad_X^S,\ad_X^S)-g(\ad^*(X),\ad^*(X))+2\, g(H(X,\cdot,\cdot),H(X,\cdot,\cdot))\\
&=4\, \left\| \tilde{f}^S \right\|^2+ 2\, a^2+2\, a^2+2\,\left\|B_2\right\|^2=4\, \left\| \tilde{f}^S \right\|^2+4\, a^2+2\,\left\|B_2\right\|^2 \, ,
\end{split}
\end{equation*}
and so $\tilde{f}^S=0$, $a=0$ and $B_2=0$. Summarising our results at this point, we have shown that
\begin{equation*}
H=H_0\in \Lambda^3 U^*\,,\quad f=\begin{pmatrix} \tilde{f} & 0 \\ \alpha & 0 \end{pmatrix} \, ,
\end{equation*}
for some skew-symmetric $\tilde{f}$. Note that then
\begin{equation*}
\ad_Z= \begin{pmatrix} 0 & 0 & -\tilde{f}(Z) \\ 0 & 0 & -\alpha(Z) \\ 0 & 0 & 0 \end{pmatrix} \, ,
\end{equation*}
and so
\begin{equation*}
	\ad_Z^S= \begin{pmatrix} 0 & 0 & -\frac{1}{2}\,\tilde{f}(Z) \\ -\frac{1}{2}\,\tilde{f}(Z)^\flat & 0 & -\alpha(Z) \\ 0 & 0 & 0 \end{pmatrix} \, ,
\end{equation*}
for all $Z\in U$ by \eqref{eq:adjoint}. Hence, $g(\ad_Z^S,\ad_Z^S)=0$ by \eqref{eq:norm}. Next, a short computation shows
\begin{equation*}
\ad^*(Z)=\begin{pmatrix} 0 & 0 & \tilde{f}^*(Z)  \\ -\tilde{f}^*(Z)^\flat & 0 & 0\\ 0 & 0 & 0 \end{pmatrix} \, ,
\end{equation*}
and then \eqref{eq:norm} yields $g(\ad^*(Z),\ad^*(Z))=0$ for all $Z\in U$. Thus, the first equation in \eqref{eq:gEdelta=0} reduces to
\begin{equation*}
0=g(H(Z,\cdot,\cdot),H(Z,\cdot,\cdot))=\left\|Z\hook H_0\right\|^2 \, ,
\end{equation*}
i.e. to $Z\hook H_0$ for all $Z\in U$. This implies $H_0=0$, and so $H=0$ and the second equation in \eqref{eq:gEdelta=0} is trivially satisfied. Finally, since $U$ is positive definite, we may find an orthonormal basis of $U$ such that the skew-symmetric endomorphism $\tilde{f}$ of $U$ equals $\diag(L_1(0,c_1),\ldots,L_1(0,c_k),0,\ldots,0)$ for some $k\in \left\{0,\ldots,\left\lfloor \frac{n-2}{2}\right\rfloor\right\}$ and certain $c_1\geq \ldots\geq c_k>0$.

Conversely, assuming the conditions in the statement, one easily checks that \eqref{eq:gEdelta=0} is fulfilled, i.e. that $(\mfg,H=0,\cG_g,0)$ is generalised Einstein. 
 
To conclude, one can show that $g$ is flat using the Koszul formula to compute that $\nabla^g_Z=0$ for all $Z\in U$ and $\nabla^g_Y=0$, which already implies that $R^g=0$, i.e. that $g$ is flat.
\end{proof}
\begin{remark}
Let $\mfg$ be an almost Abelian Lorentzian Lie algebra which admits a generalised Einstein structure of the form $(H,\cG_g,0)$ such that a codimension one Abelian ideal is degenerate. Then the real Jordan normal form of $f$ is given by
\begin{equation*}
\diag(L_1(0,c_1),\ldots,L_1(0,c_k),0,\ldots,0) \, ,
\end{equation*}
or by
\begin{equation*}
	\diag(L_1(0,c_1),\ldots,L_1(0,c_k),J_2(0),\ldots,0) \, ,
\end{equation*}
where $J_2(0)$ is a Jordan block of size two with $0$ on the diagonal. This has the following implications:
\begin{itemize}
	\item[(i)]
$\mfg$ also admits a generalised Einstein structure with non-degenerate codimension one Abelian ideal since the first class of possible real Jordan normal forms is as in Theorem \ref{th:almostAbeliannondegH=0} (a), whereas the last class of possible real Jordan normal forms is as in Theorem \ref{th:almostAbeliannondegH=0} (b) (v).
\item[(ii)]
In low dimensions, the following almost Abelian Lie algebras $\mfg$ posses a generalised Lorentzian Einstein metric $(H,\mathcal{G}_g,\delta=0)$ with degenerate codimension one Abelian ideal:
\begin{itemize}[$\bullet$]
	\item $\dim(\mfg)=3$: $\bR^3\,$, $\mfh_3\,$, $\mathfrak{e}(2)$.
	\item $\dim(\mfg)=4$: $\bR^4\,$, $\mfh_3\oplus \bR\,$, $\mathfrak{e}(2)\oplus \bR$.
	\item $\dim(\mfg)=5$: $\bR^5\,$, $\mfh_3\oplus \bR^2\,$, $\mathfrak{e}(2)\oplus \bR^2\,$, $A_{5,14}^0\,$.
\end{itemize}
\end{itemize}
\end{remark}
\section{Four-dimensional generalised Einstein Lorentzian Lie algebras with non-degenerate commutator ideal}\label{sec:4dLorentzian}
In this section, we provide a full classification of all four-dimensional generalised Lorentzian Lie algebras $(\mfg,H,\cG_g)$  which are generalised Einstein for zero divergence operator $\delta=0$ under the additional assumption that the commutator ideal $\mfg'$ is non-degenerate. 
More exactly, the result that we prove is the following:
\begin{theorem}\label{th:4dLorentzian}
Let $(\mfg,H,\cG_g)$ be a four-dimensional generalised Lorentzian Lie algebra such that the commutator ideal $\mfg'$ is non-degenerate. Then $(\mfg,H,\cG_g)$ is generalised Einstein for zero divergence if and only if either
\begin{itemize}
	\item[(i)] $\mfg$ is almost Abelian, $H=0$,  and $(\mfg,g)$ is as in Theorem \ref{th:4dalmostAbelian} (i), (ii) with $\sigma\neq 0$, (iii) or (vii) with $a\neq 0$ or $a=b_1=b_2=0$,
	\item[(ii)] or $\mfg\in \{\mathfrak{so}(3)\oplus \bR,\mathfrak{so}(2,1)\oplus \bR\}$ and $(H,g)$ is as in Theorem \ref{th:reductivecases}.
\end{itemize}
In particular, a four-dimensional Lie algebra $\mfg$ admits a generalised Lorentzian metric $(H,\cG_g)$ with non-degenerate commutator ideal $\mfg'$ which is generalised Einstein for zero divergence if and only if $\mfg$ is isomorphic to $\bR^4$, $e(2)\oplus \bR$, $e(1,1)\oplus \bR$, $A_{4,6}^{\sqrt{2}\cos(\varphi),\sin(\varphi)}$ for some $\varphi\in \bR$, $\mathfrak{r}_{3,1}'\oplus \bR$, $\mathfrak{so}(3)\oplus \bR$ or $\mathfrak{so}(2,1)\oplus \bR$.
\end{theorem}
\begin{proof}
We explain how the proof follows from the results that we obtain below.

For this, we distinguish the cases that $\mfg$ is solvable or not solvable. We note that if $\mfg$ is not solvable, then the list of all four-dimensional Lie algebras given in Table \ref{table:4d} in the appendix shows that $\mfg$ is isomorphic to $\mathfrak{so}(3)\oplus \bR$ or to $\mathfrak{so}(2,1)\oplus \bR$ and the revelant classification result on these Lie algebras is exactly Theorem \ref{th:reductivecases}.

If $\mfg$ is solvable and in fact almost Abelian, then we may use the classification result of all four-dimensional almost Abelian generalised Lorentzian Lie algebras which are generalised Einstein for zero divergence operator from Theorem \ref{th:4dalmostAbelian} and check which of those have non-degenerate commutator ideal. Doing this, we end up exactly with the cases (i), (ii) with $\sigma\neq 0$, (iii) and (vii) with $a\neq 0$ or $a=b_1=b_2=0$ in Theorem \ref{th:4dalmostAbelian}.

So what is left to show is that any four-dimensional solvable generalised Lorentzian Lie algebra $(\mfg,H,\cG_g)$ with non-degenerate commutator ideal which is generalised Einstein for zero divergence has to be almost Abelian. For this, we look at the list of all four-dimensional Lie algebras in Table \ref{table:4d} and note that if $\mfg$ is solvable but not almost Abelian, then it is either \emph{almost Heisenberg}, i.e. admits a codimension one ideal isomorphic to $\mfh_3\,$, or it is isomorphic to $\aff_{\bC}$ or $\aff_{\bR}\oplus \aff_{\bR}\,$. Now Corollary \ref{co:almostHeisenbergcommutatornondeg} shows that a four-dimensional almost Heisenberg $\mfg$ may only admit a generalised Lorentzian metric $(H,\cG_g)$ with $\mfg'$ being non-degenerate and which is generalised Einstein for $\delta=0$ if $\mfg\cong A_{4,9}^0\,$. Nevertheless, this case as well as the cases  $\mfg\cong \aff_{\bC}$ and $\mfg\cong \aff_{\bR}\oplus \aff_{\bR}$ are excluded by Theorem \ref{th:4dmissingcases}.
\end{proof}

\subsection{The almost Heisenberg case}
In this section, we are considering the situation that $\mfg$ is a four-dimensional almost nilpotent Lie algebra with codimension one ideal $\mfn$ isomorphic to the three-dimensional Heisenberg Lie algebra $\mfh_3\,$. To simplify our notation, we introduce the following terminology:
\begin{definition}
	An $n$-dimensional almost nilpotent Lie algebra is called \emph{almost Heisenberg} if it admits a codimension one ideal $\mfn$ isomorphic to $\mfh_3\oplus \bR^{n-4}$.
\end{definition}
\begin{remark}
	We note that in dimension four, all but four Lie algebras are almost Abelian or almost Heisenberg. The exceptions are the two reductive Lie algebras $\mathfrak{so}(3)\oplus \bR$ and $\mathfrak{so}(2,1)\oplus \bR$ and the two solvable Lie algebras $\mathfrak{aff}_{\bC}$ and $\mathfrak{aff}_{\bR}\oplus \mathfrak{aff}_{\bR}\,$.
\end{remark}
Moreover, we assume in this section that $\mfg$ is endowed with a Lorentzian metric $g$ such that $\mfn$ is non-degenerate. 
We investigate when $(\mfg,H,\cG_g)$ is generalised Einstein for zero divergence and note that by Corollary \ref{co:notAbeliannotdefinite}, then $\mfn$ has to be a Lorentzian subspace of $(\mfg,g)$. So we will assume this from now on.

Note that as a by-product, we obtain in Corollary \ref{co:almostHeisenbergcommutatornondeg} that only the almost Heisenbeg Lie algebra $A_{4,9}^0$ may admit a generalised Lorentzian metric which is non-degenerate on its commutator ideal and generalised Einstein for zero divergence. Recall that this result was used in the proof of our main theorem in this section, namely Theorem \ref{th:4dLorentzian}.

 We start by showing that for a four-dimensional almost Heisenberg generalised Einstein Lorentzian Lie algebra $(\mfg,H,\cG_g,\delta=0)$, the one-dimensional commutator ideal of $\mfn$ may not be negative definite:
\begin{lemma}\label{le:4dHeisenbergmfn'notnegdef}
	Let $(\mfg,H,\cG_g,\delta=0)$ be a four-dimensional almost Heisenberg generalised Einstein Lorentzian Lie algebra. Then the commutator ideal $\mfn'$ of the codimension one Heisenberg ideal $\mfn$ is not negative definite. 
\end{lemma}
\begin{proof}
	Assume the contrary and let $e_1\in \mfn'$ with $\epsilon_1:=g(e_1,e_1)=-1$. We extend $e_1$ to an orthonormal basis $(e_1,e_2,e_3,X)$ of $\mfg$ such that $(e_1, e_2,e_3)$ is an orthonormal basis of $\mfn$. Note that if $g(e_i,e_i):=\epsilon_i$ for $i=2,3$, we have $\epsilon_2=\epsilon_3=g(X,X)=1$. Note as well that there has to exist some $\lambda\in \bR^*$ such that $[e_2,e_3]=\lambda e_1\,$. Furthermore, $e_1$ is central in $\mfn$, which implies $\ad^{\mfn}_{e_1}=0$ and so $(\ad^{\mfn}_{e_1})^S=0$. Moreover,
	\begin{equation*}
		(\ad^{\mfn})^*(e_1)=\begin{pmatrix}
			0 & 0 & 0\\
			0 & 0 & \lambda \\
			0 & -\lambda & 0
		\end{pmatrix} \, ,
	\end{equation*}
	and so $g((\ad^{\mfn})^*(e_1),(\ad^{\mfn})^*(e_1))=2\lambda^2$.  Next, observe that $(\ad^{\mfn})^*(e_i)=0$ for $i=2,3$ as $e_2$ and $e_3$ are orthogonal to $\mfn'=\spa{e_1}$. Besides,
	\begin{equation*}
		2\, (\ad^{\mfn}_{e_2})^S=\begin{pmatrix}
			0 & 0 & \lambda \\
			0 & 0 & 0 \\
			-\lambda & 0 & 0
		\end{pmatrix} \, ,\qquad 2\, (\ad^{\mfn}_{e_3})^S=\begin{pmatrix}
			0 & -\lambda & 0 \\
			\lambda & 0 & 0 \\
			0 & 0 & 0
		\end{pmatrix} \, ,
	\end{equation*}
	and so $4\, g((\ad^{\mfn}_{e_i})^S,g((\ad^{\mfn}_{e_i})^S)=-2\lambda^2$ for $i=2,3$.
	
	Now write
	\begin{equation}
	\label{eq:rewritingHandB}
		H'=h\, e^{123} \, ,\quad B=b_1 e^{23}+b_2 e^{31}+b_3 e^{12} \, ,
	\end{equation}
	and note that, as an endomorphism, $B$ equals
	\begin{equation*}
	B= b_1 (e^2\otimes e_3-e^3\otimes e_2)- b_2 (e^3\otimes e_1+e^1\otimes e_3)+b_3 (e^1\otimes e_2+e^2\otimes e_1) \, .
	\end{equation*}
	Thus, by the fourth equation in Proposition \ref{pro:codim1idealnondeg}, we have
	\begin{equation*}
		0=g((\ad^{\mfn})^*(e_1),B)=-2  b_1 \lambda \, ,
	\end{equation*}
	and so $b_1=0$ since $\lambda\neq 0$.

     Now observe that
	\begin{align*}
	\sum_{i=1}^3 \epsilon_i g([f^*,f](e_i),e_i)&=	\sum_{i=1}^3 \epsilon_i (g(f(e_i),f(e_i))-g(f^*(e_u),f^*(e_i)))\\
	&=g(f,f)-g(f^*,f^*)=0 \, .
	\end{align*}
	Thus, the third equation in Proposition \ref{pro:codim1idealnondeg} yields
	\begin{equation*}
	\begin{split}
	0&=\sum_{i=1}^3 \epsilon_i \big(4 g((\ad^{\mfn}_{e_i})^S,(\ad^{\mfn}_{e_i})^S)-g((\ad^{\mfn})^*(e_i),(\ad^{\mfn})^*(e_i))\\
	&+2 g(H'(e_i,\cdot,\cdot),H'(e_i,\cdot,\cdot))+2 g(B(e_i,\cdot),B(e_i,\cdot)\big)\\
	&=-4\lambda^2+6g(H',H')+4 g(B,B)=-4\lambda^2-6h^2-4b_2^2-4b_3^2 \, ,
	\end{split}
	\end{equation*}
 which implies, in particular, $\lambda=0$, a contradiction. Hence, $\mfn'$ cannot be negative definite.
\end{proof}
In addition, $\mfn'$ can not be positive definite. The proof in this case is more involved:
\begin{proposition}\label{pro:4dHeisenbergmfn'notposdef}
	Let $(\mfg,H,\cG_g,0)$ be a four-dimensional almost Heisenberg generalised Einstein Lorentzian Lie algebra with codimension one Heisenberg ideal $\mfn$. Then, the commutator ideal $\mfn'$ is not positive definite.
\end{proposition}
\begin{proof}
    Assume $\mfn'$ is positive definite. We can then consider an orthonormal basis $(e_1,e_2,e_3)$ of $\mfn$ such that $g(e_1,e_1)=-g(e_2,e_2)=g(e_3,e_3)=1$ and $[e_2,e_3]=\lambda e_1$ for some $\lambda\in \bR^*$. This can be completed to an orthonormal basis $(e_1,e_2,e_3,X)$ of $\mfg$ with $g(X,X)=1$.

    Let $f=\ad_{X}\vert_{\mfn}$, and note that $f$ is a derivation of $\mfn$. As a result, $f$ must preserve $\mfn'$ and in the basis $(e_1,e_2,e_3)$ must take the form
    \begin{equation}
   \label{eq:fforalmostHeisenberg}
        f=\left(
\begin{array}{ccc}
 a & b & c \\
 0 & d & e \\
 0 & f & g \\
\end{array}
\right)\, ,
    \end{equation}
for some $a,b,c,d,e,f\in\bR$. In addition, $f$ must be compatible with the bracket $[e_2,e_3]=\lambda e_1\,$, which forces
 \begin{equation}
 \label{eq:compatibilityfwithbracket}
     f([e_2,e_3])=[f(e_2),e_3]+[e_2,f(e_3)]\implies \lambda(a-d-g)=0 \implies a-d-g=0 \, ,
 \end{equation}
where we have used $\lambda\neq 0$.
    
Since $(\mfg,H,\cG_g,0)$ is generalised Einstein, it must satisfy the equations in Proposition \ref{pro:codim1idealnondeg}. We will show this leads to a contradiction. To this end, we now compute all the quantities involved in the second, third and fourth equations in Proposition \ref{pro:codim1idealnondeg}. From $f$ it is immediate to obtain
    \begin{equation*}
        f^*=\left(
\begin{array}{ccc}
 a & 0 & 0 \\
 -b & d & -f \\
 c & -e & g \\
\end{array}
\right)\, , \qquad f^S=\left(
\begin{array}{ccc}
 a & \frac{b}{2} & \frac{c}{2} \\
 -\frac{b}{2} & d & \frac{e-f}{2} \\
 \frac{c}{2} & \frac{f-e}{2} & g \\
\end{array}
\right)\, ,
    \end{equation*}  
    and we can now compute the matrix $(g([f^*,f](e_i),e_j))_{i,j=1,2,3}$:
    \begin{equation*}
        \left(
\begin{array}{ccc}
 b^2-c^2 & a b-b d+c e & a c+b f-c g \\
 a b-b d+c e & b^2-e^2+f^2 & b c-(d-g) (e+f) \\
 a c+b f-c g & b c-(d-g) (e+f) & c^2-e^2+f^2 \\
\end{array}
\right) \, .
    \end{equation*}
    By the same reasoning as in the proof of Lemma \ref{le:4dHeisenbergmfn'notnegdef}, we have
    $\ad^{\mfn}_{e_1}=(\ad^{\mfn}_{e_1})^S=0$ and $(\ad^{\mfn})^*(e_i)=0$ for $i=2,3$, whereas
    \begin{equation*}
        2\left(\ad^\mfn_{e_2}\right)^S=\left(
\begin{array}{ccc}
 0 & 0 & \lambda \\
 0 & 0 & 0 \\
 \lambda & 0 & 0 \\
\end{array}
\right)\, , \qquad 2\left(\ad^\mfn_{e_3}\right)^S=\left(
\begin{array}{ccc}
 0 & -\lambda & 0 \\
 \lambda & 0 & 0 \\
 0 & 0 & 0 \\
\end{array}
\right)\, ,
    \end{equation*}
    and
    \begin{equation*}
        \left(\ad^\mfn\right)^*(e_1)=\left(
\begin{array}{ccc}
 0 & 0 & 0 \\
 0 & 0 & \lambda  \\
 0 & \lambda  & 0 \\
\end{array}
\right)\, .
    \end{equation*}
    Thus, we have the following non-zero terms contributing to the equations of Proposition \ref{pro:codim1idealnondeg}:
    \begin{align*}
        g(f^S,\left(\ad^\mfn_{e_2}\right)^S)&=\frac{c \lambda }{2} \,, & g(f^S,\left(\ad^\mfn_{e_3}\right)^S)&=\frac{b \lambda }{2} \,, & \\
         g(\left(\ad^\mfn_{e_2}\right)^S,\left(\ad^\mfn_{e_2}\right)^S)&=\frac{\lambda ^2}{2} \,, & g(\left(\ad^\mfn_{e_3}\right)^S,\left(\ad^\mfn_{e_3}\right)^S)&=-\frac{\lambda ^2}{2} \,, & \\
         g(\left(\ad^\mfn\right)^*(e_1),\left(\ad^\mfn\right)^*(e_1))&=-2 \lambda ^2 \,. & & &
    \end{align*}
	Writing $H'$ and $B$ as in \eqref{eq:rewritingHandB}, we obtain
    \begin{align*}
        g(B,H'(e_i,\cdot,\cdot))=-h \, b_i \, g(e_i,e_i) \,, \qquad
        g(H'(e_i,\cdot,\cdot),H'(e_j,\cdot,\cdot))=-h^2 \, g(e_i,e_j) \,,   \\
        (g(B(e_i,\cdot),B(e_j,\cdot)))_{i,j=1,2,3}=\left(
\begin{array}{ccc}
 b_2^2-b_3^2 & -b_1 b_2 & b_1 b_3 \\
 -b_1 b_2 & b_1^2+b_3^2 & -b_2 b_3 \\
 b_1 b_3 & -b_2 b_3 & b_2^2-b_1^2 \\
\end{array}
\right) \,,
    \end{align*}
    as well as
    \begin{equation*}
        g(\left(\ad^\mfn\right)^*(e_1),B^{\mathrm{End}})=-2\lambda b_1\, , \quad g(\left(\ad^\mfn\right)^*(e_2),B^{\mathrm{End}})=g(\left(\ad^\mfn\right)^*(e_3),B^{\mathrm{End}})=0\, .
    \end{equation*}       
    
Now, the fourth equation in Proposition \ref{pro:codim1idealnondeg} is non-trivial only for $Y=e_1$ and gives
\begin{equation*}
     0= -2\lambda b_1 \, .
\end{equation*}
Since $\lambda\neq 0$, this forces $b_1=0$ and the third equation in Proposition \ref{pro:codim1idealnondeg} for $Y\in\lbrace e_1,e_2,e_3\rbrace$ becomes
\begin{align}
 0&= \lambda^2 + b^2 -c^2 - h^2 +b_2^2-b_3^2 \, ,  \label{eq:almostheisenbergnpositive5}  \\
     0&= \lambda^2 + b^2 -e^2 +f^2 + h^2 +b_3^2 \, ,  \label{eq:almostheisenbergnpositive6}  \\
     0&= \lambda^2 - c^2 +e^2 -f^2 + h^2 -b_2^2 \, ,  \label{eq:almostheisenbergnpositive7}
\end{align}
Adding together \eqref{eq:almostheisenbergnpositive5}, \eqref{eq:almostheisenbergnpositive6} and \eqref{eq:almostheisenbergnpositive7} we find
\begin{equation}
\label{eq:combination567}
    0=3\lambda^2+2 b^2-2 c^2 + h^2 \, .
\end{equation}
This forces $c\neq 0$ and it is now useful to look at the second equation in Proposition \ref{pro:codim1idealnondeg} for $Y\in\lbrace e_2, e_3 \rbrace$:
\begin{equation*}
     0= c \lambda + h b_2 \, ,  \qquad     0= b \lambda -h b_3 \, , 
\end{equation*}
from the first equation we have $h\neq 0$ (and $b_2\neq 0$), and we can then use these equations to express $b_2$ and $b_3$ in terms of the other variables:
\begin{equation*}
    b_2=-\frac{c\lambda}{h} \, , \qquad b_3=+\frac{b\lambda}{h} \, .
\end{equation*}
Recall now that we require $H$ to be closed: this imposes additional conditions on the coefficients. To see this, we consider the structure equations of the four-dimensional algebra in terms of $\lambda$ and $f$
 \begin{align*}
     \dd e^1 &=-\lambda e^2\wedge e^3 - a X^\flat\wedge e^1 - b X^\flat\wedge e^2 - c X^\flat\wedge e^3 \, , \\
     \dd e^2 &=-d X^\flat\wedge e^2 - e X^\flat\wedge e^3 \, , \\
     \dd e^3 &=-f X^\flat\wedge e^2 - g X^\flat\wedge e^3 \, , \\
     \dd X^\flat &= 0 \, ,
 \end{align*}
and study the closedness conditions for the 3-forms
\begin{align*}
    \dd (e^1 \wedge e^2\wedge e^3)&=-(a+d+g) X^\flat\wedge e^1 \wedge e^2\wedge e^3 \, , \\
    \dd (X^\flat\wedge e^1 \wedge e^2) &= \dd (X^\flat\wedge e^3 \wedge e^1) = \dd (X^\flat\wedge e^2 \wedge e^3) = 0 \, .
\end{align*}
We conclude that we must either have $h=0$ or $a+d+g=0$, but we have just argued that $h\neq 0$, so $a+d+g=0$. This together with \eqref{eq:compatibilityfwithbracket} implies $a=0$ and $g=-d$. We now look at the polarization of the third equation in Proposition \ref{pro:codim1idealnondeg}, which simplifies to
\begin{equation}
\label{eq:almostheisenbergnpositivepolarized}
	-bd+ce=0\, , \qquad bf+cd=0\, , \qquad bc-b_2b_3-2d(e+f)=0\, .
\end{equation}
Since \eqref{eq:almostheisenbergnpositive6} forces $e\neq 0$, the first equation gives $d\neq 0$ and $b\neq 0$. We can then use the first two equations to obtain $e$ and $f$ in terms of the other variables
\begin{equation*}
    e= \frac{bd}{c} \, , \qquad
     f=  -\frac{cd}{b}  \, .
\end{equation*}
Replacing $b_2\,$, $b_3\,$, $e$ and $f$ in the last equation of \eqref{eq:almostheisenbergnpositivepolarized} we find
\begin{equation*}
	bc\left(1+\frac{\lambda^2}{h^2}\right)-2d^2\left(\frac{b}{c} -\frac{c}{b}\right)=0 \implies b^2c^2\left(1+\frac{\lambda^2}{h^2}\right)=2d^2\left(b^2 -c^2\right) \implies b^2>c^2 \, ,
\end{equation*}
where we have used that the left-hand side is positive. Note however that \eqref{eq:combination567} forces $c^2>b^2$, which is a contradiction and finishes the proof.
\end{proof}
Finally, we are left with considering the case of four-dimensional almost Heisenberg generalised Einstein Lorentzian Lie algebras with $\mfn$ being Lorentzian and $\mfn'$ being null. To formulate the classification result in this case, we use again a \emph{Witt} basis $(e_1,e_2,e_3,X)$ of $\mfg$, but now slightly differently defined than in Definition \ref{def:Wittbasis}: Namely, in our context,  a \emph{Witt basis} should be a basis $(e_1,e_2,e_3,X)$ of $\mfg$ such that $(e_1,e_2,e_3)$ is a basis of $\mfn$ with $e_1$ being a basis of $\mfn'$, $X$ is orthogonal to $\mfn$ with $g(X,X)=1$ and
\begin{equation*}
g(e_1,e_1)=g(e_2,e_2)=0 \, ,\quad g(e_1,e_2)=g(e_2,e_3)=0 \, ,\quad g(e_1,e_2)=g(e_3,e_3)=1 \, .
\end{equation*}

\begin{theorem}\label{th:4dHeisenberg}
	Let $(\mfg,H,\cG_g)$ be a four-dimensional almost Heisenberg generalised Lorentzian Lie algebra with codimension one Heisenberg ideal $\mfn$. Then $(\mfg,H,\cG_g)$ is generalised Einstein if and only if $\mfn'$ is null and there exists $(e_1,e_2,e_3,X)$ a Witt basis of $\mathfrak{g}$, satisfying $[e_2,e_3]=\lambda e_1$ for some $\lambda\neq 0$, such that $H=b_1 e^{23}\wedge X^\flat$ and $\ad_{X}\vert_{\mfn}=\left(
\begin{array}{ccc}
 0 & b & \pm\sqrt{f^2+b_1^2} \\
 0 & 0 & 0 \\
 0 & f & 0 \\
\end{array}
\right)\,$ for certain $b,f,b_1\in\mathbb{R}$.
\end{theorem}
\begin{proof}
	Lemma \ref{le:4dHeisenbergmfn'notnegdef} and Proposition \ref{pro:4dHeisenbergmfn'notposdef} imply that $(\mfg,H,\cG_g,\delta=0)$ may only be generalised Einstein if $\mfn'$ is null.
	
	So let us assume that for the rest of the proof and choose $(e_1,e_2,e_3,X)$ a Witt basis. Since $e_1$ spans $\mfn'$, $(e_1,e_2,e_3)$ span $\mfn$ and $\mfn\cong \mfh_3\,$, it is clear that there is some $\lambda\in \bR^*$ such that, up to anti-symmetry, $[e_2,e_3]=\lambda e_1$ is the only non-zero Lie bracket between elements of the basis $(e_1,e_2,e_3)$ of $\mfn$.

By the same arguments as in the proof of Proposition \ref{pro:4dHeisenbergmfn'notposdef}, we have that $f=\ad_{X}\vert_{\mfn}$ takes the form \eqref{eq:fforalmostHeisenberg} in the basis $(e_1,e_2,e_3)$ for some $a,b,c,d,e,f\in\bR$ satisfying $a-d-g=0$.

Again by the same line of reasoning as in the proof of Proposition \ref{pro:4dHeisenbergmfn'notposdef}, we have that $H$ is closed if either $h=0$ or $a+d+g=0$.

We have that $(\mfg,H,\cG_g,\delta=0)$ is generalised Einstein if and only if the equations of Proposition \ref{pro:codim1idealnondeg} are satisfied. We now compute these explicitly. In our basis:
    \begin{equation*}
        f^*=\left(
\begin{array}{ccc}
 d & b & f \\
 0 & a & 0 \\
 e & c & g \\
\end{array}
\right)\, , \qquad f^S=\left(
\begin{array}{ccc}
 \frac{a+d}{2} & b & \frac{c+f}{2} \\
 0 & \frac{a+d}{2} & \frac{e}{2} \\
 \frac{e}{2} & \frac{c+f}{2} & g \\
\end{array}
\right)\, ,
    \end{equation*}
    which give
    \begin{equation*}
        g(f^S,f^S)=\frac{1}{2} (a+d)^2+e (c+f)+g^2\, ,
    \end{equation*}
as well as the matrix $(g([f^*,f](e_i),e_j))_{i,j=1,2,3}$
    \begin{equation*}
     \begin{pmatrix}
        	-e^2 & - e c & e (a-g) \\
        	-e c & -2 a b+2 bd-c^2+f^2 & -a f+ be + dc-c g- f g\\
        	e(a-g) &  -a f+ be + dc-c g- f g & 2 ec
        \end{pmatrix} \, .
    \end{equation*}
Arguing as in Lemma \ref{le:4dHeisenbergmfn'notnegdef}, $\ad^\mfn_{e_1}=\left(\ad^\mfn_{e_1}\right)^S=0$. Similarly, in this case $e_1$ and $e_3$ are orthogonal to $\mfn'=\spa{e_1}$, so $(\ad^{\mfn})^*(e_i)=0$ for $i=1,3$. We thus have:
    \begin{equation*}
        2\left(\ad^\mfn_{e_2}\right)^S=\left(
\begin{array}{ccc}
 0 & 0 & \lambda \\
 0 & 0 & 0 \\
 0 & \lambda & 0 \\
\end{array}
\right)\, , \qquad \left(\ad^\mfn_{e_3}\right)^S=\left(
\begin{array}{ccc}
 0 & -\lambda  & 0 \\
 0 & 0 & 0 \\
 0 & 0 & 0 \\
\end{array}
\right)\, ,
    \end{equation*}
    as well as
    \begin{equation*}
        \left(\ad^\mfn\right)^*(e_2)=\left(
\begin{array}{ccc}
 0 & 0 & -\lambda \\
 0 & 0 & 0 \\
 0  & \lambda & 0 \\
\end{array}
\right)\, .
    \end{equation*}
    It turns out that $g((\ad^\mfn_{e_i})^S,(\ad^\mfn_{e_j})^S)=0$ and $g(\left(\ad^\mfn\right)^*(e_1),\left(\ad^\mfn\right)^*(e_1))=0$ for all $i,j\in\lbrace1,2,3\rbrace$, whereas
    \begin{equation*}
    	g(f^S,\left(\ad^\mfn_{e_1}\right)^S)=g(f^S,\left(\ad^\mfn_{e_3}\right)^S)=0 \, , \qquad 2 g(f^S,\left(\ad^\mfn_{e_2}\right)^S)=e \lambda \, .
    \end{equation*}    
	We now write $H'$ and $B$ as in \eqref{eq:rewritingHandB} and compute
 \begin{align*}
     g(B,B)&=-b_3^2-2b_1 b_2\,, \\
        g(B,H'(e_1,\cdot,\cdot))=-h \ b_2 \,, \quad 
        g(B,H'(e_2,\cdot,\cdot))&=-h \ b_1 \,, \quad
        g(B,H'(e_3,\cdot,\cdot))=-h \ b_3 \,, 
 \end{align*}
as well as
\begin{equation*}
\begin{split}
      (g(H'(e_i,\cdot,\cdot),H'(e_j,\cdot,\cdot)))_{i,j=1,2,3}&=\begin{pmatrix}
     0 & - h^2 & 0 \\
      -h^2 & 0 & 0\\
      0 &  0 & -h^2
  \end{pmatrix} \, , \\
(g(B(e_i,\cdot),B(e_j,\cdot)))_{i,j=1,2,3}&=\begin{pmatrix}
	b_2^2 & - b_1 b_2-b_3^2 & b_2 b_3 \\
	-b_1 b_2-b_3^2 & b_1^2 & b_1 b_3\\
	b_2 b_3 &  b_1 b_3 & -2 b_1 b_2
\end{pmatrix} \, .
\end{split}
    \end{equation*}
    Regarding the forms as endomorphisms, we find that $g(\left(\ad^\mfn\right)^*(e_i),H'_{e_j})=0$ for all $i,j\in\lbrace1,2,3\rbrace$ together with
    \begin{equation*}
        g(\left(\ad^\mfn\right)^*(e_1),B^{\mathrm{End}})=g(\left(\ad^\mfn\right)^*(e_3),B^{\mathrm{End}})=0\, , \qquad g(\left(\ad^\mfn\right)^*(e_2),B^{\mathrm{End}})=-2\lambda b_2 \, .
    \end{equation*}
    Finally, we need to compute
    \begin{equation*}
        (g(f^*(e_i),B(e_j,\cdot)^\sharp))_{i,j=1,2,3}=\left(
\begin{array}{ccc}
 -e b_2 & -d b_3 + e b_1 & d b_2 \\
 a b_3 -c b_2 &-b b_3+c b_1 & -a b_1 + b b_2 \\
 -g b_2 & -f b_3+ g b_1 & f b_2 \\
\end{array}
\right) \, .
    \end{equation*}
	Putting everything together, the equations in Proposition \ref{pro:codim1idealnondeg} (including the polarization of the third equation) are equivalent to
 \begin{align}
     0&= (a+d)^2 +2 e (c+f) + 2 g^2 - 2 b_1 b_2 - b_3^2 \, ,  \label{eq:almostheisenbergnnull1} \\
     0&= -h b_2 \, ,  \label{eq:almostheisenbergnnull2}  \\
     0&= e \lambda - h b_1 \, ,  \label{eq:almostheisenbergnnull3}  \\
     0&= -h b_3 \, ,  \label{eq:almostheisenbergnnull4}  \\
     0&= -e^2 +b_2^2 \, ,  \label{eq:almostheisenbergnnull5}  \\
     0&= ec +h^2+b_1 b_2+b_3^2 \, ,  \label{eq:almostheisenbergnnull6}  \\
      0&=e(a-g)+b_2 b_3 \, ,  \label{eq:almostheisenbergnnull7}  \\
     0&= -2ab+2bd-c^2+f^2 +b_1^2 \, ,  \label{eq:almostheisenbergnnull8}  \\
      0&= -af+be+cd-cg+fg+b_1 b_3 \, ,  \label{eq:almostheisenbergnnull9}  \\
     0&= -2ce + h^2 +2b_1b_2 \, ,  \label{eq:almostheisenbergnnull10}  \\
     0&= -2\lambda b_2 \, , \label{eq:almostheisenbergnnull11} \\
     0&= a b_3-c b_2+d b_3-e b_1 \, ,  \label{eq:almostheisenbergnnull12}  \\
     0&= d b_2 + g b_2 \, ,  \label{eq:almostheisenbergnnull13}  \\
     0&= a b_1 -b b_2 - f b_3 + g b_1 \, .  \label{eq:almostheisenbergnnull14}
 \end{align}
We now impose these equations. Since $\lambda\neq 0$, we observe that \eqref{eq:almostheisenbergnnull11} forces $b_2=0$ and as a result \eqref{eq:almostheisenbergnnull2} and \eqref{eq:almostheisenbergnnull13} are satisfied. Now \eqref{eq:almostheisenbergnnull5} requires $e=0$ which via \eqref{eq:almostheisenbergnnull3} implies $h=0$, ensuring that $H$ is closed and that \eqref{eq:almostheisenbergnnull4}, \eqref{eq:almostheisenbergnnull7} and \eqref{eq:almostheisenbergnnull10} hold. Then, from \eqref{eq:almostheisenbergnnull6} we obtain $b_3=0$ and  \eqref{eq:almostheisenbergnnull12} is trivially satisfied.

Recalling that $a=d+g$, \eqref{eq:almostheisenbergnnull1} can be rewritten as $(2d+g)^2+2g^2=0$, which forces $d=g=0$. Then $a=0$ as well, and so \eqref{eq:almostheisenbergnnull9} and \eqref{eq:almostheisenbergnnull14} hold. Thus, the only equation left is \eqref{eq:almostheisenbergnnull10}, which is equivalent to $c^2=f^2+b_1^2$ and yields the stated form of $f$ and $H$.
\end{proof}
\begin{remark}
We note that the four-dimensional almost Heisenberg Lie algebras $\mfg$ in Theorem \ref{th:4dHeisenberg} are also almost Abelian since $\spa{e_1,e_3,X\mp \frac{1}{\lambda}\sqrt{f^2+b_1^2} e_2}$ is a codimension one Abelian ideal. In fact, if $f\neq 0$, $\mfg$ isomorphic to $A_{4,1}\,$, the only indecomposable four-dimensional nilpotent Lie algebra, whereas in the case $f=0$, $\mfg$ is isomorphic to $\mfh_3\oplus \bR$.
\end{remark}
Note that by Table \ref{table:4d}, $A_{4,1}$ and $\mfh_3\oplus \bR$ are the only almost Heisenberg Lie algebras which are also almost Abelian. Moreover, by the same table, these two almost Heisenberg Lie algebras, together with the almost Heisenberg Lie algebra $A_{4,9}^0\,$, are exactly those almost Heisenberg Lie algebras where the codimension one Heisenberg ideal is not equal to the commutator ideal but has lower dimension. We note that by Theorem \ref{th:4dalmostAbelian}, the Lie algebras $\mfh_3\oplus \bR$ and $A_{4,1}$ do not admit generalised Lorentzian metrics with non-degenerate commutator ideal which are generalised Einstein for $\delta=0$. Hence, Theorem \ref{th:4dHeisenberg} implies:
\begin{corollary}\label{co:almostHeisenbergcommutatornondeg}
A four-dimensional almost Heisenberg Lie algebra $\mfg$ may only admit a generalised Lorenztian metric $(H,\cG_g)$ with non-degenerate commutator ideal which is generalised Einstein for $\delta=0$ if $\mfg$ is isomorphic to $A_{4,9}^0$.
\end{corollary}
We will show in Theorem \ref{th:4dmissingcases} that also $A_{4,9}^0$ cannot admit such a generalised Lorentzian metric, so no four-dimensional almost Heisenberg Lie algebra can.

However, before we do this, we turn our attention to the two non-solvable cases $\mathfrak{so}(3)\oplus \bR$ and $\mathfrak{so}(2,1)\oplus \bR$ as they may be treated with the same methods that we used before:
\subsection{The Lie algebras $\mathfrak{so}(3)\oplus \bR$ and $\mathfrak{so}(2,1)\oplus \bR$}
We begin by distilling out of Proposition \ref{pro:codim1idealnondeg} a characterisation of the generalised Einstein condition for zero divergence in our situation:
\begin{lemma}\label{le:reductiveeqns}
	Let $\mfn\in \{\mathfrak{so}(3),\mathfrak{so}(2,1)\}$ and write $\ad$ for the adjoint operator on $\mfn$. Moreover, let $(H,\cG_g)$ be a generalised pseudo-Riemannian metric on $\mfg=\mfn\oplus \bR$ such that $\mfn$ is non-degenerate and let $\epsilon:=g(X,X)\in \{-1,1\}$ for a normed element $X\in \mfn^{\perp}$.
	
	Then $(\mfg,H,\cG_g,\delta=0)$ is generalised Einstein if and only if $B=X\hook H=0$ and for $Y_0$ being the (not necessarily orthogonal) projection of $X$ to $\mfn$ along $\bR$ we have
	\begin{equation*}
		\begin{split}
			0&=g(\ad_{Y_0}^S,\ad_Y^S) \, ,\\
			0&=4 g(\ad_Y^S,\ad_Y^S)-g(\ad^*(Y),\ad^*(Y))+2\epsilon g(\ad_{Y_0}(Y),\ad_{Y_0}(Y))\\
			&-2\epsilon g(\ad^*_{Y_0}(Y),\ad^*_{Y_0}(Y))+2g(H(Y,\cdot,\cdot),H(Y,\cdot,\cdot)) \, ,\\
			g(\ad^*(Y),H_Z)&=g(\ad^*(Z),H_Y) \, ,
		\end{split}
	\end{equation*}
	for all $Y,Z\in \mfn$.
\end{lemma}
\begin{proof}
 By the definition of $Y_0\,$, we have $X=Y_0+\lambda$ for a unique $\lambda\in \bR$. We set, as usual, $f:=\ad^{\mfn\oplus \bR}_X|_{\mfn}$ and observe that $f=\ad^{\mfn\oplus \bR}_{Y_0}|_{\mfn}=\ad_{Y_0}$. Now we use the characterisation of the generalised Einstein condition for zero divergence from Proposition \ref{pro:codim1idealnondeg} and note that it coincides with our formulas if $B=0$.
	
	So we are left with showing that the validity of the conditions in Proposition \ref{pro:codim1idealnondeg} yields $B=0$. For this, we look at the fourth equation in Proposition \ref{pro:codim1idealnondeg}, which is given by
	\begin{equation*}
		0=g(\ad^*(Y),B) \, ,
	\end{equation*}
	for all $Y\in \mfn$. This equation implies $B=0$ if the map
	\begin{equation*}
		Y\mapsto \ad^*(Y) \, ,\quad \mfn\rightarrow \End^A(\mfn):=\{g\in \End(\mfn)|g^*=-g\} \, ,
	\end{equation*}
	from $\mfn$ into the anti-symmetric endomorphisms $\End^A(\mfn)$ of $\mfn$ is surjective. 
	
	To show this, we note that for any Lie algebra $\mfg$ the map $\mfg \ni Y \mapsto \ad^*(Y)\in\End^A(\mfg)$ is given by the transpose map
	\begin{equation*}
		[\cdot,\cdot]^t:\mfg\cong \mfg^*\rightarrow \Lambda^2 \mfg^*\cong \End^A(\mfg)
	\end{equation*}
	of the Lie bracket $[\cdot,\cdot]:\Lambda^2 \mfg\rightarrow \mfg$
	using the natural identifications $\mfg\cong \mfg^*$ and $\Lambda^2 \mfg^*\cong \End^A(\mfg)$ obtained from the metric.
	
	Now for the Lie algebras $\mfn$ under consideration, i.e. $\mfn\in\lbrace\mathfrak{so}(3),\mathfrak{so}(2,1)\rbrace$, the Lie bracket $[\cdot,\cdot]:\mfn\rightarrow \Lambda^2 \mfn$ is an isomorphism. Thus, also $\mfn\ni Y\mapsto \ad^*(Y)\in \End^A(\mfn)$ is an isomorphism, and so, in particular, surjective. This concludes the proof.
\end{proof}
Next, we address solving the equations in Lemma \ref{le:reductiveeqns}. We observe that they fully reduce to equations on $\mfn$, however these are not the generalised Einstein equations for $(\mfn,H|_{\mfn},\cG_{g|_{\mfn}},\delta=0)$ yet. This would be the case if $\ad_{Y_0}^S=0$ as then the first equation in Lemma \ref{le:reductiveeqns} would be automatically fulfilled whereas the others would correspond exactly to the generalised Einstein equations for zero divergence for $(\mfn,H|_{\mfn},\cG_{g|_{\mfn}},\delta=0)$, cf. Corollary \ref{co:gEdelta=0}. We will show that this reduction always takes place. Note that in the case where $\mfn$ is Riemannian we already know this by Corollary \ref{co:codim1fskewB=0} (b).

However, if $\mfn$ has Lorentzian signature, the situation is more complicated. In order to prove the reduction, we take an approach similar to one taken by the first author and David Krusche in \cite{CK}. For this note that, by \cite[Lemma 3.1]{CK}, there is some symmetric endomorphism $L\in \End^S(\mfn)$ of $\mfn$ such that
\begin{equation*}
[Y,Z]=L(Y\times Z) \, ,
\end{equation*}
where $\times:\mfn\rightarrow \mfn$ is the crossproduct on the Lorentzian vector space $(\mfn,g)$, uniquely defined by
\begin{equation*}
e_1\times e_2=e_3 \, ,\qquad e_2\times e_3=-e_1 \, ,\qquad e_3\times e_1=e_2 \, ,
\end{equation*}
for any orthornormal basis $(e_1,e_2,e_3)$ of $\mfn$ with $g(e_1,e_1)=-g(e_2,e_2)=-g(e_3,e_3)=-1$. We note that in our situation, i.e. $\mfn\in \{\mathfrak{so}(3),\mathfrak{so}(2,1)\}$, $L$ has to be bijective.

 The idea to show that the reduction takes place is now to use the canonical forms from Lemma \ref{le:canonicalforms} (a) for the symmetric endomorphism $L$.
 
  We start with the cases where $L$ is of first or second type:
 \begin{lemma}\label{le:reductive1st2ndtype}
 Let $\mfn\in \{\mathfrak{so}(3),\mathfrak{so}(2,1)\}$ and $(H,\cG_g)$ be a generalised Lorentzian metric on $\mfn\oplus \bR$ such that $\mfn$ is Lorentzian. Let $L\in \End^S(\mfn)$, $X\in \mfn^{\perp}$ and $Y_0$ be as above and assume that $L$ is of first or second type. Then $(\mfn,H,\cG_g,\delta=0)$ is generalised Einstein if and only if $B=X\hook H=0$, $\ad_{Y_0}^S=0$ and $(\mfn,H|_{\mfn},\cG_{g|_{\mfn}}, \delta=0)$ is generalised Einstein.
 \end{lemma}
\begin{proof}
We will consider in both cases the subspace $U:=\spa{\ad_Y^S|Y\in \mfn}$ of $\End^S(\mfn)$ and will show that it is always a non-degenerate subspace of $\End^S(\mfn)$. However, by the first equation in Lemma \ref{le:reductiveeqns}, the symmetric endomorphism $\ad_{Y_0}^S$ lies in $U$ and is orthogonal to $U$, which forces $\ad_{Y_0}^S=0$ due to $U$ being non-degenerate. As explained above, this then gives the stated assertion. So let us dig into the two different cases:
\begin{itemize}[wide]
	\item $L$ of first type:
	
	Then $L=\diag(a_1,a_2,a_3)$ for certain $a_1,a_2,a_3\in \bR^*$ with respect to an orthonormal basis $(e_1,e_2,e_3)$ of $\mfn$ with $g(e_1,e_1)=-g(e_2,e_2)=-g(e_3,e_3)=-1$. One then computes
	\begin{equation*}
	\begin{split}
	\ad_{e_1}^S&=\frac{a_2-a_3}{2} \begin{pmatrix} 0 & 0 & 0 \\ 0 & 0 & -1 \\ 0 & -	1 & 0 \end{pmatrix} \, ,\quad 
	\ad_{e_2}^S=\frac{a_3-a_1}{2} \begin{pmatrix} 0 & 0 & 1 \\ 0 & 0 & 0 \\ -1 & 0 & 0 \end{pmatrix} \, ,\\
		\ad_{e_3}^S&=\frac{a_1-a_2}{2} \begin{pmatrix} 0 & 1 & 0 \\ -1 & 0 & 0 \\ 0 & 0 & 0 \end{pmatrix} \, .
	\end{split}
	\end{equation*}
   We see that $g(\ad_{e_i}^S,\ad_{e_j}^S)=0$ for $i,j\in \{1,2,3\}$ with $i\neq j$ and that $g(\ad_{e_i}^S,\ad_{e_i}^S)=0$ if and only if $\ad_{e_i}^S=0$. This shows that $U$ is non-degenerate independently of its dimension.
   \item $L$ of second type:
   
   Then $L=\diag(L_1(\alpha,\beta),a)$ for certain $\alpha\in \bR$, $\beta>0$ and $a\in \bR^*$ with respect to an orthonormal basis $(e_1,e_2,e_3)$ of $\mfn$ with $g(e_1,e_1)=-g(e_2,e_2)=-g(e_3,e_3)=-1$. Then
   	\begin{equation*}
   	\begin{split}
   		\ad_{e_1}^S&=\frac{1}{2} \begin{pmatrix} 0 & 0 & \beta \\ 0 & 0 & a-\alpha \\ -\beta & a-\alpha & 0 \end{pmatrix} \, ,\qquad 
   		\ad_{e_2}^S=\frac{1}{2} \begin{pmatrix} 0 & 0 & a-\alpha \\ 0 & 0 & -\beta \\ \alpha-a & -\beta & 0 \end{pmatrix} \, ,\\
   		\ad_{e_3}^S&= \begin{pmatrix} -\beta & 0 & 0 \\ 0 & \beta & 0 \\ 0 & 0 & 0 \end{pmatrix} \, ,
   	\end{split}
   \end{equation*}
and one sees that, due to $\beta\neq 0$, $\ad_{e_1}^S,\ad_{e_2}^S,\ad_{e_3}^S$ are linearly independent, i.e $\dim(U)=3$. On the other hand,
\begin{equation*}
(g(\ad_{e_i}^S,\ad_{e_j}^S))_{i,j=1,2,3}=\begin{pmatrix}
\frac{(a-\alpha)^2-\beta^2}{2}	& -\beta (a-\alpha) & 0 \\
-\beta (a-\alpha) &  -\frac{(a-\alpha)^2-\beta^2}{2} & 0 \\
0 & 0 & 2\beta^2	
\end{pmatrix} \, ,
\end{equation*}
and this matrix has determinant $-\frac{\beta^2 ((a-\alpha)^2+\beta^2)^2}{2}\neq 0$ as $\beta\neq 0$. Hence, $U$ is non-degenerate.
\end{itemize}	
\end{proof}
Next, we show that $L$ cannot be of third or fourth type:
 \begin{lemma}\label{le:reductive3rd4thtype}
	Let $\mfn\in \{\mathfrak{so}(3),\mathfrak{so}(2,1)\}$ and let $(\mfn\oplus \bR, H,\cG_g,\delta=0)$ be a generalised Einstein Lorentzian metric such that $\mfn$ is non-degenerate. If $L\in \End^S(\mfn)$ is defined as above, then $L$ can be neither of third type nor of fourth type.
\end{lemma}
\begin{proof}
	We argue by contradiction and distinguish the two possible types of $L$:
\begin{itemize}[wide]
	\item 
Let $L$ first be of third type, i.e. $L=\diag(L_2(\alpha,\epsilon),a)$ for certain $\alpha,a\in \bR^*$ and $\epsilon\in \{-1,1\}$ with respect to an orthonormal basis $(e_1,e_2,e_3)$ of $\mfn$ with $g(e_1,e_1)=-g(e_2,e_2)=-g(e_3,e_3)=-1$. In order to simplify the expressions below, we replace $\alpha$ for $\epsilon\alpha$ in what follows, that is $L=\diag(L_2(\epsilon\alpha,\epsilon),a)$. One computes then
\begin{equation*}
	\begin{split}
		\ad_{e_1}^S&=\frac{1}{4}\begin{pmatrix} 0 & 0 & -\epsilon \\
			0 & 0 & 2 (a-\alpha)+\epsilon \\
			\epsilon & 2 (a-\alpha)+\epsilon & 0
		\end{pmatrix} \, , \\
		\ad_{e_2}^S&=\frac{1}{4}\begin{pmatrix} 0 & 0 & -\epsilon +2 (a-\alpha) \\
			0 & 0 & \epsilon \\
		\epsilon -2 (a-\alpha) & \epsilon & 0
		\end{pmatrix} \, , \qquad
		\ad_{e_3}^S=\frac{1}{2}\begin{pmatrix} \epsilon & \epsilon & 0\\
			-\epsilon & -\epsilon & 0 \\
			0 & 0 & 0
		\end{pmatrix} \, ,
	\end{split}
\end{equation*}
as well as
\begin{equation*}
	\begin{split}
		\ad^*(e_1)&=\frac{1}{2}\begin{pmatrix} 0 & 0 & \epsilon \\
			0 & 0 & -\epsilon-2\alpha \\
			\epsilon & \epsilon+2\alpha & 0
		\end{pmatrix} \, , \qquad
		\ad^*(e_2)=\frac{1}{2}\begin{pmatrix} 0 & 0 & \epsilon -2 \alpha \\
			0 & 0 & -\epsilon \\
			\epsilon -2 \alpha & \epsilon & 0
		\end{pmatrix} \, , \\
		\ad^*(e_3)&=\begin{pmatrix} 0 & a & 0\\
			a & 0 & 0 \\
			0 & 0 & 0
		\end{pmatrix} \, .
	\end{split}
\end{equation*}
This yields
\begin{equation*}
\begin{split}
G:=4\Bigl(g(\ad_{e_i}^S,\ad_{e_j}^S)\Bigr)_{i,j=1,2,3}&=2{\small \begin{pmatrix} (a-\alpha) (a-\alpha+\epsilon) & \epsilon (a-\alpha) & 0 \\ \epsilon (a-\alpha) & -(a-\alpha) (a-\alpha-\epsilon) & 0 \\ 0 & 0 & 0 \end{pmatrix}} \, ,\\
-\Bigl(g(\ad^*(e_i)),\ad^*(e_j)\Bigr)_{i,j=1,2,3}&=2 \begin{pmatrix} -\alpha (\alpha+\epsilon) & -\epsilon \alpha& 0 \\ -\epsilon \alpha & \alpha(\alpha-\epsilon) & 0 \\ 0 & 0 & a^2 \end{pmatrix} \, .
\end{split}
\end{equation*}
By Lemma \ref{le:reductiveeqns}, we know that $B=X\hook H=0$ and, for dimensional reasons, $H=h\, e^{123}$ for some $h\in \bR$. Thus,
\begin{equation*}
\Bigl(g(H(e_i,\cdot,\cdot),H(e_j,\cdot,\cdot))\Bigr)_{i,j=1,2,3}=\diag(h^2,-h^2,-h^2) \, .
\end{equation*}
We now show that we must have $a=\alpha$. Assume that $a\neq \alpha$:  then we see that $\ad_{e_1}^S,\ad_{e_2}^S,\ad_{e_3}^S$ are linearly independent and that the rank of $G$ is two. Consequently, as $\ad_{Y_0}^S$ has to be orthogonal to $U:=\spa{\ad_{e_1}^S,\ad_{e_2}^S,\ad_{e_3}^S}$, we must have $Y_0\in \spa{e_3}$, i.e. $Y_0=\lambda e_3$ for some $\lambda\in \bR$ here. One then computes that
\begin{equation*}
2\Bigl(g(\ad_{Y_0}(e_i),\ad_{Y_0}(e_j))-g(\ad^*_{Y_0}(e_i),\ad^*_{Y_0}(e_j))\Bigr)=-4\lambda^2 \epsilon \alpha \begin{pmatrix}
	1 & 1 & 0\\ 1 & 1 & 0 \\  0 & 0 & 0
\end{pmatrix} \, .
\end{equation*}
and the second equation in Lemma \ref{le:reductiveeqns} is equivalent to the vanishing of the matrix
\begin{equation*}
\left(
\begin{smallmatrix}
	-4\alpha \epsilon (\lambda^2+1)+2a\epsilon+2a^2-4\alpha a+2h^2 & -4 \epsilon \alpha (\lambda^2+1)+2\epsilon a & 0 \\
	-4 \epsilon \alpha (\lambda^2+1)+2\epsilon a & -4\alpha \epsilon (\lambda^2+1)+2a\epsilon-2a^2+4\alpha a-2h^2 & 0 \\
	0 & 0 & 2a^2-2h^2
\end{smallmatrix}\right) \, .
\end{equation*}
Thus, $a^2=h^2$ and so substracting the $(2,2)$-entry from the $(1,1)$-entry, we get $0=4a^2-8\alpha a+4h^2=8a^2-8\alpha a$. As $a\neq 0$, this implies $\alpha=a$, a contradiction.

Hence, we must have $a=\alpha$. Then, $G=0$ and so $Y_0$ can be arbitrary in $\mfn$. We write $Y_0=\sum_{i=1}^3 \lambda_i e_i$ for $\lambda_1,\lambda_2,\lambda_3\in \bR$ and compute that then
the insertion of $e_1$ and $e_3$ into the (polarised) second equation in Lemma \ref{le:reductiveeqns} yields
\begin{equation*}
0=3\epsilon \alpha (\lambda_1+\lambda_2)\lambda_3 \, .
\end{equation*}
Thus, $\lambda_2=-\lambda_1$ or $\lambda_3=0$.

If $\lambda_2=-\lambda_1\,$, then putting $e_1$ and $e_2$ into the second equation in Lemma \ref{le:reductiveeqns} gives
\begin{equation*}
	0=-2\epsilon \alpha (2\lambda_3^2+1) \, ,
\end{equation*} 
a contradiction as $a=\alpha\neq 0$.

Finally, if $\lambda_2\neq -\lambda_1$ but $\lambda_3=0$, then inserting the pairs $(e_1,e_1)$, $(e_1,e_2)$ and $(e_3,e_3)$ into the (polarised) second equation in Lemma \ref{le:reductiveeqns} yields
\begin{equation*}
\begin{split}
0&=-\frac{(\lambda_1+\lambda_2)^2}{2}-2a\epsilon \left(\lambda_1\lambda_2+\lambda_2^2+1\right)-2 a^2+2h^2 \, ,\\
0&=-\frac{(\lambda_1+\lambda_2)^2}{2}+a\epsilon (\lambda_1^2-\lambda_2^2-2) \, ,\\
0&=2a^2-2h^2-2a\epsilon (\lambda_1+\lambda_2)^2	 \, .
\end{split}	
\end{equation*}
Subtracting from the second equation both the first and the third one, we arrive at
\begin{equation*}
0=3\epsilon a (\lambda_1+\lambda_2)^2 \, ,
\end{equation*}
a contradiction since $\epsilon\in\{-1,1\}$, $a\neq 0$ and $\lambda_2\neq -\lambda_1$.

Thus, $L$ cannot be of third type.
\item
Let $L$ be of fourth type, i.e. $L=L_3(\alpha)$ for certain $\alpha\in \bR^*$ with respect to an orthonormal basis $(e_1,e_2,e_3)$ of $\mfn$ with $g(e_1,e_1)=-g(e_2,e_2)=-g(e_3,e_3)=-1$. One computes that then
\begin{equation*}
\begin{split}
\ad_{e_1}^S&=\frac{\sqrt{2}}{4}\begin{pmatrix} 0 &  0 & 1 \\ 0 & 2 & 0 \\ -1 & 0 & -2 \end{pmatrix} \, ,\qquad
 \ad_{e_2}^S=\frac{\sqrt{2}}{4}\begin{pmatrix} 0 &  1 & 0 \\ -1 & 0 & -1 \\ 0 & -1 & 0 \end{pmatrix} \, ,\\
 \ad_{e_3}^S&=\frac{\sqrt{2}}{4}\begin{pmatrix} -2 &  0 & -1 \\ 0 & 2 & 0 \\ 1 & 0 & 0 \end{pmatrix} \, ,
 \end{split}
\end{equation*}
and so
\begin{equation*}
4\Bigl(g(\ad_{e_i}^S,\ad_{e_j}^S)\Bigr)_{i,j=1,2,3}=\begin{pmatrix} 3 & 0 & 3 \\ 0 & 0 & 0 \\ 3 & 0 & 3
\end{pmatrix} \, .
\end{equation*}
Thus, $Y_0=\lambda_1 (e_1-e_3)+\lambda_2 e_2$ for certain $\lambda_1,\lambda_2\in \bR$. Then a lengthy but straightforward compuation shows that inserting $(e_1,e_3)$ and $(e_2,e_3)$ into the polarised second equation in Lemma \ref{le:reductiveeqns} gives
\begin{equation*}
\begin{split}
0&=4+\lambda_2^2+6\sqrt{2}\,\alpha \lambda_1\lambda_2 \, ,\\
0&=\sqrt{2}\alpha(\lambda_2^2+2) \, .
\end{split}
\end{equation*}
From the second equation we get $\alpha=0$ which, when inserted into the first equation, gives $4+\lambda_2^2=0$, a contradiction. Hence, $L$ can neither be of fourth type.
\end{itemize}
\end{proof}
Putting our results together and using previous results in three dimensions from \cite{CK}, we may now show:
\begin{theorem}\label{th:reductivecases}
Let $\mfn\in \{\mathfrak{so}(3),\mathfrak{so}(2,1)\}$ and $(H,\cG_g)$ be a generalised Lorentzian metric on $\mfg=\mfn\oplus \bR$ such that $\mfn$ is non-degenerate. Then $(\mfg,H,\cG_g)$ is generalised Einstein for zero divergence operator $\delta=0$ if and only if one of the following conditions is satisfied:
\begin{itemize}
\item[(i)] $\mfn=\mathfrak{so}(3)$ is Riemannian and there exist $a\in \bR^*$, $b\in \bR$ and $\epsilon\in \{-1,1\}$ and an orthonormal basis $(e_1,\ldots,e_4)$ of $\mathfrak{so}(3)\oplus \bR$ such that $(e_1,e_2,e_3)$ is a basis of $\mathfrak{so}(3)$,
$H=a e^{123}$ and, up to anti-symmetry, the only non-zero Lie brackets are given by
		\begin{equation*}
	[e_1,e_2]=\epsilon a\, e_3 \, ,\;\;\; [e_2,e_3]=\epsilon a\, e_1 \, ,\;\;\; [e_3,e_1]=\epsilon a\, e_2 \, ,\;\;\; [e_3,e_4]=b\, e_2 \, ,\;\;\; [e_4,e_2]=b\, e_3 \, ,
\end{equation*}
\item[(ii)] or $\mfn=\mathfrak{so}(2,1)$ is Lorentzian and there exist $a\in \bR^*$ and $\epsilon\in \{-1,1\}$ and an orthonormal basis $(e_1,\ldots,e_4)$ of $\mathfrak{so}(2,1)\oplus \bR$ such that $(e_1,e_2,e_3)$ is a basis of $\mathfrak{so}(2,1)$ with $g(e_1,e_1)=-1$,
$H=a e^{123}$ and, up to anti-symmetry, the only non-zero Lie brackets are given by
\begin{equation*}
[e_1,e_2]=\epsilon a\, e_3 \, ,\quad [e_2,e_3]=-\epsilon a\, e_1 \, ,\quad [e_3,e_1]=\epsilon a\, e_2 \, ,
\end{equation*}
and by exactly one of the following:
\begin{itemize}
	\item[($\alpha$)] $[e_4,e_2]=b\, e_3\,,\quad [e_4,e_3]=-b\, e_2$ for some $b\in \bR$,
	\item[($\beta$)] or $[e_4,e_1]=b\, e_2\,,\quad [e_4,e_2]=b\, e_1$ for some $b\in \bR^*$,
	\item[($\gamma$)] or $[e_4,e_1]=e_2\,,\quad [e_4,e_2]=e_1+e_3\,,\quad [e_4,e_3]=-e_2\,$.
\end{itemize}
\end{itemize}
\end{theorem}
\begin{proof}
By Lemma \ref{le:reductive1st2ndtype} and Lemma \ref{le:reductive3rd4thtype}, we know that the generalised Einstein condition for zero divergence is equivalent to $L$ being of first or second type, $B=X\hook H=0$, $\ad_{Y_0}^S=0$---where $Y_0$ is the projection of $X\in \mfn^{\perp}$ to $\mfn$ along $\bR$---and $(\mfn,H'=H|_{\mfn},\cG_{g|_{\mfn}})$ being generalised Einstein for zero divergence operator. 

Now we may apply the classification of generalised Einstein Lie algebras from \cite[1. in Theorem 3.4]{CK} to deduce that if $\mfn$ is Riemannian, then $\mfn=\mathfrak{so}(3)$ and $\mfn$ admits an orthonormal basis $(e_1,e_2,e_3)$ of $\mfn$ such that $H=H'=a e^{123}$ for some $a\in \bR^*$ and there exists $\epsilon\in \{-1,1\}$ such that, up to anti-symmetry, the only-non-zero Lie brackets between elements in $(e_1,e_2,e_3)$ are given by
\begin{equation*}
	[e_1,e_2]=\epsilon a\, e_3 \, ,\quad [e_2,e_3]=\epsilon a\, e_1 \, ,\quad [e_3,e_1]=\epsilon a\, e_2 \, .
\end{equation*}
Now we may apply an special orthogonal transformation to assume that $Y_0=b e_1$ for some $b\in \bR$ and so, setting $e_4:=X$, we arrive at the Lie brackets
\begin{align*}
[e_4,e_1]&=\ad_{Y_0} (e_1)=b [e_1,e_1]=0 \, , & & \\
[e_4,e_2]&=b [e_1,e_2]=b e_3 \, , &  [e_4,e_3]&=b [e_1,e_3]=-b e_2 \, .
\end{align*}
This is case (i) in the statement.

In the case that $\mfn$ is Lorentzian, \cite[1. in Theorem 3.4]{CK} yields that $\mfn=\mathfrak{so}(2,1)$ and $\mfn$ admits an orthonormal basis $(e_1,e_2,e_3)$ of $\mfn$ with $g(e_1,e_1)=-1$ such that $H=H'=a e^{123}$ for some $a\in \bR^*$ and there exists $\epsilon\in \{-1,1\}$ such that, up to anti-symmetry, the only-non-zero Lie brackets between elements in $(e_1,e_2,e_3)$ are given by
\begin{equation*}
	[e_1,e_2]=\epsilon a\, e_3 \, ,\quad [e_2,e_3]=-\epsilon a\, e_1 \, ,\quad [e_3,e_1]=\epsilon a\, e_2 \, .
\end{equation*}
Now depending on whether $Y_0$ is negative definite (or $0$), positive definite or null, we may apply a special orthogonal transformation to get that either $Y_0=\frac{b}{\epsilon a} e_1$ for some $b\in \bR$, $Y_0=\frac{b}{\epsilon a} e_3$ for some $b\in \bR^*$ or $Y_0=\frac{1}{\epsilon a}(e_1+e_3)$, which gives the cases $(ii) (\alpha)$, $(ii) (\beta)$ or $(ii) (\gamma)$, respectively.
\end{proof}
We note that the generalised Lorentzian metrics from Theorem \ref{th:reductivecases} are not only Einstein for zero divergence but for a larger class of possible divergence operators:
\begin{corollary}\label{co:reductivecasesdeltaneq0}
\begin{enumerate}[(a)]
	\item The generalised Riemannian metrics $(H,\cG_g)$ from Theorem \ref{th:reductivecases} (i) are generalised Einstein for divergence operator $\delta\in E^*$ if and only if $b\neq 0$, $\delta(e_2)=\delta(e_3)=\delta(e^2)=\delta(e^3)=0$ and $\delta(e_1+\epsilon\, e^1)=0$ or if $b=0$ and $\delta(e_i+\epsilon\, e^i)=0$ for all $i=1,2,3$.
	\item The generalised Lorentzian metrics $(H,\cG_g)$ from Theorem \ref{th:reductivecases} (ii) $(\alpha)$ are generalised Einstein for divergence operator $\delta\in E^*$ if and only if $b\neq 0$, $\delta(e_2)=\delta(e_3)=\delta(e^2)=\delta(e^3)=0$ and $\delta(e_1-\epsilon\, e^1)=0$ or if $b=0$ and $\delta(e_i+\epsilon\, e^i)=0$ for all $i=2,3$ and $\delta(e_1-\epsilon\, e^1)=0$.
		\item The generalised Lorentzian metrics $(H,\cG_g)$ from Theorem \ref{th:reductivecases} (ii) $(\beta)$ are generalised Einstein for divergence operator $\delta\in E^*$ if and only if $\delta(e_1)=\delta(e_2)=\delta(e^1)=\delta(e^2)=0$ and $\delta(e_3+\epsilon\, e^3)=0$.
			\item The generalised Lorentzian metrics $(H,\cG_g)$ from Theorem \ref{th:reductivecases} (ii) $(\gamma)$ are generalised Einstein for divergence operator $\delta\in E^*$ if and only if $\delta(e_2)=\delta(e^2)=0$, $\delta(e_3)=-\delta(e_1)$, $\delta(e^3)=\delta(e^1)$ and $\delta(e_1-\epsilon\, e^1)=0$.
\end{enumerate}
\end{corollary}
\begin{proof}
Let $(\mfg,H,\cG_g)$ be one of the generalised pseudo-Riemannian metrics from Theorem \ref{th:reductivecases} which are generalised Einstein for zero divergence. Then, in all cases,
the adjoint operators $\ad_X\,$, $X\in \mfg$, are skew-symmetric and so $\ad_X^*(X)=0$ for all $X\in \mfg$. Thus, by Corollary \ref{co:gEdeltanonzero} (a), $(\mfg,H,\cG_g,\delta)$ is generalised Einstein if and only if
\begin{equation*}
\delta([X,Y]+H(X,Y,\cdot))=0 \, ,\qquad \delta(H(X,Y,\cdot)^\sharp+[X,Y]^\flat)=0 \, ,
\end{equation*}
for all $X,Y\in \mfg$. Working these equations explictly out in the different cases in Theorem \ref{th:reductivecases}, we obtain the claimed result.
\end{proof}
\subsection{Generalised Einstein Lie algebras with a codimension two ideal and zero divergence}

So far we have studied Lie algebras with a distinguished ``nice'' codimension one ideal. However, not every Lie algebra has a ``nice'' (or any) codimension one ideal. In particular, in four dimensions the Lie algebras $\mathfrak{aff}_{\bC}$ and $\mathfrak{aff}_{\bR}\oplus \mathfrak{aff}_{\bR}$ have nice ideals of codimension at most two, which are more difficult to work with.

With this motivation in mind, in this section we study generalised Einstein Lie algebras with a codimension two ideal $\mfn$, assuming that the metric is non-degenerate on $\mfn$, and vanishing divergence operator.

Let $\mfg$ be a Lie algebra with a codimension two ideal $\mfn$ and a metric $g$ that is non-degenerate on $\mfn$. We can then choose elements $X_1,X_2\in\mfg$ orthogonal to $\mfn$ satisfying $g(X_1,X_2)=0$ and $g(X_i,X_i)=\epsilon_i\in\lbrace -1,1 \rbrace$, and write $[X_1,X_2]=a X_1+b X_2+u$ with $a,b\in\mathbb{R}$ and $u\in\mfn$. We set $f_i:=\ad(X_i)|_{\mfn}$ and consider a closed 3-form $H$, which can be decomposed as 
\begin{equation*}
    H=H'+X_1^\flat\wedge B_1+X_2^\flat\wedge B_2+X_1^\flat\wedge X_2^\flat\wedge C \, ,
\end{equation*}
where $H'\in\Lambda^3\mfn^*$, $B_1,B_2\in\Lambda^2\mfn^*$ and  $C\in\mfn^*$. We then have the following generalization of Proposition \ref{pro:codim1idealnondeg}:
\begin{proposition}\label{pro:codim2}
Let $(\mfg,H,\cG_g)$ be a generalised pseudo-Riemannian Lie algebra with non-degenerate codimension two ideal $\mfn$. Let $\epsilon_i\,$, $a$, $b$, $X_i\,$, $u$, $C$, $B_i\,$, $H'$ and $f_i$ be as above with $i\in\lbrace1,2\rbrace$. Then $(\mfg,H,\cG_g,0)$ is generalised Einstein if and only if
\begin{align*}
0&=2\, g(f_1^S,f_1^S)+ g(B_1,B_1)+2b^2+\epsilon_2\left( g(u,u) + g(C,C) \right) \, ,\\
0&=2\, g(f_2^S,f_2^S)+ g(B_2,B_2)+2a^2+\epsilon_1\left( g(u,u) + g(C,C) \right) \, ,\\
0&=2\, g(f_1^S,f_2^S)+ \epsilon_1 \epsilon_2\, g(B_1,B_2)-2ab \, ,\\
0&=2\,g(f_1^S,(\ad_Y^{\mfn})^S)+g(B_1,H'(Y,\cdot,\cdot))\\
&-\epsilon_2\left( a u^\flat (Y)+ g(u,f_2(Y)) + \epsilon_1\, g(C,B_2(Y,\cdot)) \right) \, ,\\
0&=2\,g(f_2^S,(\ad_Y^{\mfn})^S)+g(B_2,H'(Y,\cdot,\cdot))\\
&-\epsilon_1\left( b u^\flat (Y)- g(u,f_1(Y)) - \epsilon_2\, g(C,B_1(Y,\cdot)) \right) \, ,\\
0&=4\, g((\ad_Y^\mfn)^S,(\ad_Y^{\mfn})^S)-g((\ad^\mfn)^*(Y),(\ad^{\mfn})^*(Y))+ 2\, g(H'(Y,\cdot,\cdot),H'(Y,\cdot,\cdot))\\
&+2\epsilon_1\, g([f^*_1,f_1](Y),Y)+ 2\epsilon_1\, g(B_1(Y,\cdot),B_1(Y,\cdot))\\
&+2\epsilon_2\, g([f^*_2,f_2](Y),Y)+ 2\epsilon_2\, g(B_2(Y,\cdot),B_2(Y,\cdot)) + 2\,\epsilon_1\epsilon_2\left( C(Y)^2-u^\flat(Y)^2 \right) \, ,\\
0&=g((\ad^{\mfn})^*(Y),B_1)+ 2\epsilon_1\, g(C^\sharp,f_2^*(Y))- 2\epsilon_1 a\, C(Y) \, ,\\
0&=g((\ad^{\mfn})^*(Y),B_2)- 2\epsilon_2\, g(C^\sharp,f_1^*(Y))- 2\epsilon_2 b\, C(Y) \, ,\\
0&=g((\ad^{\mfn})^*(Y),H'_Z)-2 g(f_1^*(Y),B_1(Z,\cdot)^{\sharp})-2 g(f_2^*(Y),B_2(Z,\cdot)^{\sharp})\\
& - g((\ad^{\mfn})^*(Z),H'_Y)+2 g(f_1^*(Z),B_1(Y,\cdot)^{\sharp})+2 g(f_2^*(Z),B_2(Y,\cdot)^{\sharp})\\
&+ 2 u^\flat(Y) C(Z) - 2 u^\flat(Z) C(Y) \, ,
\end{align*}
for all $Y,Z\in \mfn$, where we consider $B_1$ and $B_2$ in the seventh and eight equations as elements of $\End(\mfn)$.
\end{proposition}

\begin{proof}
    We need to substitute into the equations of Corollary \ref{co:gEdelta=0}. For that, we first compute the different operators and metrics. Given $X_1\,$, $X_2$ and an arbitrary vector $Y\in\mfn$, we can write in the basis $\lbrace X_1, X_2 ,\mfn \rbrace$:  
    \begin{align*}
        \ad_{X_1}&=\left(
\begin{array}{ccc}
 0 & a & 0 \\
 0 & b & 0 \\
 0 & u & f_1 \\
\end{array}
\right)\, , & \ad_{X_2}&=\left(
\begin{array}{ccc}
 -a & 0 & 0  \\
 -b & 0 & 0 \\
 -u & 0 & f_2 \\
\end{array}
\right)\, , & \\
\ad_{Y}&=\left(
\begin{array}{ccc}
 0 & 0  & 0 \\
 0 & 0 & 0 \\
 -f_1(Y) & -f_2(Y) & \ad^\mfn_{Y} \\
\end{array}
\right)\, . &&&
    \end{align*}
    When computing the transpose, we pick up factors of $\epsilon_1$ and $\epsilon_2$:
    \begin{align*}
        \ad_{X_1}^*&=\left(
\begin{array}{ccc}
 0 & 0 & 0 \\
 \frac{\epsilon_1}{\epsilon_2} a & b & \frac{1}{\epsilon_2} u^\flat \\
 0 & 0 & f_1^* \\
\end{array}
\right)\, , & \ad_{X_2}^*&=\left(
\begin{array}{ccc}
 -a & -\frac{\epsilon_2}{\epsilon_1}b & -\frac{1}{\epsilon_1} u^\flat  \\
 0 & 0 & 0 \\
 0 & 0 & f_2^* \\
\end{array}
\right)\, , & \\
\ad_{Y}^*&=\left(
\begin{array}{ccc}
 0 & 0  & -\frac{1}{\epsilon_1}f_1(Y)^\flat \\
 0 & 0 & -\frac{1}{\epsilon_2}f_2(Y)^\flat \\
 0 & 0 & (\ad^\mfn_{Y})^* \\
\end{array}
\right)\, , &&&
    \end{align*}
    and we obtain
    \begin{align*}
        \ad_{X_1}^S&=\left(
\begin{array}{ccc}
 0 & \frac{1}{2}a & 0 \\
 \frac{1}{2}\frac{\epsilon_1}{\epsilon_2} a & b & \frac{1}{2}\frac{1}{\epsilon_2} u^\flat \\
 0 & \frac{1}{2}u & f_1^S \\
\end{array}
\right)\, , \qquad \ad_{X_2}^S=\left(
\begin{array}{ccc}
 -a & -\frac{1}{2}\frac{\epsilon_2}{\epsilon_1}b & -\frac{1}{2}\frac{1}{\epsilon_1} u^\flat  \\
 -\frac{1}{2} b & 0 & 0 \\
 -\frac{1}{2} u & 0 & f_2^S \\
\end{array}
\right)\, , \\
\ad_{Y}^S&=\left(
\begin{array}{ccc}
 0 & 0  & -\frac{1}{2\,\epsilon_1}f_1(Y)^\flat \\
 0 & 0 & -\frac{1}{2\, \epsilon_2}f_2(Y)^\flat \\
 -\frac{1}{2}f_1(Y) & -\frac{1}{2}f_2(Y) & (\ad^\mfn_{Y})^S \\
\end{array}
\right)\, .
    \end{align*}
    We also need
    \begin{align*}
        \ad^*(X_1)&=\left(
\begin{array}{ccc}
 0 & -a & 0 \\
 \frac{\epsilon_1}{\epsilon_2} a & 0 & 0 \\
 0 & 0 & 0 \\
\end{array}
\right)\, , \qquad \ad^*(X_2)=\left(
\begin{array}{ccc}
 0 & -\frac{\epsilon_2}{\epsilon_1}b & 0  \\
 b & 0 & 0 \\
 0 & 0 & 0 \\
\end{array}
\right)\, , \\
\ad^*(Y)&=\left(
\begin{array}{ccc}
 0 & -\frac{1}{\epsilon_1} u^\flat(Y)  & -\frac{1}{\epsilon_1}f_1^*(Y)^\flat \\
 \frac{1}{\epsilon_2} u^\flat(Y) & 0 & -\frac{1}{\epsilon_2}f_2^*(Y)^\flat \\
 f_1^*(Y) & f_2^*(Y) & (\ad^\mfn)^*(Y) \\
\end{array}
\right)\, ,
    \end{align*}
    Evaluating the three-form $H$ results in
    \begin{align*}
        H(X_1,\cdot,\cdot)&= \epsilon_1  \left(B_1+X_2^\flat\wedge C\right)\, ,\qquad
        H(X_2,\cdot,\cdot)= \epsilon_2  \left(B_2-X_1^\flat\wedge C\right)  ,\\
        H(Y,\cdot,\cdot)&=H'(Y,\cdot,\cdot)-X_1^\flat\wedge B_1(Y,\cdot)-X_2^\flat\wedge B_2(Y,\cdot)+ C(Y) \, X_1^\flat\wedge X_2^\flat \, ,
    \end{align*}
    whereas as endomorphisms these would take the form
    \begin{align*}
        H_{X_1}&=\left(
\begin{array}{ccc}
 0 & 0 & 0 \\
 0 & 0 & -  \epsilon_1\, C \\
 0 &  \epsilon_1 \epsilon_2 \, C^\sharp & B_1 \\
\end{array}
\right)\, , \qquad H_{X_2}=\left(
\begin{array}{ccc}
 0 & 0 &   \epsilon_2 \,C \\
 0 & 0 & 0 \\
  - \epsilon_1 \epsilon_2 \,C^\sharp & 0 & B_2 \\
\end{array}
\right)\, , \\
H_{Y}&=\left(
\begin{array}{ccc}
 0 & -   \epsilon_2\,C(Y)  & B_1(Y,\cdot) \\
   \epsilon_1\,C(Y) & 0 & B_2(Y,\cdot) \\
 -    \epsilon_1\, B_1(Y,\cdot)^\sharp & -   \epsilon_2\,B_2(Y,\cdot)^\sharp & H'_Y \\
\end{array}
\right)\, ,
    \end{align*}
    where $B_1$ and $B_2$ should be understood as endomorphisms in these expressions.

    These ingredients can now be inserted in the equations of \eqref{eq:gEdelta=0} and, after computing the norms with the appropriate factors of $\epsilon_1$ and $\epsilon_2\,$, the result follows. More precisely, substituting $X_1$ and $X_2$ in the first equation of \eqref{eq:gEdelta=0} yields the first two equations of Proposition \ref{pro:codim2}, and its polarization with $X_1$ and $X_2$ yields the third equation. Similarly, the polarizations for $X_1\,$, $Y$ and $X_2\,$, $Y$  give the fourth and fifth equations, whereas simply substituting $Y$ gives the sixth equation. Regarding the second equation of \eqref{eq:gEdelta=0}, it is trivially satisfied for the pair $(X_1,X_2)$, so the only non-trivial pairs are $(X_1,Y)$, $(X_2,Y)$ and $(Y,Z)$ for $Z\in\mfn$. The resulting conditions correspond to the last three equations of Proposition \ref{pro:codim2}. 
\end{proof}

Proposition \ref{pro:codim2} has interesting consequences already in the Lorentzian setting. For example, we obtain the following:
\begin{corollary}\label{cor:codim2nRiemannian}
    Let $(\mfg,H,\cG_g,0)$ be a generalised Einstein Lorentzian Lie algebra with non-degenerate codimension two Riemannian ideal $\mfn$. Then $\mfg=\mathbb{R}^2\ltimes\mfn$ for an antisymmetric $\mathbb{R}^2$ action, $H=H'=H\vert_\mfn$ and $(\mfn,H',\cG_g\vert_n,0)$ is generalised Einstein.
\end{corollary}
\begin{proof}
    Consider a basis of $\mfg$ as in Proposition \ref{pro:codim2}, and we can assume that $\epsilon_1=-\epsilon_2=-1$ without any loss of generality. From the first equation in Proposition \ref{pro:codim2} we deduce $f_1^S=0$, $B_1=0$, $b=0$, $u=0$ and $C=0$. From the second equation we further obtain $f_2^S=0$, $B_2=0$ and $a=0$. This shows that $[X_1,X_2]=0$ and both of them act antisymmetrically on $\mfn$. In addition, $H=H'$ and the sixth and ninth equation in Proposition \ref{pro:codim2} reduce to the generalised Einstein equations on $\mfn$. All the other equations are trivially satisfied.
\end{proof}
Note that, in the previous Corollary, one can also apply Theorem \ref{th:structRiemGE} to $\mfn$ and observe that the generalised Einstein condition on $\mfn$ further reduces to the commutator ideal $[\mfn,\mfn]$.

\subsection{The Lie algebras $\mathfrak{aff}_{\bC}\,$,  $\mathfrak{aff}_{\bR}\oplus \mathfrak{aff}_{\bR}$ and $A_{4,9}^0\,$}

We focus on four-dimensional algebras admitting Abelian codimension two ideals in this subsection. We are particularly interested in the Lie algebras $\mathfrak{aff}_{\bC}$, $\mathfrak{aff}_{\bR}\oplus \mathfrak{aff}_{\bR}$ and $A_{4,9}^0\,$, which are not almost
Abelian, and the first two also not almost Heisenberg.

Now what all these three Lie algebras have in common is that they are a semidirect product of the form $\mathbb{R}^2\rtimes\mathbb{R}^2$, where the first $\mathbb{R}^2$ factor is the commutator ideal of the entire Lie algebra and $\mathbb{R}^2\rtimes\mathbb{R}^2$ has no center. In fact, one may observe from Table \ref{table:4d} that these properties characterise the Lie algebras $\mathfrak{aff}_{\bC}\,$, $\mathfrak{aff}_{\bR}\oplus \mathfrak{aff}_{\bR}$ and $A_{4,9}^0$ uniquely among all four-dimensional Lie algebras. 

We are now able to prove:
\begin{theorem}\label{th:4dmissingcases}
Let $(\mfg,H,\cG_g)$ be a four-dimensional generalised Riemannian or Lorentzian Lie algebra such that the commutator ideal $\mfn:=\mfg'$ is two-dimensional, Abelian and non-degenerate and such that $\mfg$ has trivial center. Then $(\mfg,H,\cG_g,\delta=0)$ cannot be generalised Einstein.

Hence, $\mathfrak{aff}_{\bC}\,$, $\mathfrak{aff}_{\bR}\oplus \mathfrak{aff}_{\bR}$ and $A_{4,9}^0$ all do not admit any Riemannian or Lorentzian generalised Einstein metric $(H,\cG_g)$ with non-degenerate commutator ideal for zero divergence operator $\delta=0$.
\end{theorem}
\begin{proof}
We assume by contradiction that $(H,\cG_g)$ is generalised Einstein for zero divergence operator. We are, thus, in the situation of Proposition \ref{pro:codim2} with the codimension two Abelian ideal $\mfn$, which implies that $\ad^\mfn_{Y}=(\ad^\mfn_{Y})^S=(\ad^\mfn)^*(Y)=0$ for all $Y\in\mfn$. Moreover, since $\mfn$ is the commutator ideal, we have $a=b=0$. Furthermore, we have $H'=0$ for dimensional reasons.

We first assume that $\mfn$ is Riemannian. Then, by Corollary \ref{cor:codim2nRiemannian}, $H=H'=0$ and $\mfh:=\mfn^{\perp}$ acts skew-symmetrically on $\mfn$. As the space of skew-symmetric endomorphims on a two-dimensional Riemannian vector space is one-dimensional, there is some $0\neq X\in \mfh$ which has to act trivially on $\mfn$. Extend $X$ by $\tilde{X}$ to a basis of $\mfh$. As $\mfg$ has trivial center, $[X,\tilde{X}]\in \mfn\setminus \{0\}$ and, since $\mfn$ is the commutator ideal of $\mfg$, $\tilde{X}$ has to act non-trivially on $\mfn$. This implies that $\tilde{X}$ has full rank and so there is some $Y\in \mfn$ such that $[\tilde{X},Y]=[X,\tilde{X}]$. Therefore, $[X+Y,\tilde{X}]=0$ which shows that $X+Y$ is central, a contradicition to our assumptions. 

Hence, we may assume from now on that $\mfn$ is Lorentzian. Then $\mfh$ is Riemannian. As the space of two-forms on $\mfn$ is one-dimensional, we may choose an orthonormal basis $X_1\,$, $X_2$ of $\mfh$ such that $B_2=0$. Writing $B_1=b\, e^{12}$ with $b\in \bR$ and $(e_1,e_2)$ of $\mfn$ being an orthonormal basis of $(e_1,e_2)$ with $g(e_1,e_1)=-g(e_2,e_2)=-1$, we note that the equations in Proposition \ref{pro:codim2} reduce in our case to
\begin{align}
	0&= 2 g(f_1^S,f_1^S)-b^2+g(u,u)+g(C,C) \, , \label{eq:4dcodim21} \\
	0&= 2 g(f_2^S,f_2^S)+g(u,u)+g(C,C) \, , \label{eq:4dcodim22}  \\
	0&= g(f_1^S,f_2^S) \, , \label{eq:4dcodim23}  \\
	0&= g(u,f_2(Y)) \, , \label{eq:4dcodim24}  \\
	0&= g(u,f_1(Y))+g(C,B_1(Y,\cdot)) \, , \label{eq:4dcodim25}  \\
	0&= g([f_1^*,f_1](Y),Y)+ g([f_2^*,f_2](Y),Y)\nonumber \\
	&+g(B_1(Y,\cdot),B_1(Y,\cdot))+ C(Y)^2-g(u,Y)^2 \, , \label{eq:4dcodim26}  \\
	0&= g(f_2(C^\sharp),Y) \, , \label{eq:4dcodim27}  \\
	0&= g(f_1(C^\sharp),Y) \, , \label{eq:4dcodim28}  \\
	0&= -g(f_1^*(Y),B_1(Z,\cdot)^\sharp)+g(f_1^*(Z),B_1(Y,\cdot)^\sharp)+g(u,Y)C(Z)-g(u,Z)C(Y) \, , \label{eq:4dcodim29} 
\end{align}
for all $Y,Z\in \mfn$.

From \eqref{eq:4dcodim27} and \eqref{eq:4dcodim28}, we obtain $C^\sharp\in \ker(f_1)\cap \ker(f_2)$. However, $\ker(f_1)\cap \ker(f_2)$ is a subspace of the center of $\mfg$, which, by assumption, is trivial. Hence, $C=0$. Then, \eqref{eq:4dcodim24}  and \eqref{eq:4dcodim25} are equivalent to $u\in (\im(f_1)+\im(f_2))^\perp$.

For the rest of the proof, we distinguish the cases $\im(f_1)+\im(f_2)=\mfn$ and $\im(f_1)+\im(f_2)\neq \mfn$ and show that both assumptions give a contradiction:
\begin{itemize}[wide]
\item $\im(f_1)+\im(f_2)=\mfn$:
Then $u\in \mfn^{\perp}=\{0\}$, i.e. $u=0$. Since $[f_i^*,f_i]$ is symmetric and trace-free, we must have
\begin{equation*}
(g([f_i^*,f_i](e_j),e_k))_{j,k=1,2}=\begin{pmatrix}
	c_i & d_i \\  d_i  & c_i \end{pmatrix} \, ,
\end{equation*}
for certain $c_i,d_i\in \bR$. This means that \eqref{eq:4dcodim26} is equivalent to
\begin{equation*}
0=c_1+c_2+b^2 \, ,\qquad 0=d_1+d_2 \, , \qquad 0=c_1+c_2-b^2 \, ,
\end{equation*}
which is equivalent to $c_2=-c_1\,$, $d_2=-d_1$ and $b=0$. Thus,
$[f_2^*,f_2]=-[f_1^*,f_1]$ and $B_1=0$.

Now \eqref{eq:4dcodim21} -- \eqref{eq:4dcodim23} are given by $g(f_i^S,f_j^S)=0$ for all $i,j=1,2$. This implies that $\spa{f_1^S,f_2^S}$ is a totally isotropic subspace of the three-dimensional Lorentzian vector space of symmetric endomorphisms of $\mfn$. Thus, $\dim(\spa{f_1^S,f_2^S})\leq 1$. We note that we must have $\dim(\spa{f_1^S,f_2^S})=1$ as otherwise both $f_1$ and $f_2$ would be skew-symmetric, so a multiple of each other and, due to $u=0$, there is some non-zero central element in $\mfh=\spa{X_1,X_2}$, a contradiction to our assumptions.

Since $B_1=B_2=0$, we may rotate in the space $\mfh=\spa{X_1,X_2}$ and assume, w.l.o.g., that $f_2^S=0$, and so $f_1^S\neq 0$. We note
that $f_1^S$ cannot be of first type as then, by Lemma \ref{le:trfsquare=0}, we must have $f_1^S=0$, a contradiction.
Now write, with respect to $(e_1,e_2)$, 
\begin{equation*}
	f^S_1=\begin{pmatrix} a & -b \\ b & c \end{pmatrix}\,,\qquad f^A_1=\begin{pmatrix} 0 & \rho \\ \rho & 0 \end{pmatrix} \, ,
\end{equation*}
for certain $a,b,c,\rho\in \bR$ with $b\neq 0$, and observe that
\begin{equation*}
0=[f_2^S,f_2^A]=\frac{1}{2}[f_2^*,f_2]=\frac{1}{2}[f_1^*,f_1]=[f_1^S,f_1^A]=\rho \begin{pmatrix} -2 b & -(c-a) \\ c-a & 2b \end{pmatrix}
\end{equation*}
implies $\rho=0$, i.e. $f_1^A=0$. Thus, $f_1=f_1^S$ is symmetric and $f_2=f_2^A$ is anti-symmetric. Now the Jacobi identity for triples $(X_1,X_2,Y)$ with $Y\in \mfn$ yields that $[f_1,f_2]=0$. However, $0=[f_1,f_2]=[f_1^S,f_2^A]$ and $f_1^S$ not being of first type implies, argueing as above, that $f_2=f_2^A=0$, a contradiction as then $X_2$ would be central. Hence, the case
$\im(f_1)+\im(f_2)=\mfn$ cannot occur.

\item $\im(f_1)+\im(f_2)\neq \mfn$: Setting $U:=\im(f_1)+\im(f_2)$ and observing that $\mfn=\mfg'=U+\spa{u}$, we must have $\dim(U)=1$ and $\spa{u}$ being complementary to $U$. Now since $u$ is orthogonal to $U$, this implies that both $U$ and $\spa{u}$ are definite subspaces of $\mfn$.

Let us first consider the case that $U$ is negative definite and so $u$ is positive definite, i.e. $U=\spa{e_1}$ and $u=\tilde{u} e_2$ for $(e_1,e_2)$ being an orthonormal basis of $\mfn$ with $g(e_1,e_1)=-g(e_2,e_2)=-1$. The endomorphism $f_i\in \End(\mfn)$, $i=1,2$, may then be written as
\begin{equation*}
f_i=\begin{pmatrix} c_i & d_i \\ 0 & 0 \end{pmatrix} \, ,
\end{equation*}
for certain $c_i,d_i\in \bR$ with respect to $(e_1,e_2)$. Inserting $e_2$ into \eqref{eq:4dcodim26} yields
\begin{equation*}
0=-d_1^2-d_2^2-b^2-\tilde{u}^2 \, .
\end{equation*} 
Hence, $d_1=d_2=b=\tilde{u}=0$, i.e. $B_2=0$, $u=0$ and $f_1$ and $f_2$ are, in particular linearly dependent. However, this implies that there is some non-zero central element $X\in \mfh=\spa{X_1,X_2}$, which is a contradiction to our assumptions.

So let us finally consider the case that $U$ is positive definite and $\spa{u}$ is negative definite. We choose now an orthonormal basis $(e_1,e_2)$ of $\mfn$ such that $g(e_1,e_1)=-g(e_2,e_2)=-1$ and such that $u=\tilde{u}\, e_1$ for some $\tilde{u}\in \bR$ and $U=\spa{e_2}$. Hence, there are, for $i=1,2$, $c_i,d_i\in \bR$ such that
\begin{equation*}
	f_i=\begin{pmatrix} 0 & 0 \\ c_i & d_i \end{pmatrix} \, ,
\end{equation*}
 with respect to $(e_1,e_2)$. If we insert now $e_1$ into \eqref{eq:4dcodim26}, we obtain
 \begin{equation*}
 	0=c_1^2+c_2^2+b^2+\tilde{u}^2 \, ,
 \end{equation*}
and so $c_1=c_2=b=\tilde{u}=0$. Again, $u=0$ and $f_1$ and $f_2$ are linearly dependent, which again yields a contradiction to our assumption that $\mfg$ has trivial center.

Thus, also the case $\im(f_1)+\im(f_2)\neq \mfn$ cannot occur.
\end{itemize}
\end{proof}
\section*{Appendix}
Table \ref{table:4d} gives a list of all four-dimensional Lie algebras. The table is further subdivided according to whether the Lie algebra is unimodular or not.  The names for the Lie algebras in the first column come from \cite{PSWZ} with the exception of the names for the affine Lie algebra of motions of the real line $\aff_{\bR}$ and the complex line $\aff_{\bC}\,$, which were called $\mfr_2$ or $A_{4,12}$ in \cite{PSWZ}, respectively.
The presentation is self-contained as the second column of the table encodes the Lie bracket of $\mfg$ by giving the exterior differentials $(de^1,\ldots, de^4)$ of the dual basis of a basis $(e_1,\ldots,e_4)$ of $\mfg$. The column labelled ``$\mfa$'' contains all isomorphism classes of unimodular codimension one ideals in $\g$. So if for some Lie algebra $\mfg$, this column contains $\bR^3$, $\mfg$ is almost Abelian, while if the column contains $\mfh_3\,$, $\mfg$ is almost Heisenberg.

 The next column, labelled $\mfg' =[\mfg , \mfg ]$ contains the commutator ideal of $\mfg$. Finally, in the last column, we write $\yes \yes$ if $\mfg$ admits a generalised Lorentzian metric with non-degenerate commutator ideal which is generalised Einstein for $\delta=0$. If $\mfg$ is almost Abelian and it does admit a generalised Lorentzian metric which is generalised Einstein for $\delta=0$ but none which also has non-degenerate commutator ideal, we put only one $\yes$ into the last column. The entry ``$\times$'' in the last column means that there is no generalised Lorentzian metric which is generalised Einstein for $\delta=0$ while the entry ``$-$'' means that there is no such generalised metric with non-degenerate commutator ideal. 
 
 \tabcolsep=0.11cm
\begin{longtable}[ht]{ccccl}
	\caption{Four-dimensional Lie algebras} \\ \hline
	$\g$ & Lie bracket & $\mfa$ &  $\g'$ & gE \\
	\hline
	\endhead
	\label{table:4d}
 & \multicolumn{2}{c}{unimodular} &  \\ \hline
	$\mathfrak{so}(3)\op \bR$ & $(e^{23},-e^{13},e^{12},0)$ & $\mathfrak{so}(3)$ & $\mathfrak{so}(3)$ & \yes \yes \\
	$\mathfrak{so}(2,1)\op \bR$ & $(e^{23},e^{13},e^{12},0)$ & $\mathfrak{so}(2,1)$ &  $\mathfrak{so}(2,1)$ & \yes \yes \\
	$e(2)\op \bR$ & $(e^{23},-e^{13},0,0)$ & $\bR^3,\,e(2)$ & $\bR^2$  & \yes \yes \\
	$e(1,1)\op \bR$ & $(e^{23},e^{13},0,0)$ & $\bR^3,\,e(1,1)$ & $\bR^2$ & \yes \yes \\
	$\h_3\op \bR$ & $(e^{23},0,0,0)$ & $\bR^3$, $\h_3$ & $\bR$ & \yes \\
	$\bR^4$ & $(0,0,0,0)$ & $\bR^3$ & $\{0\}$ & \yes\yes \\
	$A_{4,1}$ & $(e^{24},e^{34},0,0)$ & $\bR^3,\,\h_3$ & $\bR^2$ &  \yes \\
	$A_{4,2}^{-2}$ & $(-2 e^{14},e^{24}+e^{34},e^{34},0)$ & $\bR^3$ & $\bR^3$ & $\times$ \\
	$A_{4,5}^{\alpha,-(\alpha+1)}$ & $( e^{14},\alpha e^{24},-(\alpha+1) e^{34},0)$  & $\bR^3$ & $\bR^3$ & $\times$ \\
	& $-1<\alpha\leq -1/2$  & & \\
	$A_{4,6}^{\alpha,-\alpha/2}$ & $( \alpha e^{14},-\frac{\alpha}{2} e^{24}+e^{34},-\frac{\alpha}{2} e^{34}-e^{24},0)$ & $\bR^3$ & $\bR^3$ & $\times$ \\
	& $\alpha>0$ & &  \\
	$A_{4,8}$ & $(e^{23},e^{24},-e^{34},0)$ & $\h_3$ & $\h_3$ & -- \\
	$A_{4,10}$ & $(e^{23},e^{34},-e^{24},0)$ & $\h_3$ & $\h_3$ & -- \\
	\hline & \multicolumn{2}{c}{non-unimodular} & \\ \hline
	$\aff_{\bR}\op \bR^2$ & $(e^{14},0,0,0)$  & $\bR^3$ & $\bR$ & $\times$ \\
	$\mathfrak{r}_3\op \bR$ & $(e^{14}+e^{24},e^{24},0,0)$  & $\bR^3$ & $\bR^2$ & $\times$ \\
	$\mathfrak{r}_{3,\mu}\op \bR$ & $(e^{14},\mu e^{24},0,0)$, $-1<\mu\leq 1$, $\mu\neq 0$ & $\bR^3$ & $\bR^2$ & $\times$\\
	$\mathfrak{r}'_{3,\mu}\op \bR$ & $(\mu e^{14}+e^{24},-e^{14}+\mu e^{24},0,0)$, $\mu>0$ & $\bR^3$ & $\bR^2$ &\yes\yes \footnote{Only for $\mu=1$.} \\
	$A_{4,2}^{\alpha}$ & $(\alpha e^{14},e^{24}+e^{34},e^{34},0)$, $\alpha\neq 0,-2$ & $\bR^3$ & $\bR^3$ & $\times$ \\
	$A_{4,3}$ & $( e^{14},e^{34},0,0)$ & $\bR^3$ & $\bR^2$  & $\times$\\
	$A_{4,4}$ & $(e^{14}+e^{24},e^{24}+e^{34},e^{34},0)$ & $\bR^3$ & $\bR^3$  & $\times$ \\
	$A_{4,5}^{\alpha,\beta}$ & $( e^{14},\alpha e^{24},\beta e^{34},0)$ & $\bR^3$ &  $\bR^3$ & $\times$ \\
	&  $-1<\alpha\leq \beta\leq 1$, $\alpha\beta\neq 0$, $\beta\neq -(\alpha+1)$ & & \\
	$A_{4,6}^{\alpha,\beta}$ & $( \alpha e^{14},\beta e^{24}+e^{34},\beta e^{34}-e^{24},0)$ &  $\bR^3$ &  $\bR^3$ &\yes \yes \footnote{Only for $\alpha^2+2\beta^2=1$.}  \\
	& $\alpha>0$, $\beta\neq-\alpha/2$ & & & \\
	$A_{4,7}$ & $( 2 e^{14}+e^{23},e^{24}+e^{34},e^{34},0)$ & $\h_3$ & $\h_3$ & -- \\
	$A_{4,9}^{\alpha}$ & $((\alpha+1) e^{14}+e^{23},e^{24},\alpha e^{34},0)$, & & & \\
	&  $\alpha\in (-1,1]$, $\alpha\neq 0$ & $\h_3$ & $\h_3$ & -- \\
	& $\alpha=0$ & $\h_3$ & $\bR^2$ & -- \\
	$A_{4,11}^{\alpha}$ & $(2 \alpha e^{14}+e^{23},\alpha e^{24}+e^{34}, \alpha e^{34}-e^{24},0)$ & $\h_3$ & $\h_3$ & -- \\
	&  $\alpha>0$ & & \\
	$\aff_{\bC}$ & $\left( e^{14}+e^{23},e^{24}-e^{13},0,0\right)$ & $e(2)$ & $\bR^2$  & -- \\
	$\aff_{\bR}\op \aff_{\bR}$ & $\left( e^{12},0,e^{34},0\right)$ & $e(1,1)$ & $\bR^2$ & -- \\
\end{longtable}

\end{document}